\documentclass[12pt,a4paper]{article}

\usepackage{hyperref}

\usepackage[english]{babel}
\usepackage{fancyhdr}
\usepackage{fullpage}
\usepackage[dvips]{graphicx}
\usepackage[T1]{fontenc}
\usepackage{fouriernc}
\usepackage[latin1]{inputenc}
\usepackage{amsmath}
\usepackage{amsfonts}
\usepackage{amssymb}
\usepackage{mathrsfs}
\usepackage{amsthm}
\usepackage{latexsym}
\usepackage{array}
\usepackage{mathrsfs}
\usepackage{amsxtra}
\usepackage{amscd}
\usepackage{mathtools}
\usepackage{verbatim}
\usepackage{enumerate}
\usepackage[title]{appendix}
\usepackage[usenames]{color}
\definecolor{beige}{rgb}{0.96, 0.96, 0.86}
\definecolor{airforceblue}{rgb}{0.36, 0.54, 0.66}
\definecolor{antiquefuchsia}{rgb}{0.57, 0.36, 0.51}
\definecolor{awesome}{rgb}{1.0, 0.13, 0.32}

\usepackage{tikz} 
\usetikzlibrary{arrows,shapes,matrix}

\usepackage{xcolor}
\hypersetup{
    colorlinks,
    linkcolor={green!60!blue},
    citecolor={yellow!30!red},
    urlcolor={green!60!blue}
}

\newcommand{\CNA}[2]{CNA([0,T]\times C([0,T],#1),{#2})}
\newcommand{\CNAW}[1]{CNA([0,T]\times \mathbb{W},{#1})}
\newcommand*{\PT}{{\mathcal{P}_T}}

\newcommand*{\rad}{M([0,T])}

\newcommand*{\Bb}[1]{B_b([0,T],#1)}
\newcommand*{\Bbz}[1]{B_{b,0}([0,T],#1)}
\newcommand*{\Cb}[1]{C([0,T],#1)}

\newcommand*{\GatB}[2]{\mathcal{G}^2_{\sigma^s}(\Cb{#1},#2)}
\newcommand*{\GatWB}[1]{\mathcal{G}^2_{\sigma^s}(\mathbb{W},#1)}

\newcommand*{\Gatot}[3]{\mathcal{G}^{#3}(#1,#2)}

\newcommand*{\medcap}{\mathbin{\scalebox{1.5}{\ensuremath{\cap}}}}

\numberwithin{equation}{section}

\theoremstyle{plain}
\newtheorem{theorem}{Theorem}[section]
\newtheorem{corollary}[theorem]{Corollary}
\newtheorem{proposition}[theorem]{Proposition}
\newtheorem{lemma}[theorem]{Lemma}
\newtheorem{definition}[theorem]{Definition}

\newtheorem{assumption}[theorem]{Assumption}

\theoremstyle{definition}
\newtheorem{remark}[theorem]{Remark}
\newtheorem{example}[theorem]{Example}

\renewcommand{\epsilon}{\varepsilon}
\renewcommand{\phi}{\varphi}

\def\qed{{\hfill\hbox{\enspace${ \blacksquare}$}} \smallskip}

\newcommand{\e}{e}

  \def\Swiech
{{\accent1 S}wie{\hbox{\kern -0.31em\lower
0.77ex \hbox{$\textfont1=\scriptfont1 \lhook$}}}ch}

\def\SWIECH
{{\accent1 S}WIE{\hbox{\kern -0.21em\lower
0.77ex\hbox{$\textfont1=\scriptfont1
\lhook$}}}CH}

\usepackage{dsfont} 

\def\b*{\begin{eqnarray*}}
\def\e*{\end{eqnarray*}}

\def \0{\mathbf{0}}

\theoremstyle{plain}

\theoremstyle{definition}

\newcommand*{\hatc}{{\hat t,\hat Y}}

\newcommand*{\Bone}[1]{\mathbb{B}^1(#1)}
\newcommand*{\Boninf}[1]{\mathbb{B}^1_\infty(#1)}
\newcommand*{\Bonl}[1]{\mathbb{B}^1_\lambda(#1)}
\newcommand*{\Bones}[1]{\mathbb{B}^1_{\sigma{^s}}(#1)}
\newcommand*{\Sone}[1]{\mathbb{V}(#1)}
\newcommand*{\Soninf}[1]{\mathbb{V}_\infty(#1)}
\newcommand*{\Sones}[1]{\mathbb{V}_{\sigma{^s}}(#1)}

\newcommand{\NA}[2]{NA([0,T]\times C([0,T],#1),{#2})}

\newcommand*{\Gatdue}[2]{\mathcal{G}^2(\Cb{#1},#2)}

\theoremstyle{plain}

\theoremstyle{definition}

 \theoremstyle{definition}

\def\<{\left\langle }
\def\>{\right\rangle }

\title{Functional It\=o calculus in
Hilbert spaces and
  application to  path-dependent 
 Kolmogorov equations}
\date{}
\author{Mauro Rosestolato\thanks{CMAP, \'Ecole Polytechnique, Paris, France,
e-mail: \texttt{mauro.rosestolato@polytechnique.edu}. 
This research has been
partially supported by
the ERC
321111 Rofirm.
The author is sincerely grateful to Salvatore Federico for valuable discussions.
}
}

\begin{document}

\maketitle

\begin{abstract} 
 Recently, in \cite{Cont2010,Cont2013,Dupire2009}, functional It\=o calculus has been introduced and developed in finite dimension for functionals of continuous semimartingales.
With different techniques, we develop a functional It\=o calculus for functionals of Hilbert space-valued diffusions.
In this context, we first prove a path-dependent It\=o's formula, then we show applications to classical solutions of path-dependent Kolmogorov equations in Hilbert spaces and derive a Clark-Ocone type formula.
Finally, we explicitly verify that all the theory developed can be applied to a class of diffusions driven by SDEs with a 
 path-dependent drift 
(suitably regular) and constant diffusion coefficient.
\end{abstract}

\vspace{10pt}

\vspace{10pt} 

\section{Introduction}

The present paper extends to infinite dimensional spaces
the so called functional It\=o calculus, so far developed
in finite-dimensional spaces, and some of its applications.

In \cite{Dupire2009} the first ideas for a functional It\=o calculus were presented for one-dimensional continuous semimartingales, by introducing suitable notions of time/space derivatives which reveal to be adequate for dealing with non-anticipative functionals.
In that paper, a functional It\=o's formula is provided and then employed to represent solutions of a backward Kolmogorov equation
with path-dependent terminal value.
This allows to obtain an explicit representation of the stochastic integrand in the martingale representation theorem, when the martingale is closed by a functional of the process solving the SDE associated to the Kolmogorov equation.
In \cite{Cont2010a,Cont2010,Cont2013} these ideas are furtherly developed  and generalized.
In \cite{Cont2010a} the functional It\=o's formula is proved for a large class of finite-dimensional c\`adl\`ag processes, including semimartingales and Dirichlet processes, and for functionals which can depend on the quadratic variation.
In \cite{Cont2013} 
the notion of vertical derivative is extended to 
square integrable continuous martingales and it is proved that it coincides with the stochastic integrand in the martingale representation theorem.

Functional It\=o calculus in finite dimension can be also viewed as an application to the spaces of continuous/c\`adl\`ag functions  of  stochastic calculus in Banach spaces
(\cite{DiGirolami2014a,DiGirolami2011,DiGirolami2012,DiGirolami2014,Flandoli2013}).
In \cite{DiGirolami2014} the notion of $\chi$-quadratic variation is introduced for Banach space-valued processes (not necessarily semimartingales) and the related It\=o's formula is discussed.
This general framework finds application to ``window'' processes in $C([-T,0],\mathbb{R}^n)$, whose values, at each time $t\in[0,T]$, is essentially the path up to time $t$ of an $\mathbb{R}^n$-valued continuous process.
When applied to window processes, such
 It\=o's formula allows to derive a Clark-Ocone type representation formula by recurring to solutions of a path-dependent Kolmogorov equation.
In \cite{Flandoli2013} finite dimensional It\=o processes $X$ with constant diffusion coefficient and path-dependent drift are considered.
By embedding the dynamics of $X$ into a Banach space of functions $[-T,0]\rightarrow \mathbb{R}^n$, it is proved that the Feynman-Kac formula provides a solution to the path-dependent backward Kolmogorov equation associated to $X$, with a non-path-dependent terminal value.

Another approach to path-dependent functionals and path dependent stochastic systems  is represented by the embedding in infinite dimensional Hilbert spaces. Indeed, when the dependence on the history is sufficiently regular --- precisely regular with respect to a $L^2$ norm --- a representation in the Hilbert space of the form $\mathbb{R}\times L^2$ is possible. This approach goes back to
\cite{Chojnowska-Michalik1978}
 and was further developed in other papers (\cite{Federico2011,Federico2015,Gozzi2004}). With this approach, the very well-developed theory  of stochastic calculus in Hilbert space (\cite{DaPrato2014}) can be applied. On the other hand, it leaves out some important classes of problems, in particular all those where the dependence on the history involves  pointwise evaluations at past times.

Up to our knowledge, so far the functional It\=o calculus has been developed only in finite dimensional spaces.
We generalize it to infinite dimension as follows.
Consider two real separable Hilbert spaces $U,H$ and
a $U$-valued cylindrical Wiener process $W$.
Given $T>0$, denote by $\mathbb{W}$ the space $\Cb{H}$ of continuous functions $[0,T]\rightarrow H$.
Given $t\in [0,T]$ and $\mathbf{x}\in \mathbb{W}$, consider the process
$$
X^{t,\mathbf{x}}_s=\mathbf{x}_{t\wedge \cdot}+\int_t^{t\vee s}b_rdr
+\int_t^{t\vee s}\Phi_rdW_r\qquad s\in [0,T],
$$
where 
$$
\mathbf{x}_{t\wedge \cdot}(s)\coloneqq
\begin{dcases}
  \mathbf{x}(s)&s\in[0,t]\\
  \mathbf{x}(t)&s\in(t,T],
\end{dcases}
$$
$b$ is a square-integrable $H$-valued process, and $\Phi$ is a square-integrable process  valued in the space of Hilbert-Schmidt operators $L_2(U,H)$.
We develop a functional It\=o calculus for processes of the form
$$
u(\cdot,X^{t,\mathbf{x}})\coloneqq   \left\{ u(s,X^{t,\mathbf{x}}_s) \right\} _{s\in[0,T]}
$$
where $u\colon [0,T]\times \mathbb{W}\rightarrow \mathbb{R}$ is a \emph{non-anticipative} functional, meaning that $u(s,\mathbf{y})=u(s,\mathbf{y}')$ whenever $\mathbf{y}=\mathbf{y}'$ on $[0,s]$ for a given $s\in[0,T]$.
Under suitable regularity assumptions on $u$, we prove an It\=o formula for $u(\cdot,X^{t,\mathbf{x}})$.
Then, assuming that $X^{t,\mathbf{x}}$ is driven by an SDE
of the form
\begin{equation}\label{2017-05-02:00}
  \begin{dcases}
    dX_s=b(s,X)ds + \Phi(s,X)dW_s &\quad \forall s\in [t,T]\\
    X_{ t\wedge \cdot}= \mathbf{x}_{  t\wedge \cdot},&
  \end{dcases}
\end{equation}
where $b\colon[0,T]\times \mathbb{W}\rightarrow H$,
$\Phi\colon [0,T]\times \mathbb{W}\rightarrow L_2(U,H)$ are 
non-anticipative
coefficients
satisfying usual Lipschitz conditions, and letting
 $f\colon \mathbb{W}\rightarrow \mathbb{R}$ be 
a function, we show that, if the non-anticipative function $\varphi$ defined by
$$
\varphi(t,\mathbf{x})\coloneqq \mathbb{E} \left[ f(X^{t,\mathbf{x}}) \right] \qquad (t,\mathbf{x})\in [0,T]\times \mathbb{W}
$$
is suitably regular, then $\varphi$
solves the path-dependent backward Kolmogorov equation associated to \eqref{2017-05-02:00} with terminal value $f$ at time $T$.
As a corollary, we obtain a Clark-Ocone type formula for the process $\varphi(\cdot,X^{t,\mathbf{x}})$.
Finally,
 we 
accomplish a complete study of the regularity of the solution $X^{t,\mathbf{x}}$ to SDE~\eqref{2017-05-02:00} with respect to $t,\mathbf{x}$, when $\Phi$ is constant and $b$ contains a convolution of the path of $X$ with a Radon measure.
In particular,
 the case of pointwise delay in the coefficient $b$ 
 will be covered.
For the latter class of dynamics, by a pathwise analysis, we show in detail
that the assumptions required 
by the general results previously obtained (It\=o's formula, representation of solution to the  path-dependent Kolmogorov equation, Clark-Ocone type formula) are satisfied, hence the theory can be applied.

Our methods deviate from the ones used in the aforementioned liteature.
In \cite{Cont2010a,Cont2010,Cont2013,Dupire2009} non-anticipative functionals are considered
on the metric space $\Lambda$ of couples ``(time $t$,c\`adl\`ag path on $[0,t]$)''.
Due to the lack of a linear structure for $\Lambda$,
this choice leads to introduce non-standard notions of derivatives (vertical/horizontal) and to deal with ad-hoc continuity assumptions.
On the contrary, we do not use the space $\Lambda$ and, in a more standard perspective, 
we look at the set of continuous non-anticipative functionals as a subvector space of the space of continuous functions on $[0,T]\times \mathbb{W}$.
Our choice is equivalent to take the restriction of $\Lambda$ to couples with continuous path in the second component as working space, 
but shows the advantage
to allow to deal with classical G\^ateaux derivatives in space.
The choice of G\^ateaux derivatives in space reveals to be particularly adequate when proving regularity of solutions to path-dependent SDEs
with respect to the intial value by using contraction methods in
 Banach spaces, as in 
Section~\ref{2016-02-12:11}:
if one wishes to apply
the theoretical results in practice, this is a key step in order to show that the assumptions of the theory are satisfied.
Nevertheless, also in our setting, the introduction of an ad-hoc time derivative for non-anticipative functionals cannot be avoided.
It is remarkable that it is convenient for us to use a left-sided time derivative, instead of the right-sided derivative introduced in \cite{Dupire2009} and then adopted also in \cite{Cont2010a,Cont2010,Cont2013}.
Our choice turns out to be very natural when studying the link between the path-dependent SDE and the associated Kolmogorov equation.
Moreover, unlike \cite{DiGirolami2014a,DiGirolami2011,DiGirolami2012,DiGirolami2014,Flandoli2013},
we do not rephrase our path-dependent problem in a  Banach space.
This allows to avoid to work with stochastic calculus in Banach spaces. 

\bigskip
The paper is organized as follows.
 In Section~\ref{2017-05-11:06}, after introducing some notation, we define the locally convex space with respect to which the regularity of  non-anticipative functionals will be considered.
In Section~\ref{2016-02-03:02} we prove the path-dependent It\=o's formula (Theorem~\ref{2017-05-30:02}).
In Section~\ref{2017-05-11:07}
we show that the Feynman-Kac formula for the strong solution of a path-dependent SDE in Hilbert spaces,
if it is enough regular,
 provides a solution to the associated Kolmogorov equation (Theorem~\ref{2016-02-10:10}).
We then use this fact to derive a Clark-Ocone type formula (Corollary~\ref{2017-06-08:02}).
Finally, in Section~\ref{2016-02-12:11}, we explicitly show that the previously developed theory can be applied to 
a class of SDEs with path-dependent drift and constant diffusion coefficient (Theorem~\ref{2016-02-13:02}).

\section{Preliminaries}
\label{2017-05-11:06}

\subsection{Notation}
\label{not:2017-04-25:22}

Let $T>0$,
let
  $(\Omega,\mathbb{F}\coloneqq \{\mathcal{F}_t\}_{t\in[0,T]},\mathcal{F},\mathbb{P})$ be a complete filtered probability space, 
and let $(E,|\cdot|_E)$ be a Banach space.
Unless otherwise specified, every Banach space $E$ is considered endowed with its Borel $\sigma$-algebra $\mathcal{B}_E$.
$\Bb{E}$ denotes the space of bounded Borel measurable functions $\mathbf{x}\colon[0,T]\rightarrow E$.
If $\mathbf{x}\in \Bb{E}$, 
 then $\mathbf{x}_t$ and $\mathbf{x}(t)$ denote the evaluation at time $t\in[0,T]$ of the function $\mathbf{x}$, whereas $\mathbf{x}_{t\wedge \cdot}$ denotes the function defined by  $(\mathbf{x}_{t\wedge \cdot})_s\coloneqq \mathbf{x}_{t\wedge s}$ for $s\in [0,T]$.
  We denote by $\Bbz{E}$ the subspace of $\Bb{E}$ of bounded Borel functions $\mathbf{x}\colon [0,T]\rightarrow E$ with separable range.
Unless otherwise specified, $\Bbz{E}$ is considered with the topology of the uniform convergence.
Then $\Bbz{E}$ is a Banach space
and  $\Cb{E}\subset \Bbz{E}$.
 $\mathcal{L}^0_{\mathcal{P}_T}(C([0,T],E))$ denotes the space of 
$E$-valued $\mathbb{F}$-adapted continuous processes.
Notice that this implies the measurability of the continuous process
$$
\Omega\rightarrow C([0,T],E),\ \omega \mapsto X(\omega)
$$
for all $\omega\in \Omega$ (\footnote{$E$ is not assumed to be separable.}), hence the measurability of
$$
  (\Omega_T,\mathcal{P}_T)\rightarrow C([0,T],E),\ (\omega,t)  \mapsto  X_{t\wedge \cdot}(\omega).
$$
If $X,X'\in \mathcal{L}^0_{\mathcal{P}_T}(C([0,T],E))$, we write $X=X'$ if and only if \mbox{$\mathbb{P}(|X-X'|_\infty=0)=1$}.
For $p\in [1,\infty)$,
 $\mathcal{L}^p_{\mathcal{P}_T}(C([0,T],E))$ denotes the space of  
functions
$X\in \mathcal{L}^0_{\mathcal{P}_T}(C([0,T],E))$
such that
$\Omega\rightarrow C([0,T],E),\ \omega  \mapsto  X
(\omega)$ has separable range and
\begin{equation*}
  |X|_{\mathcal{L}^p_{\mathcal{P}_T}(C([0,T],E))}\coloneqq  \left( \mathbb{E} \left[ |X|_\infty^p \right]  \right) ^{1/p}<\infty.
\end{equation*}

By $\rad$ we denote 
the space of Radon measures on the interval $[0,T]$.
For $\nu\in \rad$, $|\nu|_1$ denotes the total variation of $\nu$.

 Let $F$ be another Banach space.
$\mathcal{G}^n(E,F)$ denotes the space of continuous functions $f\colon E\rightarrow F$
which are G\^ateaux differentiable up to order $n$ and such that, for $j=1,\ldots,n$,
$$
 E^{i+1}\rightarrow F,\ (x,y_1,\ldots,y_i) \mapsto  \partial ^i_{y_1\ldots y_i}f(x)
$$
is continuous.
If $f\in[0,T]\times E\rightarrow F$ is such that $f(t,\cdot)\in \Gatot{E}{F}{n}$ for all $t\in[0,T]$, then we denote by $ \partial^j _Ef$, $j=1,\ldots,n$, the G\^ateaux differentials of $f$ with respect $E$.
Similarly, if $f(t,\cdot)\in C^n(E,F)$, i.e.\ $f(t,\cdot)$ is
continuously
 Fr\'echet differentiable 
up to order $n$,
we denote by $D^j_Ef$, $j=1,\ldots,n$, the Fr\'echet differentials of $f$ with respect to $E$.

  $\NA{E}{F}$\label{not:2017-04-25:12}  
denotes
the subspace of $F^{[0,T]\times \Cb{E}}$ whose members are non-anticipative functions, that is
   \begin{multline*}
   \NA{E}{F}\coloneqq
\left\{f\in F^{[0,T]\times \Cb{E}}        \colon\right.\\
\left.\phantom{F^{[0,T]}} f(t,\mathbf{x})=f(t,\mathbf{x}_{t\wedge \cdot})\        \forall (t,\mathbf{x})\in [0,T]\times \Cb{E} \right\}.
   \end{multline*}
By
 $\CNA{E}{F}$\label{not:2017-04-25:14} we denote the subspace of $C([0,T]\times \Cb{E},F)$
whose members are non-anticipative functions, 
that is
$$
\hskip-20pt\CNA{E}{F}\coloneqq 
C([0,T]\times \Cb{E},F)\medcap \NA{E}{F}.
$$

 $(H,|\cdot|_H)$ and $(U,|\cdot|_U)$ denote two real separable Hilbert spaces, with scalar product denoted by $\langle \cdot,\cdot\rangle_H$ and $\langle \cdot,\cdot\rangle_U$,
respectively.
Let  $ \mathfrak e\coloneqq \{e_n\}_{n\in \mathcal{N}}$
be an orthonormal basis of $H$,
 where $\mathcal{N}=\{1,\ldots,N\}$
if $H$ has dimension $N\in \mathbb{N}\setminus\{0\}$, or $\mathcal{N}=\mathbb{N}$ if $H$ has infinite dimension.
Similarly, $\mathfrak e'\coloneqq \{e'_m\}_{m\in \mathcal{M}}$
denotes an orthonormal basis of $U$,
 where $\mathcal{M}=\{1,\ldots,M\}$ if $U$ has dimension $M\in \mathbb{N}\setminus \{0\}$, or $\mathcal{M}=\mathbb{N}$ if $U$ has infinite dimension.
We use the short notation $\mathbb{W}$  for the space $\Cb{H}$ of continuous functions $[0,T]\rightarrow H$.

\subsection{The space $\Bones{E}$}\label{2016-01-25:02}

In this section we introduce a topology
 with respect to which we will
often
consider the  regularity of the differentials of
 path-dependent functions in the
remaining of the manuscript.


\smallskip
Let $E$ denote a Banach space.
We begin by introducing on $\Bbz{E}$ the family of seminorms $\mathbf{p}^s\coloneqq \{p^s_\nu\}_{\nu\in \rad}$
defined by
\begin{equation*}
    p^s_\nu(\mathbf{x})\coloneqq \left|\int_{[0,T]} \mathbf{x}(s)\nu(ds)\right|_E\qquad \forall \mathbf{x}\in \Bbz{E},\ \forall \nu\in \rad.
  \end{equation*}
Since we are
considering only bounded Borel functions $\mathbf{x}$ with separable range, the integral $\int_{[0,T]} \mathbf{x}d\mu$ is well defined.

We denote by $\sigma^s$  the locally convex vector topology induced on $\Bbz{E}$  by $\mathbf{p}^s$. 
If
 $\tau_\infty$ denotes the topology of the uniform convergence $\Bbz{E}$, it is easily seen that
\begin{equation}
  \label{eq:2016-01-25:03}
 \sigma^s\subsetneq \tau_\infty.
\end{equation}
The inclusion $\sigma^s\subset \tau_\infty$ is immediate,
whereas
the strict inclusion
is due to the fact that $\sigma^s$ is contained in the weak topology
 of $(\Bbz{E},|\cdot|_\infty)$,
 and, unless $E$ is trivial, the weak topology
 is strictly weaker than the topology induced by the norm, because  $\Bbz{E}$ is infinite dimensional.
The same holds for the restrictions to $\Cb{E}$, i.e.\
$
\sigma^s_{|\Cb{E}}\subsetneq \tau_{\infty|\Cb{E}}
$.

\begin{proposition}
  \label{2015-08-24:00}
Convergent and Cauchy sequences in $\sigma^s$
 are characterized as follows.
\begin{enumerate}[(i)]
\item \label{2016-01-25:08}   A sequence $\{\mathbf{x}_n\}_{n\in \mathbb{N}}$ converges to $\mathbf{x}$ in  $(\Bbz{E},\sigma^s)$ if and only if
  \begin{equation}
    \label{eq:2016-01-25:05}
    \begin{dcases}
      (a)\ \ \sup_{n\in \mathbb{N}}|\mathbf{x}_n|_\infty<\infty\\
      (b)\ \ \lim_{n\rightarrow \infty}\mathbf{x}_n(s)=\mathbf{x}(s)\quad\forall s\in[0,T].
    \end{dcases}
  \end{equation}
\item 
\label{2016-01-25:07}
A sequence $\{\mathbf{x}_n\}_{n\in\mathbb{N}}$ is Cauchy in $(\Bbz{E},\sigma^s)$ if and only if
\eqref{eq:2016-01-25:05}$(a)$
holds and the sequence $\{ \mathbf{x}_n(s)\}_{n\in \mathbb{N}}$ is Cauchy for every $s\in[0,T]$.
\end{enumerate}
\end{proposition}
\begin{proof}
We prove  only \emph{(\ref{2016-01-25:08})}. The proof of 
\emph{(\ref{2016-01-25:07})} is similar.
  Suppose that $\{\mathbf{x}_n\}_{n\in \mathbb{N}}$ converges to $\mathbf{x}$ in $(\Bbz{E},\sigma^s)$. For $s\in[0,T]$,
if
$\delta_s$
is the Dirac measure in $s$, we have
$$
\lim_{n\rightarrow \infty} |\mathbf{x}_n(s)-\mathbf{x}(s)|_H=
\lim_{n\rightarrow \infty}p_{\delta_s}(\mathbf{x}_n-\mathbf{x})=0\qquad,
$$
which shows
\eqref{eq:2016-01-25:05}\emph{(b)}.

To show
\eqref{eq:2016-01-25:05}\emph{(a)}, consider the family of continuous linear operators
$$
\Phi_{n}\colon \rad\rightarrow E,\ \nu\mapsto \int_{[0,T]}  \mathbf{x}_n(s)\nu(ds),
$$
for $n\in \mathbb{N}$. Since $\{\mathbf{x}_n\}_{n\in \mathbb{N}}$ is convergent,
the orbit $\{\Phi_{n}(\nu)\}_{n\in\mathbb{N}}$ is bounded, for
all $\nu\in \rad$, 
 then, by Banach-Steinhaus theorem, we have
$$
\sup_{n\in\mathbb{N}}
\left|\mathbf{x}_n\right|_\infty=\sup_{n\in\mathbb{N}}\sup_{\substack{\nu\in \rad\\|\nu|_1\leq 1}}\left|\int_{[0,T]} \mathbf{x}_n(s)\nu(ds)\right|_E=\sup_{n\in \mathbb{N}}|\Phi_{n}|_{L(\rad,E)}<\infty,
$$
where $|\Phi_n|_{L(\rad,E)}$ denotes the operator norm of $\Phi_n$.
This
shows
\eqref{eq:2016-01-25:05}\emph{(a)} and concludes the proof for one direction of the claim.

Conversely, if \eqref{eq:2016-01-25:05} holds, then $p_{\nu}(\mathbf{x}_n-\mathbf{x})\rightarrow 0$ by Lebesgue's  dominated convergence theorem, for all $\nu\in \rad$,
 hence $\mathbf{x}_n\rightarrow \mathbf{x}$ in $\sigma^s$.
\end{proof}

By \eqref{eq:2016-01-25:03}, 
it follows that
 bounded sets in $\tau_\infty$
are bounded in $\sigma^s$.
By using Banach-Steinhaus theorem similarly as done in the proof of 
Proposition \ref{2015-08-24:00}, one can see that bounded sets in $\sigma^s$ are bounded in $\tau_\infty$.
Then the bounded sets in $\sigma^s$ and $\tau_\infty$ are the same.

\begin{definition}\label{2016-01-27:04}
  We define $\Bone{E}$ as the vector space of all functions $\mathbf{x}\colon [0,T]\rightarrow E$ which are  pointwise limit of a uniformly bounded sequence $\{\mathbf{x}_n\}_{n\in \mathbb{N}}\subset \Cb{E}$, i.e.\ 
$$
\Bone{E}\coloneqq  \left\{ \mathbf{x}\in E^{[0,T]}\colon
 \exists\ \{\mathbf{x}_n\}_{n\in \mathbb{N}}\subset \Cb{E}\ \text{s.t.}\ 
 \begin{dcases}
\lim_{n\rightarrow\infty}\mathbf{x}_n(s)=\mathbf{x}(s)&\forall s\in[0,T]\\
\sup_{n\in \mathbb{N}}|\mathbf{x}_n|_\infty<\infty&
\end{dcases}
\qquad\right\} .
$$
We denote  by $\Bones{E}$ the  space $\Bone{E}$ endowed with the 
locally convex
topology induced by $\sigma^s$.
Then a net $\{\mathbf{x}_\iota\}_{\iota\in \mathcal{I}}$ converges to $0$ in $\Bones{E}$ if and only if $\lim_\iota p_{\nu}(\mathbf{x}_\iota)= 0$ for all $\nu\in \rad$.
\end{definition}

\medskip
 \begin{remark}\label{2017-06-10:01}
   By
Proposition \ref{2015-08-24:00}\emph{(\ref{2016-01-25:08})}, it follows that $\Bone{E}$ is the sequential closure
of $\Cb{E}$ in $(\Bbz{E},\sigma^s)$.
In particular, for any
\mbox{$T_2$-space $\mathcal{T}$} and any
 function $\Cb{E}\rightarrow \mathcal{T}$, there exists at most one sequentially continuous extension 
$(\Bone{E},\sigma^s)\rightarrow \mathcal{T}$.
 \end{remark}

\begin{remark}\label{2016-01-27:05}
  In Definition \ref{2016-01-27:04}, 
by
multiplying $\mathbf{x}_n$ by $|\mathbf{x}|_\infty/|\mathbf{x}_n|_\infty$
if necessary,
we can assume without loss of generality that $\sup_{n\in \mathbb{N}}|\mathbf{x}_n|_\infty\leq |\mathbf{x}|_\infty$.
By 
Proposition~\ref{2015-08-24:00}\emph{(\ref{2016-01-25:08})},
we then see that the unit ball of $(\Cb{E},|\cdot|_\infty)$ is $\sigma^s$-sequentially dense in the unit ball of 
$(\Bone{E},|\cdot|_\infty)$.
\end{remark}

Since we have the inclusion $\Bone{\mathbb{R}}\subsetneq  \Bb{\mathbb{R}}$
(see \cite[Theorem 11.4]{Rooij1982}), through the identification
$B_b([0,T],\mathbb{R})=B_b([0,T],\mathbb{R}e)$
 in $\Bbz{E}$, for some $e\in E$, $|e|_E=1$ ($E\neq \{0\}$), we also have the strict inclusion $\Bone{E}\subsetneq \Bbz{E}$.

\smallskip

The space
  $\Bone{E}$ is closed in  $\Bbz{E}$ (hence in $\Bb{E}$)
  with respect to the uniform norm.
The proof of \cite[Theorem 11.7]{Rooij1982}, that is made for the case $E=\mathbb{R}$
and for a space of Borel functions  larger than our $\Bone{\mathbb{R}}$,
can be adapted to cover our case.
Since the completeness
of $\Bone{E}$
 is essential to us, we prove it.

\begin{proposition}
  $(\Bone{E},|\cdot|_\infty)$ is a Banach space.
\end{proposition}
\begin{proof}
We show that every absolutely convergent sum is convergent in $\Bone{E}$.
To this end, let $\{\mathbf{x}_n\}_{n\in \mathbb{N}}\subset \Bone{E}$  be sequence such that $\sum_{n\in \mathbb{N}}|\mathbf{x}_n|_\infty<\infty$.
By completeness of $\Bb{E}$, $\sum_{n\in \mathbb{N}}\mathbf{x}_n$ is convergent in $\Bb{E}$, say to $\mathbf{z}$.
We are done if we show that $\mathbf{z}\in \Bone{E}$.
  By definition of $\Bone{E}$, for each $n\in \mathbb{N}$, there exists a sequence $\{\mathbf{y}_n^{(k)}\}_{k\in \mathbb{N}}\subset \Cb{E}$ 
such that
$$
M_n\coloneqq\sup_{k\in \mathbb{N}}|\mathbf{y}^{(k)}_n|_\infty<\infty\qquad\mbox{and}\qquad \lim_{k\rightarrow \infty}\mathbf{y}^{(k)}_n(s)=\mathbf{x}_n(s) \ \forall s\in[0,T].
$$
By multiplying $\mathbf{y}^{(k)}_n$ by $|\mathbf{x}_n|_\infty/|\mathbf{y}^{(k)}_n|_\infty$ if necessary,
without loss of generality we can assume that $M_n\leq |\mathbf{x}_n|_\infty$.
Define
$
\mathbf{z}_k\coloneqq \sum_{n=1}^k \mathbf{y}_{n}^{(k)}$,  $k\in \mathbb{N}$.
Then $\mathbf{z}_k\in C([0,T],E)$
and
\begin{equation}
  \label{2017-05-23:06}
  \sup_{k\in \mathbb{N}}|\mathbf{z}_k|_\infty\leq
\sup_{k\in \mathbb{N}}\sum_{n=1}^\infty
|\mathbf{y}_n^{(k)}|_\infty
\leq
\sum_{n=1}^\infty
|\mathbf{x}_n|_\infty<\infty.
\end{equation}
Moreover, for $s\in[0,T]$, $0\leq \bar k\leq k$,
\begin{equation*}
  \begin{split}
    |\mathbf{z}(s)-\mathbf{z}_k(s)|_E&
    =
    |\sum_{n=1}^\infty\mathbf{x}_n(s)-\sum_{n=1}^k\mathbf{y}_n^{(k)}(s)|_E
    \leq
    \sum_{n=\bar k}^\infty
(|\mathbf{x}_n|_\infty+|\mathbf{y}_n^{(k)}|_\infty)+
    \sum_{n=1}^{\bar k}|\mathbf{x}_n(s)-\mathbf{y}_n^{(k)}(s)|_E\\
    &\leq
   2 \sum_{n=\bar k}^\infty |\mathbf{x}_n|_\infty+
    \sum_{n=1}^{\bar k}|\mathbf{x}_n(s)-\mathbf{y}_n^{(k)}(s)|_E.
  \end{split}
\end{equation*}
By taking first the ${\displaystyle\limsup_{k\rightarrow \infty}}$, 
recalling the pointwise convergence
 $\mathbf{y}^{(k)}_n(s)\rightarrow \mathbf{x}_n(s)$ as $k\rightarrow \infty$,
 and then taking the ${\displaystyle\lim_{\bar k\rightarrow \infty}}$, we obtain
$
\mathbf{z}_k(s)\rightarrow \mathbf{z}(s)
$ as $k\rightarrow \infty$.
Since $s\in[0,T]$ was arbitrary, 
this, together with \eqref{2017-05-23:06}, proves that $\mathbf{z}\in \Bone{E}$.
\end{proof}



\subsection{$\Sones{E}$-sequentially continuous derivatives}

  We introduce the following subspace of $\Bone{E}$:
\begin{equation}
  \label{eq:2016-02-10:00}
  \Sone{E}\coloneqq  \operatorname{Span} \left\{ \mathbf{x}+v\mathbf{1}_{[t,T]}\colon \mathbf{x}\in \Cb{E},\ v\in E,\ t\in[0,T]
 \right\}.
\end{equation}
A member of $\Sone{E}$ is the sum of a continuous function and  a right-continuous step function (with finite number of jumps).
We denote by $\Sones{E}$ the space $\Sone{E}$ endowed with the locally convex topology induced by $\Bones{E}$
and by $\Soninf{E}$ the space $\Sone{E}$ endowed with the topology induced by the supremum norm $|\cdot|_\infty$.

\begin{definition}\label{2016-04-22:07}
 We say that a function $f\in \Gatdue{E}{F}$ 
 has 
\emph{derivatives with $\Sones{E}$-sequentially continuous extensions} if
$$
 \partial f\colon \Cb{E}\times \Cb{E}\rightarrow F,\ (\mathbf{x},\mathbf{v}) \mapsto  \partial _\mathbf{v}f(\mathbf{x})
$$
and
$$
 \partial^2 f\colon \Cb{E}\times \Cb{E}\times \Cb{E}\rightarrow F,\ (\mathbf{x},\mathbf{v},\mathbf{w}) \mapsto  \partial^2_{\mathbf{v}\mathbf{w}}f(\mathbf{x})
$$
admit sequentially continuous extensions, respectively,
$$
 \overline{\partial f}\colon \Cb{E}\times \Sones{E}\rightarrow F,\ (\mathbf{x},\mathbf{v}) \mapsto  \overline{\partial f}(\mathbf{x}).\mathbf{v}
$$
and
$$
\overline{ \partial^2 f}\colon \Cb{E}\times \Sones{E}\times \Sones{E}\rightarrow F,\ (\mathbf{x},\mathbf{v},\mathbf{w}) \mapsto  \overline {\partial^2 f}(\mathbf{x}).
(\mathbf{v},\mathbf{w}).
$$
We denote by $\GatB{E}{F}$
\label{2016-02-11:21} the subspace of $\Gatdue{E}{F}$ containing the functions having derivatives with $\Sones{E}$-sequentially continuous extensions.
\end{definition}

  If $u\in \NA{E}{F}$, 
$t\in [0,T]$, and $u(t,\cdot)\in \GatB{E}{F}$, then 
the notation $ \overline {\partial _Eu}(t,\mathbf{x}).\mathbf{v}$,
for $\mathbf{x}\in \Cb{E}$ and
 $\mathbf{v}\in \Sone{E}$,
stands 
for $\overline{ \partial _Eu(t,\cdot)}(\mathbf{x}),\mathbf{v}$.
Similarly, $\overline{ \partial _Eu}(t,\cdot)$
stands for
$\overline{ \partial _Eu (t,\cdot)}$.

\begin{remark}\label{2017-06-10:02}
  If $u\in \NA{E}{F}$ is such that, for some
$t\in [0,T]$,  $u(t,\cdot)\in \Gatot{E}{F}{2}$, then, by non-anticipativity,
$$
 \partial _Eu(t,\mathbf{x}).\mathbf{v}
=
 \partial _Eu(t,\mathbf{x}).\mathbf{v}'\qquad
\forall
\mathbf{x},
\mathbf{v},\mathbf{v}'
\in \Cb{E}\ 
 \mbox{s.t.}\
\mathbf{v}(s)=
\mathbf{v}'(s)\ \mbox{for}\ s\in[0,t].
$$
If
$u(t,\cdot)\in \GatB{E}{F}$, then it also holds
$$
\overline{ \partial _Eu}(t,\mathbf{x}).\mathbf{v}
=
\overline{ \partial _Eu}(t,\mathbf{x}).\mathbf{v}'\qquad
\forall
\mathbf{x}\in \Cb{E},\ 
\forall \mathbf{v},\mathbf{v}'\in \Sone{E}\ \mbox{s.t.}\
\mathbf{v}(s)=
\mathbf{v}'(s)\ \mbox{for}\ s\in[0,t].
$$
In particular,
$$
\overline{ \partial _Eu}(t,\mathbf{x}).(\mathbf{1}_{[t,T]}v)
=
\overline{ \partial _Eu}(t,\mathbf{x}).
(\mathbf{1}_{[t,T')}v)\qquad
\forall
\mathbf{x}\in \Cb{E},\ 
\forall v\in E,\ \forall T'\in(t,T).
$$
A similar remark holds for the second-order differential.
Because of that, the directional derivatives
$ \overline{\partial _E}u(t,\mathbf{x}).(\mathbf{1}_{[t,T]}v),
 \overline{\partial^2 _Eu}(t,\mathbf{x}).(\mathbf{1}_{[t,T]}v,\mathbf{1}_{[t,T]}w)$, $\mathbf{x}\in \Cb{E}$, $v,w\in E$, express in our framework the so-called vertical derivatives  of 
\cite{Cont2010a,Cont2010,Cont2013}.
\end{remark}

\begin{example}\label{2016-02-10:09}
Let $\mu\in \rad$ and $g\in C([0,T]\times E,F)$ such that $g(t,\cdot)\in \mathcal{G}^2(E,F)$ for all $t\in[0,T]$, and
let us assume that  $ \partial_E g$ and $ \partial ^2_{E}g$ are bounded on bounded sets of $[0,T]\times E$. Define
$$
f(\mathbf{x})\coloneqq \int_{[0,T]} g(s,\mathbf{x}(s))\mu(ds)\qquad \forall \mathbf{x}\in  \Cb{E}.
$$
Then $f\in \mathcal{G}^2(\Cb{E},F)$, with
\begin{equation*}
  \begin{aligned}
   \partial f(\mathbf{x}).\mathbf{v}&=\int_{[0,T]} \partial_E g(s,\mathbf{x}(s)).\mathbf{v}(s)\mu(ds)\qquad\qquad\ \ \ \ 
&& \forall\mathbf{x},\mathbf{v}\in \Cb{H}\\
 \partial^2 f(\mathbf{x}).(\mathbf{v},\mathbf{w})&=\int_{[0,T]} \partial^2_{E} g(s,\mathbf{x}(s)).(\mathbf{v}(s),\mathbf{w}(s))\mu(ds)
\qquad  &&\forall \mathbf{x},\mathbf{v},\mathbf{w}\in \Cb{H}.
\end{aligned}
\end{equation*}
It is clear that $ \partial f(\mathbf{x}).\mathbf{v}$ and
$ \partial ^2 f(\mathbf{x}).(\mathbf{v},\mathbf{w})$ can be computed with the same expressions when $\mathbf{v},\mathbf{w}\in \Sone{E}$.
Moreover, by 
Proposition \ref{2015-08-24:00}\emph{(\ref{2016-01-25:08})}, by Lebesgue's dominated convergence theorem, and by strong continuity of the G\^ateaux differentials of $g$,
 we have that $ \partial f(\mathbf{x}).\mathbf{v}$ and $ \partial f(\mathbf{x}).(\mathbf{v},\mathbf{w})$ are sequentially continuous with respect to $(\mathbf{x},\mathbf{v})\in \Cb{E}\times \Sones{E}$
and
$(\mathbf{x},\mathbf{v},\mathbf{w})\in \Cb{E}\times \Sones{E}\times \Sones{E}$, respectively.
Then $f\in \GatB{E}{F}$.
\end{example}

\section{A path-dependent It\=o's  formula}
\label{2016-02-03:02}

In this section we prove
an It\=o's  formula for processes of the form $\{u(t,X)\}_{t\in[0,T]}$, 
where $X$ is a diffusion with values in $H$ and $u$ is a non-anticipative function with regular time-space derivatives, in a sense  specified 
below by 
Assumption~\ref{2017-05-30:11}.

\smallskip
For a  non-anticipative function $u$, we introduce 
 the following left-sided time derivative.

\begin{definition}
  For $u\in \NA{E}{F}$ and  $(t,\mathbf{x})\in (0,T)\times \Cb{E}$, we define
the following left-sided derivative,
if it exists:
      \begin{equation}
      \mathcal{D}^-_tu(t,\mathbf{x})\coloneqq \lim_{h\rightarrow
  0^+}\frac{u(t,\mathbf{x}_{(t-h)\wedge \cdot})-u(t-h,\mathbf{x})}{h}.\label{2016-02-10:08}
\end{equation}
  \end{definition}

\bigskip
  \begin{remark}\label{2017-06-01:01}
    Notice that,
by the very definition,
 for $t,t'\in(0,T)$, $t<t'$, and $\mathbf{x}\in \Cb{E}$,  the derivative
$\mathcal{D}^-_tu(t',\mathbf{x}_{t\wedge \cdot})$
concides 
 with the
left-sided derivative of the map
$$
(t,T)\rightarrow F,\ s \mapsto u(s,\mathbf{x}_{t\wedge \cdot})
$$
computed in $t'$.
  \end{remark}


We will prove the path-dependent It\=o's formula under the following assumption.

  \begin{assumption}\label{2017-05-30:11}
The function $u$ belongs to $CNA([0,T]\times\mathbb{W},\mathbb{R})$ and has the following properties.
\begin{enumerate}[(i)]
\item\label{2017-05-31:10} For all  $t\in(0,T)$,
 $\mathcal{D}^-_tu(t,\mathbf{x})$ exists for all $\mathbf{x}\in \mathbb{W}$.
For a.e.\ $t\in(0,T)$, the map
 \begin{equation*}
   \mathbb{W}\rightarrow \mathbb{R},\ \mathbf{x}  \mapsto \mathcal{D}^-_tu(t,\mathbf{x})
 \end{equation*}
is continuous.
For all compact set $K\subset \mathbb{W}$ there exists $M_K>0$ such that
\begin{equation}\label{2017-05-31:00}
  \sup_{\mathbf{x}\in K}|\mathcal{D}^-_tu(t,\mathbf{x})|\leq M_K\qquad \mbox{for a.e.\ } t\in(0,T).
\end{equation}
\item\label{2017-05-31:06} For all $t\in[0,T]$, $u(t,\cdot)\in \GatWB{\mathbb{R}}$ and the differentials $ \partial _\mathbb{W}u$ and $ \partial ^2_\mathbb{W}u$ are bounded:
  \begin{equation}
\sup_{t\in[0,T]}  \sup_{\substack{\mathbf{x},\mathbf{v}\in  \mathbb{W}\\
|\mathbf{v}|_\infty\leq 1}}\left| \partial_\mathbb{W} u(t,\mathbf{x}).\mathbf{v}\right|< \infty
\end{equation}
\begin{equation}
\sup_{t\in[0,T]}  \sup_{\substack{\mathbf{x},\mathbf{v},\mathbf{w}\in
\mathbb{W}\\
|\mathbf{w}|\vee|\mathbf{v}|_\infty\leq 1}}\left| \partial _\mathbb{W}^2u(t,\mathbf{x}).(\mathbf{v},\mathbf{w})
\right|< \infty.
\end{equation}
\item\label{2017-05-31:07} For a.e.\ $t\in(0,T)$,
    \begin{align}
      \lim_{h\rightarrow 0^+}{\overline{\partial_\mathbb{W} u}}(t+h, \mathbf{x}_{t\wedge \cdot}).(\mathbf{1}_{[t,T]}(\cdot)v)&
=
{\overline{\partial_\mathbb{W} u}}(t, \mathbf{x}_{t\wedge \cdot}).(\mathbf{1}_{[t,T]}(\cdot)v),\\[5pt]
         \lim_{h\rightarrow 0^+}{\overline{\partial^2_\mathbb{W} u}}(t+h, \mathbf{x}_{t\wedge \cdot}).
(\mathbf{1}_{[t,T]}(\cdot)v,\mathbf{1}_{[t,T]}(\cdot)v)&
=
{\overline{\partial^2_\mathbb{W} u}}(t, \mathbf{x}_{t\wedge \cdot}).(\mathbf{1}_{[t,T]}(\cdot)v,\mathbf{1}_{[t,T]}(\cdot)v),
    \end{align}
for all $ \mathbf{x}\in \mathbb{W}$ and all $v\in H$.
\end{enumerate}
  \end{assumption}

We give  some simple examples for 
 which 
Assumption~\ref{2017-05-30:11} is verified.

\begin{example}\label{2017-06-12:00}
  Let $\hat u\in C^{1,2}_b([0,T]\times H,\mathbb{R})$ and 
$u(t,\mathbf{x})\coloneqq \hat u(t,\mathbf{x}(t))$, $(t,\mathbf{x})\in[0,T]\times \mathbb{W}$.
Then 
Assumption~\ref{2017-05-30:11} is verified, with
$\mathcal{D}^-_tu(t,\mathbf{x})= \partial _t\hat u(t,\mathbf{x}(t))$,
for $t\in(0,T)$ and $\mathbf{x}\in \mathbb{W}$,
 $ \overline{\partial _\mathbb{W}u}(t,\mathbf{x}).\mathbf{v}= D_H\hat u(t,\mathbf{x}(t)).\mathbf{v}(t)$,
$ \overline{\partial^2 _\mathbb{W}u}(t,\mathbf{x}).
(\mathbf{v}.\mathbf{w})= D_H\hat u(t,\mathbf{x}(t)).(\mathbf{v}(t),\mathbf{w}(t))$,
for $t\in [0,T]$, $\mathbf{x}\in \mathbb{W}$, $\mathbf{v}
,\mathbf{w}\in \Sone{H}$.
\end{example}


\begin{example}
  Let $\gamma\in C^1([0,T],\mathbb{R})$, $h\in C^{0,2}_b([0,T]\times H,\mathbb{R})$.
For $(t,\mathbf{x})\in [0,T]\times \mathbb{W}$, define
$$
u(t,\mathbf{x})\coloneqq
\int_0^t
h(s,\mathbf{x}(s))\gamma(t-s)ds.
$$
A direct computation gives,
for $(t,\mathbf{x})\in[0,T]\times \mathbb{W}$,
\begin{gather*}
  \mathcal{D}^-_tu(t,\mathbf{x})=h(t,\mathbf{x}(t))\gamma(0)+\int_0^th(s,\mathbf{x}(s))\gamma'(t-s)ds\\
    \overline{ \partial_\mathbb{W} u}(t,\mathbf{x}).\mathbf{v}=
\int_0^t
D_Hh(s,\mathbf{x}(s)).\mathbf{v}(s)
\gamma(t-s)ds\\
    \overline{ \partial^2_\mathbb{W} u}(t,\mathbf{x}).(\mathbf{v},\mathbf{w})=
\int_0^t
D^2 _Hh(s,\mathbf{x}(s)).(\mathbf{v}(s),\mathbf{w}(s)
)
\gamma(t-s)ds
\end{gather*}
and one can easiliy see that
Assumption~\ref{2017-05-30:11} is verified by $u$.
\end{example}

\begin{example}
Let $u$
be a function verifying
Assumption~\ref{2017-05-30:11} and
let
$h\in C_b^{1,2}([0,T]\times \mathbb{R},\mathbb{R})$.
For $(t,\mathbf{x})\in [0,T]\times \mathbb{W}$, define
$
\hat u(t,\mathbf{x})\coloneqq
h(t,u(t,\mathbf{x}))
$.
We have
$$
\mathcal{D}^-_t \hat u(t,\mathbf{x})= \partial _th(t,u(t,\mathbf{x}))+D_Hu(t,u(t,\mathbf{x})).\mathcal{D}^-_tu(t,\mathbf{x})
$$
and 
 $ \overline{ \partial _\mathbb{W}\hat u}, \overline{ \partial^2 _\mathbb{W}\hat u}$ are given by the
 chain rule.
Assumption~\ref{2017-05-30:11} are verified.
\end{example}

\bigskip
Let $B\colon \Soninf{H}\times \Soninf{H}\rightarrow \mathbb{R}$ be a 
continuous
bilinear functional
and let
 $C>0$ such that $|B(\mathbf{x},\mathbf{y})|\leq C|\mathbf{x}|_\infty|\mathbf{y}|_\infty$, for all $\mathbf{x},\mathbf{y}\in \Soninf{H}$.
Let $\mathbf{a}\in \Sone{\mathbb{R}}$, $|\mathbf{a}|_\infty\leq 1$,
and $T\in L_2(U,H)$.
Then $\mathbf{a}Tu\in \Sone {H}$, for all $u\in U$,
and $\mathbf{a}v\in \Sone{H}$, for all $v\in H$.
Clearly
$$
U\times H\rightarrow \mathbb{R},\ (u,v)  \mapsto B(\mathbf{a}Tu,\mathbf{a}v)
$$
is bilinear and continuous.
Let $Q\in L(U,H)$ be the unique linear and continuous operator such that
\begin{equation}
  \label{2017-06-08:01}
  \langle Qu,v\rangle_H=B(\mathbf{a}Tu,\mathbf{a}v)\qquad \forall u\in U,\ \forall v\in H.
\end{equation}
We claim that $Q\in L_2(U,H)$.
Indeed, 
\begin{equation*}
    \sum_{m\in \mathcal{M}}|Qe'_m|_H^2
=
    \sum_{m\in \mathcal{M}}
\sup_{\substack{v\in H\\ |v|_H\leq 1}}
 \left( 
B(\mathbf{a}Te'_m,\mathbf{a}v)
 \right) ^2
\leq
    \sum_{m\in \mathcal{M}}
C^2
|\mathbf{a}Te'_m|_\infty^2
\leq
C^2|T|^2_{L_2(U,H)}<\infty.
\end{equation*}
Then $Q^*\in L_2(H,U)$ and, by \cite[Proposition C.4]{DaPrato2014}, $Q^*T\in L(U)$  is a nuclear operator.
In particular, the number
\begin{equation*}
  \sum_{m\in \mathcal{M}}
  B(\mathbf{a}Te'_m,\mathbf{a}Te'_m)
=
  \sum_{m\in \mathcal{M}}
  \langle Qe'_m,Te'_m\rangle_H
=
  \sum_{m\in \mathcal{M}}
  \langle e'_m,Q^*Te'_m\rangle_U
=
\operatorname{Tr}(Q^*T)
\end{equation*}
is well-defined, finite, and does not depend on the chosen orthonormal basis $\{e'_m\}_{m\in \mathcal{M}}$.
This observation leads to introduce the following well-defined notion.

\begin{definition}
Let $B\colon \Soninf{H}\times \Soninf{H}\rightarrow \mathbb{R}$ be a 
continuous
bilinear functional,
 $\mathbf{a}\in \Sone{\mathbb{R}}$,
 $T\in L_2(U,H)$.
We define
\begin{equation}
\mathbf{T}[B,\mathbf{a}T]\coloneqq  \sum_{m\in \mathcal{M}}B(\mathbf{a}Te'_m,\mathbf{a}Te'_m).
\end{equation}
\end{definition}

\bigskip
Let 
$
b\in \mathcal{L}_\PT^1(\mathbb{W})$, 
$\Phi\in \mathcal{L}_\PT^2(C([0,T],L_2(U,H)))$, and let
$W$ be a $U$-valued cylindrical Wiener process.
For $(\hatc)\in [0,T]\times \mathcal{L}_\PT^1(\mathbb{W})$,
let $X^{\hatc}\in
\mathcal{L}_\PT^1(\mathbb{W})$
 be the process defined by
\begin{equation}
  \label{eq:2015-08-29:03}
  X_t=\hat Y_{\hat t\wedge t}+\int_{\hat t}^{ \hat t\vee t} b_sds+
\int_{\hat t}^{ \hat t\vee t}\Phi_sdW_s\qquad\forall t\in [0,T].
\end{equation}

The first main result of the paper is the following  path-dependent It\=o's formula.

\begin{theorem}\label{2017-05-30:02}
  Suppose that $u$ satisfies 
Assumption~\ref{2017-05-30:11}.
For  $\hat Y\in \mathcal{L}_\PT^1(\mathbb{W})$
and $\hat t\in[0,T]$,
  let $X^{\hatc}$ be  the process defined by \eqref{eq:2015-08-29:03}.
Then
  \begin{enumerate}[(i)]
  \item\label{2017-05-31:02} for all $\omega\in\Omega$, $\mathcal{D}^-_tu(\cdot,X^\hatc(\omega))\in L^1((0,T),\mathbb{R})$;
  \item\label{2017-05-31:03} $\left\{\overline{
 \partial_\mathbb{W} u}(t,X^\hatc).(\mathbf{1}_{[t,T]}b_t)\right\}_{t\in[0,T]}\in L^1_\PT(\mathbb{R})$;
  \item\label{2017-05-31:04} $    \left\{ \overline{\partial_\mathbb{W} u}
(t,X^\hatc).( \mathbf{1}_{[t,T]}\Phi_t)\right\}_{t\in[0,T]} \in L^2_\PT(U^*)$;
  \item\label{2017-05-31:05}
$\left\{\mathbf{T}\left[\overline { \partial_\mathbb{W} ^2u}(t,X^\hatc),\mathbf{1}_{[t,T]}\Phi_t\right]
\right\}_{t\in[0,T]}\in L^1_\PT(\mathbb{R})$.
  \end{enumerate}
For $t\in[\hat t, T]$,
\begin{equation}\label{2017-05-31:09}
  \begin{split}
      u(&t,X^\hatc)=u(\hat t,\hat Y)+\int_{\hat t}^t\Big(
 \mathcal{D}^-_tu(s,X^\hatc)ds
+ 
\overline{ \partial_\mathbb{W} u}(s,X^\hatc).(\mathbf{1}_{[s,T]}b_s)
\Big)ds
\\
&
+\frac{1}{2}\int_{\hat t}^t
\mathbf{T}
\left[ \overline{\partial_\mathbb{W} ^2 u}(s,X^\hatc),\mathbf{1}_{[s,T]}\Phi_s\right]
  ds+
\int_{\hat t}^t \overline{ \partial_\mathbb{W} u}(s,X^\hatc).( \mathbf{1}_{[s,T]}\Phi_s) dW_s,  \qquad\mathbb{P}\mbox{-a.e.}.
\end{split}
\end{equation}
\end{theorem}

\medskip

\begin{remark}
  Notice that, by Example~\ref{2017-06-12:00}, 
\eqref{2017-05-31:09} is a generalization of the
standard It\=o's formula in the  non-path-dependent case.
\end{remark}

The proof of Theorem~\ref{2017-05-30:02}
is obtained
through several partial results.
We begin by preparing a setting useful to approximate path-dependent functionals by non-path-dependent ones, for which we can use the standard (non-path-dependent) stochastic analysis on Hilbert spaces, as presented e.g.\ in \cite{DaPrato2014}.

\bigskip
\label{appr-begin}
For $n\geq 1$, we consider the product Hilbert space $H^n$ endowed with the
scalar product $\langle\cdot,\cdot\rangle_{H^n}$ defined by
$$
\langle x,x'\rangle_{H^n}\coloneqq \sum_{k=1}^n \langle x_k,x_k'\rangle_H\qquad
\forall x=(x_1,\ldots,x_n),\ x'=(x_1',\ldots,x_n')\in H^n.
$$
Let
$\pi\coloneqq \{0=t_1<t_2<\ldots<t_n=T\}$ be a partition of the interval $[0,T]$ and let 
$$
\delta(\pi)\coloneqq \sup_{i=1,\ldots,n-1}|t_{i+1}-t_i|.
$$
Define the operator
$$
\ell_\pi\colon H^n\rightarrow \mathbb{W}
$$
as the linear interpolation on the partition $\pi$, i.e.\
\begin{equation}\label{2015-08-30:08}
  \ell_\pi(x_1,\ldots,x_n)(t)\coloneqq x_1+\sum_{i=1}^{n-1}\frac{t\wedge t_{k+1}-t\wedge t_k}{t_{k+1}-t_k} \left( x_{k+1}-x_k \right)\qquad\forall t\in [0,T].
\end{equation}
The operator $\ell_\pi$ is linear and continuous, with operator norm $1$.
If $\mathbf{x}\in \mathbb{W}$ and if
$w_\mathbf{x}$
 denotes a modulus of continuity for $\mathbf{x}$, then
\begin{equation}
  \label{eq:2015-08-22:00}
 |\ell_\pi \left( \mathbf{x}_{t_2\wedge \cdot}(t),
\mathbf{x}_{t_3\wedge \cdot}(t),
\ldots
\mathbf{x}_{t_{n-1}\wedge \cdot}(t),
\mathbf{x}_{t_n\wedge \cdot}(t),
\mathbf{x}_{t_n\wedge \cdot}(t) \right)
-\mathbf{x}_{t\wedge \cdot}|_\infty\leq 2w_\mathbf{x} \left( \delta(\pi) \right) .
\end{equation}

Let $X$ be given by \eqref{eq:2015-08-29:03}.
We introduce the following $H$-valued processes, obtained by stopping $X$ at certain fixed times.
For $i=1,\ldots,n-1$ and $t\in[0,T]$, let $X^{(\pi,i)}_t$ be 
the continuous process
defined by
\begin{equation}
  \label{2016-02-08:00}
        X^{(\pi,i)}_t\coloneqq X^\hatc_{t_{i+1}\wedge t}=
\hat Y_{\hat t\wedge t\wedge t_{i+1}}+\int_{\hat t}^{\hat t \vee t} \mathbf{1}_{[0,t_{i+1})}(s)b_sds+
\int_{\hat t}^{\hat t \vee t}
\mathbf{1}_{[0,t_{i+1})}(s)\Phi_sdW_s
\end{equation}
and let $X^{(\pi,n)}_t\coloneqq X^\hatc_t$, $t\in[0,T]$.
We define the $H^n$-valued process $X^{(\pi)}$ by
$$
X^{(\pi)}_t\coloneqq ( X^{(\pi,1)}_t,\ldots,X^{(\pi,n)}_t ) \qquad\forall t\in[0,T].
$$
Notice that 
 $X^{(\pi)}\in
\mathcal{L}^1_{\PT} \left( C([0,T],H^n )\right) $.
The dynamics  of  $X^{(\pi)}$ is given by
$$
X^{(\pi)}_t=X^{(\pi)}_{\hat t}+\int_{\hat t}^tb^{(\pi)}_sds+\int_{\hat t}^t\Phi^{(\pi)}_sdW_s \qquad\forall t\in[\hat t,T],
$$
where
$$
X_{\hat t}^{(\pi)}=(
\hat Y_{\hat t\wedge t_{2}},
\hat Y_{\hat t\wedge t_{3}},
\ldots,
\hat Y_{\hat t},
\hat Y_{\hat t}
)\in H^n
$$
and where the coefficients $b^{(\pi)}$ and $\Phi^{(\pi)}$ are 
 the following
\begin{equation}\label{2016-01-28:00}
\hskip-4pt
\begin{dcases}
        b_s^{(\pi)}\coloneqq  (\mathbf{1}_{[0,t_2)}(s)b_s,\mathbf{1}_{[0,t_3)}(s)b_s,\ldots,&\\
\qquad \qquad \ldots,\mathbf{1}_{[0,t_{n-1})}(s)b_s,  
\mathbf{1}_{[0,t_{n})}(s)b_s,
\mathbf{1}_{[0,t_{n}]}(s)b_s) &\forall s\in[0,T]\\
      \Phi_s^{(\pi)}u\coloneqq  (\mathbf{1}_{[0,t_2)}(s)\Phi_su,\mathbf{1}_{[0,t_3)}(s)\Phi_su,\ldots, &\\
\qquad\qquad       \ldots,\mathbf{1}_{[0,t_{n-1})}(s)\Phi_su,
        \mathbf{1}_{[0,t_{n})}(s)\Phi_su,
        \mathbf{1}_{[0,t_{n}]}(s)\Phi_su) &\forall
      s\in[0,T],\ \forall u\in U.
\end{dcases}
\end{equation}
We can verify that $b^{(\pi)}\in L_\PT^1(H^n)$ by
\begin{gather*}
  \mathbb{E} \left[ \int_0^T|b^{(\pi)}_s|_{H^n} \right] =
  \mathbb{E} \left[ \int_0^T|b_s|_H \left( 1+\sum_{j=2}^n \mathbf{1}_{[0,t_i)}(s) \right)^{1/2} ds\right] \leq
n^{1/2}\mathbb{E} \left[ \int_0^T|b_s|_{H} \right]
\end{gather*}
and that  $\Phi^{(\pi)}\in L^2_\PT  \left(L_2(U,H^n) \right) $ by
\begin{gather*}
  \mathbb{E} \left[ \int_0^T|\Phi^{(\pi)}_s|^2_{L_2(U,H^n)} \right] =
  \mathbb{E} \left[ \int_0^T|\Phi_s|^2_{L_2(U,H)} \left( 1+\sum_{j=2}^n \mathbf{1}_{[0,t_i)}(s) \right) ds\right] \leq
n\mathbb{E} \left[ \int_0^T|\Phi_s|^2_{L_2(U,H)} \right].
\end{gather*}
We notice that,  by \eqref{eq:2015-08-22:00} and \eqref{2016-02-08:00},
\begin{equation}
  \label{eq:2015-08-29:00}
  \lim_{\delta\rightarrow 0^+}\sup_{\pi\colon \delta(\pi)\leq \delta}
\sup_{t\in[0,T]}\left|\ell_\pi ( X^{(\pi)}_t(\omega) ) -X^\hatc_{t\wedge \cdot}(\omega)
\right|_\infty=0\qquad \forall\omega\in \Omega.
\end{equation}
\label{appr-end}

\begin{remark}\label{2016-02-05:01}
  The importance of the choice of $b^{(\pi)}$ as in \eqref{2016-01-28:00} can be understood when we consider the composition $\ell_\pi(b^{(\pi)}_s(\omega))$.
If $\delta(\pi)\rightarrow 0$, then $\ell_\pi(b^{(\pi)}_s(\omega))$ converges pointwise to $\mathbf{1}_{[s,T]}(\cdot)b_s(\omega)$ everywhere on $[0,T]$.
On the contrary, if we consider
$$ 
\tilde b_s^{(\pi)}\coloneqq  (\mathbf{1}_{[0,t_1)}(s)b_s,\mathbf{1}_{[0,t_2)}(s)b_s,\ldots,
\mathbf{1}_{[0,t_{n-1})}(s)b_s,
\mathbf{1}_{[0,t_n]}(s)b_s)
$$
then the pointwise limit as $\delta(\pi)\rightarrow 0$ of $\ell_\pi(\tilde b_s^{(\pi)}(\omega))$
is  $0$ on $[0,s)$ and $b_s$ on $(s,T]$,
but it is not guaranteed that the limit  in $s$ exists.
In our approximation framework, we deal  with
sequential continuity with respect to the topology $\sigma^s$ in $\Sone{H}$, wich implies pointwise convergence, as clarified by
Proposition~\ref{2015-08-24:00}\emph{(\ref{2016-01-25:08})}.
Because of that, the choice of $b^{(\pi)}$ as in 
\eqref{2016-01-28:00} will be relevant.
The same comment holds for $\Phi^{(\pi)}$.
\end{remark}

We  will need the following measurability lemma.

\begin{lemma}\label{2015-08-30:05}
  Let $V,Y,Z$ be $H$-valued 
continuous
$\mathbb{F}$-adapted processes.
 Let $E$ be a Banach space and let 
  \begin{gather*}
    \bar f\colon  \mathbb{W}\times \Sones{H}
\times \Sones{H}\rightarrow E
  \end{gather*}
be a  sequentially continuous function.
 Then the process
$$
\Psi\coloneqq \left\{\bar f(V_{t\wedge \cdot},
  \mathbf{1}_{[t,T]} Y_{t},
  \mathbf{1}_{[t,T]}Z_{t})\right\}_{t\in[0,T]}
$$
 is
$\mathbb{F}$-adapted and left-continuous.
\end{lemma}
\begin{proof}
For all $\mathbf{x}\in \mathbb{W}$, the map
$$
[0,T]\rightarrow \mathbb{W},\ t\mapsto \mathbf{x}_{t\wedge \cdot}
$$
is continuous.
Then
$\{V_{t\wedge \cdot}\}_{t\in[0,T]}$ is a $\mathbb{W}$-valued continuous process.
We now show that $\{V_{t\wedge \cdot}\}_{t\in[0,T]}$ is $\mathbb{F}$-adapted.
Let $t\in[0,T]$. Let \mbox{$\pi=\{0=t_1<\ldots<t_n=T\}$} be a partition of $[0,T]$. It is clear that $(V_{t_1\wedge t},\ldots,V_{t_n\wedge t})$ is an $H^n$-valued $\mathcal{F}_t$-measurable random variable. Then $\ell_\pi(V_{t_1\wedge t},\ldots,V_{t_n\wedge t})$ is a $\mathbb{W}$-valued $\mathcal{F}_t$-adapted random variable. 
For all $\mathbf{x}\in\mathbb{W}$,
\begin{equation*}
 |\ell_\pi \left( \mathbf{x}_{t_1\wedge t},
\ldots,
\mathbf{x}_{t_n\wedge t})
 \right)
-\mathbf{x}_{t\wedge \cdot}|_\infty\leq w_\mathbf{x} \left( \delta(\pi) \right),
\end{equation*}
where $w_{\mathbf{x}}$ is a modulus of continuity for $\mathbf{x}$, hence,
 for all $\omega\in \Omega$,
$$
\lim_{\delta(\pi)\rightarrow 0}
\ell_\pi(V_{t_1\wedge t}(\omega),\ldots,V_{t_n\wedge t}(\omega))=
 V_{t\wedge \cdot}(\omega)\ 
\mbox{in }\mathbb{W},\mbox{ uniformly for }t\in[0,T].
$$
This shows 
that $\{V_{t\wedge \cdot}\}_{t\in[0,T]}$ is a $\mathbb{W}$-valued
$\mathbb{F}$-adapted
process. 
The same considerations hold for $\{Y_{t\wedge \cdot}\}_{t\in[0,T]}$ and for $\{Z_{t\wedge \cdot}\}_{t\in[0,T]}$.

Now let $t\in[0,T]$
and let $\{\varphi_n\}_{n\in \mathbb{N}}\subset C([0,T],\mathbb{R})$ be a sequence such that
 \begin{equation}
   \label{eq:2015-08-30:11}
   \begin{dcases}
        0\leq \varphi_n\leq1&\forall n\in \mathbb{N}\\
 \lim_{n\rightarrow \infty}\varphi_n(s)=\mathbf{1}_{[t,T]}(s)&\forall s\in[0,T].
\end{dcases}
\end{equation}
 Since, for every $n\in \mathbb{N}$, the map 
$
H\rightarrow \mathbb{W},\ h\mapsto \varphi_nh
$
is linear and continuous, we have that
$\varphi_n Y_{t}$
and
$\varphi_n Z_t$ are \mbox{$\mathbb{W}$-valued}, 
$\mathcal{F}_t$-measurable random variables.
It follows that $(V_{t\wedge \cdot},\varphi_nY_t,\varphi_nZ_t)$ is a $\mathbb{W}\times \mathbb{W}\times \mathbb{W}$-valued $\mathcal{F}_t$-measurable random variable.
The sequential continuity of $\bar f$
 implies the continuity of the restriction of 
$\bar f$ to $\mathbb{W}\times \mathbb{W}\times \mathbb{W}$, 
then
$\bar f(V_{t\wedge\cdot},\varphi_nY_t,\varphi_nZ_t)$
is an $E$-valued $\mathcal{F}_t$-measurable random variable. 
Now, by \eqref{eq:2015-08-30:11} and
Proposition \ref{2015-08-24:00}\emph{(\ref{2016-01-25:08})}, we have
$$
\begin{dcases}
  \lim_{n\rightarrow \infty}\varphi_nY_t(\omega)=\mathbf{1}_{[t,T]}Y_t(\omega)&\mbox{in }\Sones{H},\ \forall \omega\in \Omega,\\
\lim_{n\rightarrow \infty}\varphi_nZ_t(\omega)=\mathbf{1}_{[t,T]}Z_t(\omega)
&\mbox{in }\Sones{H},\ \forall \omega\in \Omega.
\end{dcases}
$$
By sequential continuity of $\bar f$, we conclude 
$$
\lim_{n\rightarrow \infty}
\bar f(V_{t\wedge\cdot},\varphi_nY_t,\varphi_nZ_t)
=
\bar f(V_{t\wedge\cdot},\mathbf{1}_{[t,T]}Y_t,\mathbf{1}_{[t,T]}Z_t)
\ \mbox{pointwise.}
$$ 
This shows that
 $\Psi_t$  is
 an $E$-valued $\mathcal{F}_t$-measurable random variable,
hence $\Psi$ is $\mathbb{F}$-adapted.

Let $\{t_n\}_{n\in \mathbb{N}}\subset [0,T]$ be a sequence converging to $t$ in $(0,T]$ from the left.
Then the sequence $\{V_{t_n\wedge \cdot}(\omega)\}_{n\in \mathbb{N}}$ converges to $V_{t\wedge \cdot}(\omega)$ in $\mathbb{W}$, for all  $\omega\in\Omega$. Moreover, by 
Proposition \ref{2015-08-24:00}\emph{(\ref{2016-01-25:08})} and continuity of $Y,Z$,
$$
\forall \omega\in\Omega,\ 
\begin{dcases}
  \lim_{n\rightarrow \infty}\mathbf{1}_{[t_n,T]}(\cdot)Y_{t_n}(\omega)=
\mathbf{1}_{[t,T]}(\cdot)Y_{t}&\mbox{in }\Sones{H}\\
\lim_{n\rightarrow \infty}\mathbf{1}_{[t_n,T]}(\cdot)Z_{t_n}=
\mathbf{1}_{[t,T]}(\cdot)Z_{t}&\mbox{in }\Sones{H}.
\end{dcases}
$$
Then, by sequential continuity of $\bar f$, 
we conclude
$\Psi_{t_n}(\omega)\rightarrow \Psi_t(\omega)$.
This proves the left continuity of $\Psi$.
\end{proof}

The  following
proposition
provides a version of It\=o's formula
for G\^ateaux differentiable functions that
will be used later.


\begin{proposition}\label{2016-02-06:01}
Let $\tilde b\in \mathcal{L}_\PT^1(\mathbb{W})$, $\tilde\Phi\in \mathcal{L}_\PT^2(C([0,T],L_2(U,H)))$,
and let $W$ be a $U$-valued cylindrical Wiener process.
Let $t_0\in[0,T]$ and $Y\in \mathcal{L}^1_{\mathcal{P}_T}(\mathbb{W})$.
Let $\tilde X\in \mathcal{L}^1_\PT(\mathbb{W})$ be the It\=o process defined by
\begin{equation}
  \tilde X_t=Y_{t\wedge t_0}+\int_{t_0}^{t_0\vee t} \tilde b_sds+\int_{t_0}^{t_0\vee t}\tilde \Phi_sdW_s\qquad\forall t\in [0,T].
\end{equation}
  Let $f\colon [0,T]\times H\rightarrow \mathbb{R}$ be such that the derivatives $ \partial _tf(t,x)$, $ \partial_vf(t,x)$, $ \partial ^2_{vw}f(t,x)$ exist for all $t\in[0,T]$, $x,v,w\in H$, and are jointly continuous with respect to $t,x,v,w$.
Suppose that
\begin{equation}\label{2016-02-06:07}
  \begin{dcases}
    \sup_{(t,x)\in[0,T]\times H}\frac{| \partial_ tf(t,x)|}{1+|x|_H}<\infty\\
\sup_{\substack{(t,x)\in[0,T]\times H\\v\in H,\ |v|_H\leq 1}}| \partial _vf(t,x)|<\infty\\
\sup_{\substack{(t,x)\in[0,T]\times H\\v,w\in H,\ |v|_H\vee |w|_H\leq 1}}| \partial^2 _{vw}f(t,x)|<\infty.
\end{dcases}
\end{equation}
Then
\begin{enumerate}[(i)]
\item\label{2016-02-06:02}
 $ \{\partial _tf(t,\tilde X_t)\}_{t\in[0,T]}\in \mathcal{L}^1_{\mathcal{P}_T}(C([0,T],\mathbb{R}))$;
\item
\label{2016-02-06:03}
 $\{ \partial _Hf(t,\tilde X_t).\tilde b_t\}_{t\in[0,T]}\in L^1_{\mathcal{P}_T}(\mathbb{R})$;
\item
\label{2016-02-06:04}
 $\{ \partial _Hf(t,\tilde X_t). \tilde \Phi_t\}_{t\in[0,T]}\in L^2_{\mathcal{P}_T}(U^*)$;
\item
\label{2016-02-06:05}
 $\{\operatorname{Tr}[\tilde \Phi_t^* \partial ^2_H(t,\tilde X_t)\tilde\Phi_t]\}_{t\in[0,T]}\in L^1_{\mathcal{P}_T}(\mathbb{R})$;
\end{enumerate}
and, for $t\in[t_0,T]$,
\begin{equation}\label{2016-02-06:06}
  \begin{split}
\hskip-10ptf(t,\tilde X_t)=& f(t_0,Y_{t_0})+\int_{t_0}^t \left( 
 \partial _tf(s,\tilde X_s)+ \partial _Hf(s,\tilde X_s).\tilde b_s+\frac{1}{2}
\operatorname{Tr}[\tilde \Phi_s^* \partial ^2_H(s,\tilde X_s)\tilde \Phi_s]
 \right) ds\\
&+
\int_{t_0}^t
 \partial _Hf(s,\tilde X_s).\tilde \Phi_s dW_s\qquad \qquad \mathbb{P}\mbox{-a.e..}
\end{split}
\end{equation}
\end{proposition}
\begin{proof}
  \emph{(\ref{2016-02-06:02})},
  \emph{(\ref{2016-02-06:03})},
  \emph{(\ref{2016-02-06:04})},
and  \emph{(\ref{2016-02-06:05})} are easily obtained by
the assumptions on
  continuity and boundedness
of
 the differentials of $f$.

We  show how to obtain \eqref{2016-02-06:06}.
Let $\{H_n\}_{n\in \mathbb{N}}$ be an increasing sequence of finite dimensional subspaces of $H$ such that $\bigcup_{n\in \mathbb{N}}H_n$ is dense in $H$.
Let $P_n\colon H\rightarrow H_n$ be the orthogonal projection of $H$ onto $H_n$.
Define
$f_n(t,x)\coloneqq f(t,P_nx)$ for $(t,x)\in [0,T]\times H$, $n\in \mathbb{N}$.
Due to the continuity assumptions on $ \partial _tf$, $  \partial _Hf$, $\partial ^2_Hf$,  the restriction
$f_{|[0,T]\times H_n}$ of $f$ to $[0,T]\times H_n$ belongs to $C^{1,2}([0,T]\times H_n,\mathbb{R})$, hence
$f_n\in C^{1,2}([0,T]\times H,\mathbb{R})$.
Moreover, \eqref{2016-02-06:07} holds also for $f_n$, with bounds uniform in $n$.
Then, by \cite[p.\ 69, Theorem 2.10]{Gawarecki2011}), formula
\eqref{2016-02-06:06}
 holds for all $f_n$.
To conclude the proof it is enough to prove the following limits
\begin{align}
  f_n(t,\tilde X_t)&\rightarrow f(t,\tilde X_t)
&&
 \mathbb{P}\mbox{-a.s.},\ 
 \forall t\in[0,T]
\label{2016-02-06:08}\\
 \partial _tf_n(\cdot,\tilde X_\cdot)&\rightarrow \partial _tf(\cdot,\tilde X_\cdot)&&
\mbox{in\ } L^1_{\mathcal{P}_T}(\mathbb{R})
\label{2016-02-06:09}
\\
 \partial _Hf_n(\cdot,\tilde X_\cdot).\tilde b_\cdot
&\rightarrow \partial _H f(\cdot,\tilde X_\cdot).\tilde b_\cdot
&&\mbox{in}\ L^1_{\mathcal{P}_T}(\mathbb{R})
\label{2016-02-06:10}
\\
\operatorname{Tr}[\tilde \Phi^*_\cdot \partial ^2_Hf_n(\cdot,\tilde X_\cdot)\tilde \Phi_\cdot]&\rightarrow
\operatorname{Tr}[\tilde \Phi^*_\cdot \partial ^2_Hf(\cdot,\tilde X_\cdot)\tilde \Phi_\cdot]
&&\mbox{in\ }L^1_{\mathcal{P}_T}(\mathbb{R})
\label{2016-02-06:11}
\\
\partial _Hf_n(\cdot,\tilde X_\cdot).\tilde \Phi_\cdot&\rightarrow \partial _H f(\cdot,\tilde X_\cdot).\tilde \Phi_\cdot
&&\mbox{in\ }L^2_{\mathcal{P}_T}(U^*).
\label{2016-02-06:12}
\end{align}
Convergence \eqref{2016-02-06:08} is clear.
Since
\eqref{2016-02-06:07} holds with $f_n$ in place of $f$, with bounds uniform in $n$, 
in order
to prove 
\eqref{2016-02-06:09},
\eqref{2016-02-06:10},
\eqref{2016-02-06:11},
\eqref{2016-02-06:12},
it is sufficient to show that those convergences hold pointwise.
Let $\phi\in L_2(U,H)$ and $(t,x)\in [0,T]\times H$.
Let $\{u_n\}_{n\in \mathbb{N}}\subset U$ be a sequence such that $|u_n|_U\leq 1$ for all $n$ and $u_n\rightharpoonup u$. Since $\varphi$ is compact, $\varphi u_n\rightarrow \varphi u$ in $H$, hence $P_n\varphi u_n\rightarrow \varphi u$. By continuity of $ \partial _vf(t,x)$ in $t,x,v$, we then have
$$
\partial _Hf_n(t,x).(\varphi u_n)
=
\partial _Hf(t,P_nx).(P_n\varphi u_n)
\rightarrow
\partial _Hf(t,x).(\varphi u).
$$
Since we also have $\partial _Hf(t,x).(\varphi u_n)
\rightarrow
\partial _Hf(t,x).(\varphi u)$,
we conclude
 $\partial _Hf_n(t,x).\varphi\rightarrow \partial _Hf(t,x).\varphi $ in $U^*$.
This provides 
\eqref{2016-02-06:12}.
The other pointwise convergences  can be proved 
with similar arguments.
\end{proof}

Under the following assumption,
we 
prove
in Proposition~\ref{2016-01-25:01}
 a 
less general
version of Theorem~\ref{2017-05-30:02}, in which the functional $u$ is of the form $u(t,\mathbf{x})=f(\mathbf{x}_{t\wedge \cdot})$.

\begin{assumption}\label{2015-08-28:03}
The function
 $f$ belongs to
$\GatWB{\mathbb{R}}$ and
 its differentials $ \partial f$ and $ \partial ^2f$ are bounded, that is
  \begin{equation}
  \label{eq:2015-08-28:01}
M_1\coloneqq  
  \sup_{\substack{\mathbf{x},\mathbf{v}\in  \mathbb{W}\\
|\mathbf{v}|_\infty\leq 1}}\left| \partial f(\mathbf{x}).\mathbf{v}\right|< \infty
\end{equation}
\begin{equation}
  \label{eq:2015-08-28:02}
M_2\coloneqq
  \sup_{\substack{\mathbf{x},\mathbf{v},\mathbf{w}\in
\mathbb{W}\\
|\mathbf{w}|\vee|\mathbf{v}|_\infty\leq 1}}\left| \partial ^2f(\mathbf{x}).(\mathbf{v},\mathbf{w})
\right|< \infty.
\end{equation}
\end{assumption}

\bigskip
By Remark
\ref{2016-01-27:05},
due to the sequential continuity of the differentials, 
\eqref{eq:2015-08-28:01} and \eqref{eq:2015-08-28:02} are equivalent to
  \begin{equation}
\label{eq:2015-08-28:04}  
M_1=  \sup_{\substack{\mathbf{x}\in  \mathbb{W}\\
\mathbf{v}\in \Sone{H},\ |\mathbf{v}|_\infty\leq 1}}
\left|
\overline{ \partial f}(\mathbf{x}).\mathbf{v}\right|< \infty,
\end{equation}
\begin{equation}
  \label{eq:2015-08-28:05}
  M_2=  \sup_{\substack{\mathbf{x}\in\mathbb{W}\\
\mathbf{v},\mathbf{w}\in \Sone{H},\ |\mathbf{w}|\vee|\mathbf{v}|_\infty\leq 1}}\left|\overline{\partial^2f}(\mathbf{x}).(\mathbf{v},\mathbf{w})\right|< \infty.
\end{equation}

\medskip
\begin{proposition}
\label{2016-01-25:01}
  Suppose that $f$ satisfies Assumption \ref{2015-08-28:03}.
For  $\hat Y\in \mathcal{L}_\PT^1(\mathbb{W})$
and $\hat t\in[0,T]$,
  let $X^{\hatc}$ be  the process defined by \eqref{eq:2015-08-29:03}.
Then
  \begin{enumerate}[(i)]
  \item\label{2015-08-30:04} $\left\{\overline{
 \partial f}(X^\hatc_{t\wedge \cdot}).(\mathbf{1}_{[t,T]}b_t)\right\}_{t\in[0,T]}\in L^1_\PT(\mathbb{R})$;
  \item\label{2015-08-30:14} $    \left\{ \overline{\partial f}
(X^\hatc_{t\wedge \cdot}).( \mathbf{1}_{[t,T]}\Phi_t)\right\}_{t\in[0,T]} \in L^2_\PT(U^*)$;
  \item\label{2015-08-30:15} 
$\left\{\mathbf{T}\left[\overline { \partial ^2f}(X^\hatc_{t\wedge \cdot}),\mathbf{1}_{[t,T]}\Phi_t\right]
\right\}_{t\in[0,T]}\in L^1_\PT(\mathbb{R})$.
  \end{enumerate}
Moreover, for $t\in[\hat t, T]$,
\begin{equation}
  \label{eq:2015-08-29:04}
  \begin{split}
      f(X_{t\wedge \cdot}^\hatc)=f(\hat Y_{\hat t\wedge \cdot})&+\int_{\hat t}^t \left(
\overline{ \partial f}(X^\hatc _{s\wedge \cdot}).(\mathbf{1}_{[s,T]}b_s)+
\frac{1}{2}\mathbf{T}
\left[ \overline{\partial ^2 f}(X^\hatc_{s\wedge \cdot}),\mathbf{1}_{[s,T]}\Phi_s\right]
 \right) ds\\
&+
\int_{\hat t}^t \overline{
 \partial f}(X^\hatc_{s\wedge \cdot}).( \mathbf{1}_{[s,T]}\Phi_s) dW_s,  \qquad\mathbb{P}\mbox{-a.e.}.
\end{split}
\end{equation}
\end{proposition}
\begin{proof}
By Lemma~\ref{2015-08-30:05},
 the process
$$ 
 \left\{ \overline{ \partial f}(X^\hatc_{t\wedge \cdot}).(\mathbf{1}_{[t,T]}b_t) \right\} _{t\in[0,T]}
$$
is left-continuous and adapted, hence  predictable.
Similarly, the process
\begin{equation}
  \label{eq:2016-01-27:06}
   \left\{ \overline {
 \partial f}(X^\hatc_{t\wedge \cdot}).(\mathbf{1}_{[t,T]} \Phi_tu)
 \right\} _{t\in[0,T]}
\end{equation}
is left-continuous and adapted, hence predictable, for all $u\in U$.

If
$(\omega,t)\in\Omega_T$ and $\{u_n\}_{n\in \mathbb{N}}$ is a sequence converging to $0$
 in $U$, then
 \begin{equation}
   \label{eq:2016-04-22:00}
   \{\mathbf{1}_{[t,T]}\Phi_t(\omega)u_n\}_{n\in \mathbb{N}}
 \end{equation}
 is a uniformly bounded sequence in $\Sone{H}$,
converging pointwise to $0$. Then,
by
Proposition \ref{2015-08-24:00}\emph{(\ref{2016-01-25:08})},
the sequence
\eqref{eq:2016-04-22:00} converges to $0$ in $\Sones{H}$.
By $\Sones{H}$-sequential continuity of 
$
\overline { \partial f}(X^\hatc_{t\wedge \cdot}(\omega))
$,
we conclude 
$$
\lim_{n\rightarrow \infty}\overline {
 \partial f}(X^\hatc_{t\wedge \cdot}(\omega)).(\mathbf{1}_{[t,T]}\Phi_t(\omega)u_n)=0.
$$
This shows that, for all $(\omega,t)\in\Omega_T$, $ \overline { \partial f}(X^\hatc_{t\wedge \cdot}(\omega)).(\mathbf{1}_{
[t,T}\Phi_t(\omega))\in U^*$.
Then, by separability of $U$ and by Pettis's  measurability theorem, we have that
$$
   \left\{ \overline { \partial f}(X^\hatc_{t\wedge \cdot}).(\mathbf{1}_{[t,T]}\Phi_t)
 \right\} _{t\in[0,T]}
$$
is a $U^*$-valued predictable process.

We now show the integrability properties in
\emph{(\ref{2015-08-30:04})} and
\emph{(\ref{2015-08-30:14})}.
By \eqref{eq:2015-08-28:04}, we have
$$
\mathbb{E} \left[ \int_0^T\left|\overline{
 \partial f}(X^\hatc_{s\wedge \cdot}).(\mathbf{1}_{[s,T]}b_s)\right|ds \right] \leq M_1T |b|_{\mathcal{L}^1_\PT(C([0,T],H))},
$$
which concludes the proof of \emph{(\ref{2015-08-30:04})}.
Similarly, by  \eqref{eq:2015-08-28:05},
\begin{equation*}
  \begin{split}
    \mathbb{E} \left[ \int_0^T\sup_{\substack{u\in U\\ |u|_U\leq 1}}\left|\overline{ \partial f}(X^\hatc_{s\wedge \cdot}).(\mathbf{1}_{[s,T]}\Phi_su)\right|^2ds \right] &\leq
M_1^2
\mathbb{E} 
\left[\int_0^T
  \left|
    \Phi_s
  \right|_{L(U,H)}^2 ds
\right]\\
&\leq
M_1^2T
|\Phi|^2_{\mathcal{L}^2_\PT(C([0,T],L_2(U,H)))}.
\end{split}
\end{equation*}
This concludes the proof of \emph{(\ref{2015-08-30:14})}.

To  show \emph{(\ref{2015-08-30:15})},
we first prove that the sum defining $\mathbf{T}\left[\overline { \partial ^2 f}(X^\hatc_{t\wedge \cdot}),\mathbf{1}_{[t,T]}\Phi_t\right]$ is convergent.
By \eqref{eq:2015-08-28:05}, we have, 
\begin{equation}\label{2015-08-30:16}
  \begin{split}
   \sum_{m\in \mathcal{M}}\left|\overline {
 \partial ^2f}(X^\hatc_{t\wedge \cdot}).(\mathbf{1}_{[t,T]}\Phi_te'_m,\mathbf{1}_{[t,T]}\Phi_te'_m)\right|
&\leq
M_2\sum_{m\in \mathcal{M}}
|\mathbf{1}_{[t,T]}\Phi_te'_m|^2_\infty\\
&=
M_2\sum_{m\in \mathcal{M}}
|\Phi_te'_m|^2_H\\
&=M_2|\Phi_t|^2_{L_2(U,H)}.
  \end{split}
\end{equation}
Then $\mathbf{T}\left[\overline { \partial ^2f}(X^\hatc_{t\wedge \cdot})
,\mathbf{1}_{[t,T]}\Phi_t\right]$ is well defined, for all $t\in [0,T]$.
By Lemma
\ref{2015-08-30:05},
for every $m\in \mathcal{M}$, the process
\begin{equation}\label{2017-05-31:08}
   \left\{ \overline{ \partial ^2f}(t,X^\hatc_{t\wedge \cdot}).(\mathbf{1}_{[t,T]}\Phi_te'_m,\mathbf{1}_{[t,T]}\Phi_te'_m) \right\}_{t\in[0,T]}
 \end{equation}
 is adapted and left-continuous, hence predictable.
Then $\left\{\mathbf{T}\left[\overline { \partial ^2
f}
(X^\hatc_{t\wedge \cdot}),\mathbf{1}_{[t,T]}\Phi_t\right]\right\}_{t\in[0,T]}$ is predictable. It is also integrable, by
\eqref{2015-08-30:16}.

\smallskip
We finally address formula \eqref{eq:2015-08-29:04}.
We will derive it from the standard It\=o's formula in Hilbert spaces,
by using the approximation framework introduced at pp.~\pageref{appr-begin}--\pageref{appr-end}.

Since,
by Assumption \ref{2015-08-28:03},
$f\in 
\Gatot{\mathbb{W}}{\mathbb{R}}{2}
$, by linearity of $\ell_\pi$ we have that
$$
f_\pi\colon  H^n\rightarrow \mathbb{R},\ x\mapsto f(\ell_\pi (x))
$$
  $f_\pi$ is strongly continuously G\^ateaux differentiable up to order $2$ on $H^n$, with
    \begin{equation}
      \label{2015-08-30:20}
       \partial f_\pi(x).v= \partial 
f(\ell_\pi(x)).\ell_\pi(v),
    \end{equation}
for all $ (x,v)\in H^n\times H^n$,
    \begin{equation}
      \label{2015-08-30:21}
       \partial ^2f_\pi(x).(v,w)= \partial ^2f(\ell_\pi(x)
).(\ell_\pi(v),\ell_\pi(w)),
    \end{equation}
for all $ (x,v,w)\in  H^n\times H^n\times H^n$.
Then we can apply
the standard
 It\=o's formula, in the version
provided by Proposition \ref{2016-02-06:01},
 to the predictable pathwise continuous process
$$
\left\{f_\pi ( X^{(\pi)}_t )\right\}_{t\in[0,T]} =
 \left\{ f ( \ell_\pi ( X^{(\pi)}_t )  ) \right\}_{t\in[0,T]}.
$$
For $t\in[\hat t, T]$, we have
\begin{equation}\label{2015-08-30:18}
  \begin{multlined}[c][.85\displaywidth]
  f_\pi(X^{(\pi)}_t)
=f_\pi(X^{(\pi)}_{\hat t})  +
\int_{\hat t}^t 
  \left( 
     \partial f_\pi(X^{(\pi)}_s). b^{(\pi)}_s  
   \frac{1}{2}\operatorname{Tr} \left( ( \Phi^{(\pi)}_s ) ^* \partial ^2f_\pi(X_s^{(\pi)})\Phi^{(\pi)}_s  \right)  
    \right) ds\\
  +\int_{\hat t}^t       \partial f_\pi(X^{(\pi)}_s). \Phi^{(\pi)}_s dW_s \quad \mathbb{P}\mbox{-a.e.}.
  \end{multlined}
\end{equation}

Through several steps,
we are going to prove that the terms appearing in 
\eqref{2015-08-30:18} converge to the corresponding terms in 
\eqref{eq:2015-08-29:04}, as $\delta(\pi)\rightarrow 0$.

Let $\{\pi_n\}_{n\in \mathbb{N}}$ be a sequence of partition of $[0,T]$ such that $\lim_{n\rightarrow \infty}\delta(\pi_n)=0$.

\smallskip
\emph{\underline{Step 1.}} 
By
\eqref{eq:2015-08-29:00} and by continuity of $f$, we 
 immediately have that, for $t\in[0,T]$,
$f_\pi(X^{(\pi_n)}_t)\rightarrow f(X^\hatc_{t\wedge \cdot})$
 $\mathbb{P}$-a.e..

\smallskip
\emph{\underline{Step 2.}} We show that 
\begin{equation}
  \label{eq:2016-01-28:02}
  \lim_{n\rightarrow \infty}  \partial  f_{\pi_n}(X_{\#}^{(\pi_n)})
.(b_\#^{(\pi_n)})=
\overline{
 \partial f}(X^\hatc_{\#\wedge \cdot}).(\mathbf{1}_{[\#,T]}b_\#)\ \mbox{in }L^1_\PT(\mathbb{R}).
\end{equation}
We notice that, by the very definition of $b_s^{(\pi_n)}$ in \eqref{2016-01-28:00} and of $\ell_\pi$ (see also Remark~\ref{2016-02-05:01}), we have, for all $\omega\in\Omega$ and $s\in[0,T]$,
$$
\ell_{\pi_n}(b_s^{(\pi_n)}(\omega))=
\begin{dcases}
  b_s(\omega) &\mbox{on }[s,T]\\
  0& \mbox{on }[0,s-2\delta(\pi_n)]
\end{dcases}
$$
and  $\sup_{n\in \mathbb{N}} |\ell_{\pi_n}(b_s^{(\pi_n)}(\omega))|_\infty\leq |b_s(\omega)|_H$.
By Proposition~\ref{2015-08-24:00}\emph{(\ref{2016-01-25:08})}, it follows
\begin{equation}
  \label{2015-08-30:22}
  \lim_{n\rightarrow \infty}\ell_{\pi_n}(b_s^{(\pi_n)}(\omega))=\mathbf{1}_{[s,T]}b_s(\omega)\mbox{ in }\Sones{H}, \ \forall (\omega,s)\in\Omega_T.
\end{equation}
By  \eqref{eq:2015-08-28:04}
and \eqref{2015-08-30:20},
\begin{equation}\label{2015-08-30:23}
  \begin{multlined}[c][.85\displaywidth]
    \sup_{s\in[0,T]}|   \partial f_{\pi_n}(X^{(\pi_n)}_s).b_s^{(\pi_n)}|+
\sup_{s\in[0,T]}|\overline{
 \partial f}(X^\hatc_{s\wedge \cdot}).(\mathbf{1}_{[s,T]}b_s)|\\
\leq
M_1 \left( \sup_{s\in[0,T]}|\ell_{\pi_n}( b_s^{(\pi_n)})|_\infty
+\sup_{s\in[0,T]}|\mathbf{1}_{[s,T]}b_s|_\infty \right) 
= 2M_1|b|_\infty.
\end{multlined}
\end{equation}
By
\eqref{eq:2015-08-29:00}, \eqref{2015-08-30:22},
 \eqref{2015-08-30:23}, sequential continuity of $\overline{ \partial f}$, and
  Lebesgue's  dominated convergence theorem,  we obtain
$$
\lim_{n\rightarrow \infty}\mathbb{E} \left[ \int_0^T
\left| \partial f_{\pi_n}(X_s^{(\pi_n)}).b_s^{(\pi_n)}-
\overline{ \partial f}(X^\hatc_{s\wedge \cdot}).(\mathbf{1}_{[s,T]}b_s)\right|ds
 \right] =0,
$$
which
provides \eqref{eq:2016-01-28:02}.

\smallskip
\emph{\underline{Step 3.}} We show that
\begin{equation}
  \label{eq:2016-01-28:01}
  \lim_{n\rightarrow \infty} \partial f_{\pi_n}(X_\#^{(\pi_n)}). \Phi_\#^{(\pi_n)}=\overline{
 \partial f}(X^\hatc_{\#\wedge \cdot}).(\mathbf{1}_{[\#,T]}\Phi_\#)\ \mbox{in }L^2_\PT(U^*).
\end{equation}
Let $\{u_n\}_{n\in \mathbb{N}}$ be a sequence weakly convergent to $u$ in the unit ball of $U$. 
Since $\Phi_s(\omega)$ is compact, 
 $\Phi_s(\omega)u_n\rightarrow \Phi_s(\omega)u$ strongly in $H$
 for all $(\omega,s)\in \Omega_T$.
We also have, for $n\in \mathbb{N}$,
$$
\ell_{\pi_n} ( \Phi_s^{(\pi_n)}(\omega)u_n ) =
\begin{dcases}
  \Phi_s(\omega)u_n&\mbox{on }[s,T]\\
  0& \mbox{on }[0,s-2\delta(\pi_n)).
\end{dcases}
$$
and 
$$
\sup_{n\in \mathbb{N}} \left|\ell_{\pi_n} ( \Phi_s^{(\pi_n)}(\omega)u_n)\right|_\infty\leq \sup_{n\in \mathbb{N}}|\Phi_s(\omega)u_n|_H\leq |\Phi_s(\omega)|_{L(U,H)}.
$$
Then, by Proposition~\ref{2015-08-24:00}\emph{(\ref{2016-01-25:08})},
\begin{equation}
  \label{2015-08-30:22-b}
  \lim_{n\rightarrow \infty}\ell_{\pi_n} ( \Phi_s^{(\pi_n)}(\omega)u_n ) =\mathbf{1}_{[s,T]} \Phi_s(\omega)u\  \mbox{in}\ \Sones{H},\ \forall (\omega,s)\in\Omega_T.
\end{equation}
By \eqref{eq:2015-08-29:00},
\eqref{2015-08-30:20}, 
 \eqref{2015-08-30:22-b},
we obtain
$$
\lim_{n\rightarrow \infty}
\left| \partial f_{\pi_n}(X_s^{(\pi_n)}). (\Phi_s^{(\pi_n)}u_n)-\overline{ \partial f}(X^\hatc_{s\wedge \cdot}).(\mathbf{1}_{[s,T]}\Phi_su)
\right|
=0\qquad \forall (\omega,s)\in \Omega_T.
$$
By \eqref{2015-08-30:22-b} and sequential continuity of $\overline{ \partial f}$,
we  have
$$
\lim_{n\rightarrow \infty}
\left|\overline{ \partial f}(X^\hatc_{s\wedge \cdot}).(\mathbf{1}_{[s,T]}\Phi_s(u_n-u))
\right|
=0\qquad \forall (\omega,s)\in \Omega_T.
$$
Since the weakly convergent sequence $\{u_n\}_{n\in \mathbb{N}}$ is arbitrary, the two limits above let us to conclude
\begin{equation}
  \label{eq:2016-01-28:04}
  \lim_{n\rightarrow \infty}
\left| \partial  f_{\pi_n}(X_s^{(\pi_n)}).\Phi_s^{(\pi_n)}-\overline{
 \partial f}(X^\hatc_{s\wedge \cdot}).(\mathbf{1}_{[s,T]}\Phi_s)
\right|_{U^*}
=0\qquad \forall (\omega,s)\in \Omega_T.
\end{equation}
Moreover,  by \eqref{eq:2015-08-28:04}
and \eqref{2015-08-30:20}, for $u\in U$, $|u|_U=1$,
\begin{equation}\label{2015-08-30:23-b}
  \begin{split}
\hskip-10pt        \sup_{s\in[0,T]}|   \partial f_{\pi_n}
(X^{(\pi_n)}_s).&\Phi_s^{(\pi_n)}u|+
\sup_{s\in[0,T]}|\overline{ \partial f}(X^\hatc_{s\wedge \cdot}).(\mathbf{1}_{[s,T]}\Phi_su)|\\
&\leq
M_1 \left( \sup_{s\in[0,T]}\left|\ell_{\pi_n}  ( \Phi_s^{(\pi_n)}u ) \right|_\infty
+\sup_{s\in[0,T]}\left|\mathbf{1}_{[s,T]}  \Phi_su  \right|_\infty \right) \\
&\leq 2M_1\sup_{s\in[0,T]}|\Phi_s|_{L(U,H)}
\leq 2M_1\sup_{s\in[0,T]}|\Phi_s|_{L_2(U,H)}.
\end{split}
\end{equation}
By \eqref{eq:2016-01-28:04},
 \eqref{2015-08-30:23-b}, 
 and by Lebesgue's  dominated convergence theorem,  we obtain
$$
\lim_{n\rightarrow \infty}\mathbb{E} \left[ \int_0^T
\left| \partial  f_{\pi_n}(X_s^{(\pi_n)})(\Phi_s^{(\pi_n)}u)-\overline{ \partial f}(X^\hatc_{s\wedge \cdot})(\mathbf{1}_{[s,T]}(\Phi_su)\right|^2_{U^*}ds
 \right] =0.
$$
This provides \eqref{eq:2016-01-28:01}.

\smallskip
\emph{\underline{Step 4.}} We show that
\begin{equation}
  \label{eq:2016-01-28:05}
  \lim_{n\rightarrow \infty}\operatorname{Tr} \left( ( \Phi_\#^{(\pi_n)} ) ^* \partial ^2f_{\pi_n}(X^{(\pi_n)}_\#)\Phi^{(\pi_n)}_\#  \right)=\mathbf{T}
\left[
\overline { \partial ^2f}(X^\hatc_{\#\wedge \cdot}),\mathbf{1}_{[\#,T]}\Phi_{\#}
\right]\mbox{ in }L^1_\PT(\mathbb{R}).
\end{equation}
By
\begin{multline*}
\left|  \partial ^2f_{\pi_n} ( X^{(\pi_n)}_s ). (  \Phi_s^{(\pi_n)}e'_m  , \Phi_s^{(\pi_n)}e'_m   )\right|+
     \left|\overline { \partial ^2f}(X^\hatc_{s\wedge \cdot}).(\mathbf{1}_{[s,T]}\Phi_se'_m,\mathbf{1}_{[s,T]}\Phi_se'_m)\right|
\\ 
\leq
M_2
\left(|\ell_{\pi_n} ( \Phi^{(\pi_n)}_se'_m ) |^2_\infty+
|\mathbf{1}_{[s,T]}\Phi_se'_m|^2_\infty\right)
=
2M_2
|\Phi_se'_m|^2_H,
\end{multline*}
and 
$$
\sum_{m\in \mathcal{M}}\mathbb{E} \left[ \int_0^T |\Phi_se'_m|_H^2ds \right]=
\mathbb{E} \left[ \int_0^T |\Phi_s|_{L_2(U,H)}^2ds \right] <\infty,
$$
we can apply Lebesgue's  dominated convergence theorem and
obtain
\begin{equation*}
  \begin{split}
      \lim_{n\rightarrow \infty}\mathbb{E} &\left[ \int_0^T
\left|\operatorname{Tr} \left( ( \Phi^{(\pi_n)}_s ) ^* \partial ^2f_{\pi_n}(X_s^{(\pi_n)})\Phi^{(\pi_n)}_s  \right)-\mathbf{T}
\left[
\overline { \partial ^2f}(X_{s\wedge \cdot}),\mathbf{1}_{[s,T]}\Phi_s
\right]
\right|
ds \right]\\
&\leq
      \lim_{n\rightarrow \infty}
\sum_{m\in \mathcal{M}}
\mathbb{E} \left[ \int_0^T
\left|  \partial ^2f_{\pi_n} ( X^{(\pi_n)}_s ).(  \Phi_s^{(\pi_n)}e'_m  , \Phi_s^{(\pi_n)}e'_m   )\right.\right.\\
&\hskip150pt \left.\left.-\overline { \partial ^2
f}(X^\hatc_{s\wedge \cdot}).(\mathbf{1}_{[s,T]}\Phi_se'_m,\mathbf{1}_{[s,T]}\Phi_se'_m)\right|ds
 \right] \\
&
=
\sum_{m\in \mathcal{M}}
\mathbb{E} \left[ \int_0^T
      \lim_{n\rightarrow \infty}
\left|  \partial ^2f_{\pi_n} ( X^{(\pi_n)}_s ). (  \Phi_s^{(\pi_n)}e'_m  , \Phi_s^{(\pi_n)}e'_m   )\right.\right.\\
&\hskip150pt\left.\left.-\overline {
 \partial ^2 f}(X^\hatc_{s\wedge \cdot}).(\mathbf{1}_{[s,T]}\Phi_se'_m,\mathbf{1}_{[s,T]}\Phi_se'_m)\right|ds
 \right] \\
&=0
\end{split}
\end{equation*}
where the pointwise convergence of the latter integrand comes from the sequential continuity of $\overline{ \partial ^2
f}$,
 from \eqref{eq:2015-08-29:00},
and  from
$$
\lim_{n\rightarrow \infty}\ell_{\pi_n}(\Phi_s^{(\pi_n)}(\omega)e'_m)=\mathbf{1}_{[s,T]} \Phi_s(\omega)e'_m  \ \mbox{in }\Sones{H}, \ \forall (\omega,s)\in\Omega_T,\ \forall m\in \mathcal{M}
$$
(that comes from \eqref{2015-08-30:22-b} with $u_n=u=e'_m$ for all $n$).

\smallskip
\emph{\underline{Step 5.}} We can now conclude the proof of the theorem, because
\eqref{eq:2015-08-29:04} is obtained by passing to the limit $n\rightarrow \infty$  in
\eqref{2015-08-30:18} (with $\pi$ replaced by $\pi_n$),
and by considering the partial results of  Step 1, {Step 2}, {Step 3}, {Step 4}.
\end{proof}

We can now prove Theorem~\ref{2017-05-30:02}.

\begin{proof}[\textbf{Proof of Theorem~\ref{2017-05-30:02}}]
\emph{(\ref{2017-05-31:02})}
By continuity of $u$, for $h\in(0,T)$, both $\{u(t,X_{(t-h)\wedge \cdot}^\hatc)\}_{t\in[h,T]}$ and 
$\{u(t-h,X^\hatc_{(t-h)\wedge \cdot})\}_{t\in[h,T]}$ are pathwise continuous and $\mathbb{F}$-adapted, hence predictable.
In particular, $\mathcal{D}^-_tu(\cdot,X^{\hatc})$ is predictable on $(0,T)$ and then
$\mathcal{D}^-_tu(\cdot,X^{\hatc}(\omega))$ is measurable for all $\omega\in\Omega$.
Moreover, for $\omega\in\Omega$, 
the map $[0,T]\rightarrow \mathbb{W},\ t \mapsto  X^\hatc_{t\wedge \cdot}(\omega)$ is continuous,
hence $\{X^\hatc_{t\wedge \cdot}(\omega)\}_{t\in[0,T]}$ is compact in $ \mathbb{W}$
 and \eqref{2017-05-31:00} implies $\mathcal{D}^-_tu(\cdot,X^\hatc(\omega))\in L^1((0,T),\mathbb{R})$.

\emph{(\ref{2017-05-31:03})}+
\emph{(\ref{2017-05-31:04})}+
\emph{(\ref{2017-05-31:05})}
For $n\geq 1$, let $t^n_k\coloneqq kT/n$,
for 
$k=0,\ldots,n$.
By applying
Lemma~\ref{2015-08-30:05}
to
$
\overline{ \partial _\mathbb{W}u}(t^n_k,\cdot)
$, 
 for $k=1,\ldots,n$,
we obtain the predictability of the process
\begin{gather*}
  \left\{\overline{
 \partial_\mathbb{W} u}(t^n_k,X^\hatc_{t\wedge \cdot}).(\mathbf{1}_{[t,T]}b_t)\right\}_{t\in[0,T]}\in L^1_\PT(\mathbb{R})
\qquad
\forall k=1,\ldots,n.
\end{gather*}
By
Assumption~\ref{2017-05-30:11}%
\emph{(\ref{2017-05-31:07})}, for all $t\in (0,T]$ and all $\omega\in\Omega$,
\begin{equation*}
  \overline{
 \partial_\mathbb{W} u}(t,X^\hatc(\omega)).(\mathbf{1}_{[t,T]}b_t(\omega))=
\lim_{n\rightarrow \infty}
\sum_{k=1}^n
\mathbf{1}_{(t^n_{k-1},t^n_k]}(t)
\overline{
 \partial_\mathbb{W} u}(t^n_k,X^\hatc_{t\wedge \cdot}(\omega)).(\mathbf{1}_{[t,T]}b_t(\omega)),
\end{equation*}
which shows that 
$  \left\{\overline{
 \partial_\mathbb{W} u}(t,X^\hatc).(\mathbf{1}_{[t,T]}b_t)\right\}_{t\in[0,T]}$
is predictable.

In  the same way, by applying
Lemma~\ref{2015-08-30:05}
and Pettis's measurability theorem,
 we see that the $U^*$-valued process
$    \left\{ \overline{\partial_\mathbb{W} u}
(t,X^\hatc).( \mathbf{1}_{[t,T]}\Phi_t)\right\}_{t\in[0,T]} $ is predictable.

We now address 
$\left\{\mathbf{T}\left[\overline { \partial_\mathbb{W} ^2u}(t,X^\hatc),\mathbf{1}_{[t,T]}\Phi_t\right]
\right\}_{t\in[0,T]}$.
Again by Lemma~\ref{2015-08-30:05},
 the process
\begin{equation*}
   \left\{ \overline{ \partial_\mathbb{W} ^2u}(t^n_k,X^\hatc_{t\wedge \cdot}).(\mathbf{1}_{[t,T]}\Phi_te'_m,\mathbf{1}_{[t,T]}\Phi_te'_m) \right\}_{t\in[0,T]}
 \end{equation*}
is predictable, for all $m\in \mathcal{M}$.
Thanks to 
Assumption~\ref{2017-05-30:11}\emph{(\ref{2017-05-31:07})}, we have, for all $t\in(0,T]$ and $\omega\in\Omega$,
\begin{multline*}
\overline{ \partial_\mathbb{W} ^2u}(t,X^\hatc_{t\wedge \cdot}).(\mathbf{1}_{[t,T]}\Phi_te'_m,\mathbf{1}_{[t,T]}\Phi_te'_m) \\
=
\lim_{n\rightarrow \infty}
\sum_{k=1}^n
\mathbf{1}_{(t^n_{k-1},t^n_k]}(t)
 \overline{ \partial_\mathbb{W} ^2u}(t^n_k,X^\hatc_{t\wedge \cdot}).(\mathbf{1}_{[t,T]}\Phi_te'_m,\mathbf{1}_{[t,T]}\Phi_te'_m). 
\end{multline*}
Then 
$
 \left\{\overline{ \partial_\mathbb{W} ^2u}(t,X^\hatc).(\mathbf{1}_{[t,T]}\Phi_te'_m,\mathbf{1}_{[t,T]}\Phi_te'_m)  \right\}_{t\in[0,T]} 
$
is predictable, hence
$$
\left\{\mathbf{T}\left[\overline { \partial_\mathbb{W} ^2u}(t,X^\hatc),\mathbf{1}_{[t,T]}\Phi_t\right]
\right\}_{t\in[0,T]}
$$
is predictable too.

Finally, the integrability properties claimed in 
\emph{(\ref{2017-05-31:03})},\emph{(\ref{2017-05-31:04})},\emph{(\ref{2017-05-31:05})}
are  proved exactly as for
Proposition~\ref{2016-01-25:01}\emph{(\ref{2015-08-30:04})},\emph{(\ref{2015-08-30:14})},\emph{(\ref{2015-08-30:15})}
by using 
Assumption~\ref{2017-05-30:11}\emph{(\ref{2017-05-31:06})}.

\smallskip
We now prove formula \eqref{2017-05-31:09}.
Considering Remark~\ref{2017-06-10:02},
without loss of generality we can assume $t=T$.
Let $n\geq 1$ and let  $\hat t=t^n_0<\ldots< t^n_n=T$ be a partition of $[\hat t,T]$, with $t_k^n-t^n_{k-1}= (T-\hat t)/n$, for $k=1,\ldots,n$.
We first write
\begin{equation}\label{2017-05-31:15}
  \begin{split}
    u(T,X^\hatc)&-u(\hat t,\hat Y)
    =
    \sum_{k=1}^n
     \left( u(t^n_k,X^\hatc)-
    u(t^n_{k-1},X^\hatc)\right) \\
&=
    \sum_{k=1}^n
     \left( 
       u(t^n_k,X^\hatc)-
    u(t^n_k,X^\hatc_{t^n_{k-1}\wedge \cdot})
  \right) +
  \sum_{k=1}^n
     \left( 
       u(t^n_k,X^\hatc_{t^n_{k-1}\wedge \cdot})-
    u(t^n_{k-1},X^\hatc)
  \right) \\
  &\eqqcolon\mathbf{I}_n+\mathbf{II}_n.
  \end{split} 
\end{equation}
For $k=1,\ldots,n$, 
due to our assumptions on $u$,
we can apply
Proposition~\ref{2016-01-25:01} to  $u(t^n_k,\cdot)$,
then \eqref{eq:2015-08-29:04} gives
\begin{equation*}
  \begin{split}
      u(t^n_k,X^\hatc)=
&u(t^n_k,\hat Y_{\hat t\wedge \cdot})+\int_{\hat t}^{t^n_k} \left(
\overline{ \partial_\mathbb{W} u}(t^n_k,X^\hatc _{s\wedge \cdot}).(\mathbf{1}_{[s,T]}b_s)+
\frac{1}{2}\mathbf{T}
\left[ \overline{\partial ^2_\mathbb{W} u}(t^n_k,X^\hatc_{s\wedge \cdot}),\mathbf{1}_{[s,T]}\Phi_s\right]
 \right) ds\\
&+
\int_{\hat t}^{t^n_k} \overline{
 \partial_\mathbb{W} u}(t^n_k,X^\hatc_{s\wedge \cdot}).( \mathbf{1}_{[s,T]}\Phi_s) dW_s\\
=&u(t^n_k,X^\hatc_{t^n_{k-1}\wedge \cdot}
)+\int_{t^n_{k-1}}^{t^n_k} \left(
\overline{ \partial_\mathbb{W} u}(t^n_k,X^\hatc _{s\wedge \cdot}).(\mathbf{1}_{[s,T]}b_s)+
\frac{1}{2}\mathbf{T}
\left[ \overline{\partial ^2_\mathbb{W} u}(t^n_k,X^\hatc_{s\wedge \cdot}),\mathbf{1}_{[s,T]}\Phi_s\right]
 \right) ds\\
&+
\int_{t^n_{k-1}}^{t^n_k} \overline{
 \partial_\mathbb{W} u}(t^n_k,X^\hatc_{s\wedge \cdot}).( \mathbf{1}_{[s,T]}\Phi_s) dW_s,\qquad\qquad\qquad\qquad\qquad
\mbox{$\mathbb{P}$-a.e..}
\end{split}
\end{equation*}
Then
\begin{equation*}
  \begin{split}
    \mathbf{I}_n=&\int_{\hat t}^T
\sum_{k=1}^n
\mathbf{1}_{(t^n_{k-1},t^n_k]}(s)\left(
\overline{ \partial_\mathbb{W} u}(t^n_k,X^\hatc _{s\wedge \cdot}).(\mathbf{1}_{[s,T]}b_s)+
\frac{1}{2}\mathbf{T}
\left[ \overline{\partial ^2_\mathbb{W} u}(t^n_k,X^\hatc_{s\wedge \cdot}),\mathbf{1}_{[s,T]}\Phi_s\right]
 \right) ds\\
&+
\int_{\hat t}^T
\sum_{k=1}^n
\mathbf{1}_{(t^n_{k-1},t^n_k]}(s)
 \overline{
 \partial_\mathbb{W} u}(t^n_k,X^\hatc_{s\wedge \cdot}).( \mathbf{1}_{[s,T]}\Phi_s) dW_s,\qquad\qquad\qquad\qquad\qquad
\mbox{$\mathbb{P}$-a.e..}
  \end{split}
\end{equation*}
By 
Assumption~\ref{2017-05-30:11}\emph{(\ref{2017-05-31:06})},\emph{(\ref{2017-05-31:07})},
we can apply Lebesgue's dominated convergence theorem 
 (the integrands are estimated similarly as done in Steps~2--4 of the proof of Proposition~\ref{2016-01-25:01})
and obtain
\begin{multline*}
\lim_{n\rightarrow \infty}    
\sum_{k=1}^n
\mathbf{1}_{(t^n_{k-1},t^n_k]}(\#)\left(
\overline{ \partial_\mathbb{W} u}(t^n_k,X^\hatc _{\#\wedge \cdot}).(\mathbf{1}_{[\#,T]}b_\#)+
\frac{1}{2}\mathbf{T}
\left[ \overline{\partial ^2_\mathbb{W} u}(t^n_k,X^\hatc_{\#\wedge \cdot}),\mathbf{1}_{[\#,T]}\Phi_\#\right]
 \right) \\
=
\overline{ \partial_\mathbb{W} u}(\#,X^\hatc _{\#\wedge \cdot}).(\mathbf{1}_{[\#,T]}b_\#)+
\frac{1}{2}\mathbf{T}
\left[ \overline{\partial ^2_\mathbb{W} u}(\#,X^\hatc_{\#\wedge \cdot}),\mathbf{1}_{[\#,T]}\Phi_\#\right]
\mbox{\ in\ }L^1_{\mathcal{P}_T}(\mathbb{R})
\end{multline*}
and
\begin{equation*}
\lim_{n\rightarrow \infty}    \sum_{k=1}^n
\mathbf{1}_{(t^n_{k-1},t^n_k]}(\#)
 \overline{
 \partial_\mathbb{W} u}(t^n_k,X^\hatc_{\#\wedge \cdot}).( \mathbf{1}_{[\#,T]}\Phi_\#)=
 \overline{
 \partial_\mathbb{W} u}(\#,X^\hatc_{\#\wedge \cdot}).( \mathbf{1}_{[\#,T]}\Phi_\#)
\mbox{\ in\ }L^2_{\mathcal{P}_T}(U^*).
\end{equation*}
The two limits above permit to 
obtain the following limit in 
$L^1(\Omega,\mathbb{R})$:
\begin{equation}\label{2017-05-31:16}
  \begin{split}
\lim_{n\rightarrow \infty}    \mathbf{I}_n=&\int_{\hat t}^T
\left(
\overline{ \partial_\mathbb{W} u}(s,X^\hatc _{s\wedge \cdot}).(\mathbf{1}_{[s,T]}b_s)+
\frac{1}{2}\mathbf{T}
\left[ \overline{\partial ^2_\mathbb{W} u}(s,X^\hatc_{s\wedge \cdot}),\mathbf{1}_{[s,T]}\Phi_s\right]
 \right) ds\\
&\qquad \qquad\qquad \qquad\qquad\qquad+
\int_{\hat t}^T
 \overline{
 \partial_\mathbb{W} u}(s,X^\hatc_{s\wedge \cdot}).( \mathbf{1}_{[s,T]}\Phi_s) dW_s.
\end{split}
\end{equation}

We now address the term $\mathbf{II}_n$.
By Assumption~\ref{2017-05-30:11}\emph{(\ref{2017-05-31:10})},
continuity of $u$,
and 
recalling Remark~\ref{2017-06-01:01},
 we can apply 
\cite[(1.4.4), p.\ 23]{Flett1980} 
and conclude that
$(t,T)\rightarrow \mathbb{R},\ t \mapsto  u(s,\mathbf{x}_{t\wedge \cdot})$  is Lipschitz.
We can then write
\begin{equation}\label{2017-05-31:13}
  \begin{split}
        \mathbf{II}_n&=
    \sum_{k=1}^n
    \int_{t^n_{k-1}}^{t^n_k}\frac{d}{ds}u(s,X^\hatc_{t^n_{k-1}\wedge \cdot})ds=
    \sum_{k=1}^n
    \int_{t^n_{k-1}}^{t^n_k}\mathcal{D}^-_tu(s,X^\hatc_{t^n_{k-1}\wedge \cdot})ds\\
&=
    \int_{\hat t}^T
 \left(     \sum_{k=1}^n
\mathbf{1}_{(t^n_{k-1},t^n_k]}(s)
\mathcal{D}^-_tu(s,
X^\hatc_{t^n_{k-1}\wedge \cdot})
 \right) 
ds.
\end{split}
\end{equation}
Fix $\omega\in\Omega$.
As noticed at the beginning of the proof, the set $K\coloneqq \{X^\hatc_{t\wedge \cdot}(\omega)\}_{t\in[0,T]}$ is compact in $\mathbb{W}$.
Then, by Assumption~\ref{2017-05-30:11}\emph{(\ref{2017-05-31:10})}, there exists $M_K>0$ (depending on $\omega$, since our compact set  $K$ depends on $\omega$ too) such that
\begin{equation}\label{2017-05-31:11}
\Big|  \sum_{k=1}^n
\mathbf{1}_{(t^n_{k-1},t^n_k]}(s)
\mathcal{D}^-_tu(s,
X^\hatc_{t^n_{k-1}\wedge \cdot}(\omega))\Big|_H\leq M_K
\ \mbox{for a.e.\ }s\in(0,T).
\end{equation}
For fixed $s\in (0,T)$, let  $\{k_n\}_{n\in \mathbb{N}}$
be the sequence such that 
 $s\in (t^n_{k_n-1},t^n_{k_n}]$ for all $n\in \mathbb{N}$, $n\geq 1$.
Then
$X^\hatc_{t^n_{k_n-1}\wedge \cdot}(\omega)\rightarrow X^\hatc_{s\wedge \cdot}(\omega)$
in $\mathbb{W}$ as $n\rightarrow \infty$.
Since this holds for all $s\in(0,T)$ and since
 $\mathbb{W}\rightarrow \mathbb{R},\ \mathbf{x} \mapsto \mathcal{D}^-_tu(s,\mathbf{x})$, is continuous
 for a.e.\ $s\in (0,T)$
 because of Assumption~\ref{2017-05-30:11}\emph{(\ref{2017-05-31:10})}, we have
\begin{equation}
  \label{eq:2017-05-31:12}
  \lim_{n\rightarrow \infty}
 \sum_{k=1}^n
\mathbf{1}_{(t^n_{k-1},t^n_k]}(s)
\mathcal{D}^-_tu(s,
X^\hatc_{t^n_{k-1}\wedge \cdot}(\omega))
=
\mathcal{D}^-_tu(s,X^\hatc_{s\wedge \cdot}(\omega))\ \mbox{for a.e.\ }s\in (0,T).
\end{equation}
By
\eqref{2017-05-31:11}
and
\eqref{eq:2017-05-31:12}, we can apply Lebesgue's dominated convergence theorem to
\eqref{2017-05-31:13} evaluated in $\omega$
and obtain
\begin{equation*}
  \lim_{n\rightarrow \infty}\mathbf{II}_n(\omega)
  =
  \int_{\hat t}^T\mathcal{D}^-_tu(s,X^\hatc_{s\wedge \cdot}(\omega))ds.
\end{equation*}
Since $\omega\in \Omega$  was arbitrary, we have
\begin{equation}
  \label{2017-05-31:14}
  \lim_{n\rightarrow \infty}\mathbf{II}_n
  =
  \int_{\hat t}^T\mathcal{D}^-_tu(s,X^\hatc_{s\wedge \cdot})ds\ \mbox{pointwise\ on $\Omega$}.
\end{equation}
This concludes the proof, because, by 
passing to the limit $n\rightarrow \infty$ in
\eqref{2017-05-31:15} and
 considering  \eqref{2017-05-31:16} and \eqref{2017-05-31:14}, we  obtain
 \eqref{2017-05-31:09} with $t=T$.
\end{proof}

\section{Application to path-dependent PDEs}
\label{2017-05-11:07}

In this section we use the path-dependent
It\=o's formula
 to relate
the solution of an $H$-valued path-dependent SDE with
a path-dependent Kolmogorov equation, 
similarly
as
in the
classical non-path-dependent case (see e.g.\ \cite[Ch.\ 7]{DaPrato2004}).
As a corollary, we will derive a Clark-Ocone type formula.

\medskip
The following assumption on $b,\Phi$ will be standing for the remaining of the present section.


\begin{assumption}\label{2017-06-09:05}
$b\in \CNAW{H}$, $\Phi\in\CNAW{L_2(U,H)}$,
and there exists $M>0$ such that
$$
\begin{dcases}
     |b(t,\mathbf{x})-b(t,\mathbf{x}')|_H \leq
   M|\mathbf{x}-\mathbf{x}'|_\infty\\
   |b(t,\mathbf{x})|_H \leq M(1+|\mathbf{x}|_\infty)
 \end{dcases}
\qquad
 \begin{dcases}
        |\Phi(t,\mathbf{x})-\Phi(t,\mathbf{x}')|_{L_2(U,H)} \leq
     M|\mathbf{x}-\mathbf{x}'|_\infty\\
     |\Phi(t,\mathbf{x})|_{L_2(U,H)} \leq M(1+|\mathbf{x}|_\infty)
   \end{dcases}
   $$
for all $t\in[0,T]$, $\mathbf{x},\mathbf{x}'\in \mathbb{W}$.
\end{assumption}

For  $p>2$,  $\hat Y\in \mathcal{L}^p_{\mathcal{P}_T}(\mathbb{W})$,
$ \hat t\in[0,T]$, we consider the following path-dependent SDE
\begin{equation}
  \label{eq:2016-02-09:00}
  \begin{dcases}
    dX_s=b(s,X)ds + \Phi(s,X)dW_s &\quad \forall s\in [\hat t,T]\\
    X_{ \hat t\wedge \cdot}= \hat Y_{ \hat t\wedge \cdot}. &
  \end{dcases}
\end{equation}
By a standard  contraction argument (see e.g.\ \cite[Ch.\ 3]{Gawarecki2011} and \cite[Theorem 3.6]{Cosso2018}),
there exists a unique strong solution $X^\hatc$ to \eqref{eq:2016-02-09:00} in $\mathcal{L}^p_{\mathcal{P}_T}(\mathbb{W})$, 
i.e.\ 
a unique process $X^\hatc\in \mathcal{L}^p_{\mathcal{P}_T}(\mathbb{W})$ such that, for all $t\in[0,T]$,
$$
X^\hatc_t=\hat  Y_{ \hat t\wedge t}+
\int_{\hat t}^{\hat t\vee t}
b(r,X^\hatc)dr
+\int_{\hat t}^{\hat t\vee t}
\Phi(r,X^\hatc)dW_r\qquad \mathbb{P}\mbox{-a.e..}
$$
Moreover, 
the map
\begin{equation}
  \label{eq:2016-02-09:02}
  [0,T]\times \mathcal{L}^p_{\mathcal{P}_T}(\mathbb{W})\rightarrow \mathcal{L}^p_{\mathcal{P}_T}(\mathbb{W}),\ (t, Y) \mapsto X^{ t,Y}
\end{equation}
is Lipschitz continuous with respect to $ Y$, uniformly 
for $ t\in[0,T]$, and jointly continuous in $(t,Y)$.
Uniqueness of solution yields the flow property
\begin{equation}
  \label{eq:2016-02-09:03}
  X^{t,\mathbf{x}}=X^{s,X^{t,\mathbf{x}}}\mbox{ in }\mathcal{L}^p_{\mathcal{P}_T}(\mathbb{W}),\ \forall (t,\mathbf{x})\in [0,T]\times \mathbb{W},\ \forall s\in [t,T].
\end{equation}

\medskip
Let $f\colon \mathbb{W}\rightarrow \mathbb{R}$ be a Lipschitz function.
Hereafter in this section, we denote by $\varphi$ the function
$$
\varphi\colon [0,T]\times \mathbb{W}\rightarrow \mathbb{R}
$$
defined  by
\begin{equation}
  \label{eq:2016-02-09:01}
\varphi(t,\mathbf{x})\coloneqq \mathbb{E} \left[ f(X^{t,\mathbf{x}}) \right] \qquad \forall (t,\mathbf{x})\in [0,T]\times \mathbb{W}.
\end{equation}
Due to the continuity properties of the map \eqref{eq:2016-02-09:02}, 
$\varphi(t,\mathbf{x})$ is Lipschitz 
continuous with respect to  $\mathbf{x}$, uniformly for $t\in[0,T]$, and jointly continuous in $(t,\mathbf{x})$.
It is
 clear that $\varphi(t,\mathbf{x})=\varphi(t,\mathbf{x}_{t\wedge \cdot})$.
Then $\varphi\in \CNAW{\mathbb{R}}$.
Since  $X^{t,\mathbf{x}}$ is independent of $\mathcal{F}_t$,  we can write,
by
\eqref{eq:2016-02-09:03} and  \cite[Lemma~3.9, \mbox{p.\ 55}]{Baldi2000},
\begin{equation}
  \label{eq:2016-02-09:04}
  \begin{split}
    \varphi(t',\mathbf{x})=&
\mathbb{E} \left[ f(X^{t',\mathbf{x}}) \right] 
=
\mathbb{E} \left[  f(X^{t,X_{t\wedge \cdot}^{t',\mathbf{x}}})
 \right] \\
=&
\mathbb{E} \left[ \mathbb{E} \left[ f(X^{t,X_{t\wedge \cdot}^{t',\mathbf{x}}})|\mathcal{F}_{t} \right]  \right] 
=
\mathbb{E} \left[ \varphi(t,X^{t',\mathbf{x}}_{t\wedge \cdot})  \right]
=
\mathbb{E} \left[ \varphi(t,X^{t',\mathbf{x}})  \right]\qquad \forall t\in[t',T].
\end{split}
\end{equation}

In what follows, we will show that,
in case $\varphi(t,\mathbf{x})$ is
sufficiently
regular with respect to the variable $\mathbf{x}$,
then
 Proposition~\ref{2016-01-25:01}
can be used
to conclude that $\mathcal{D}^-_t\varphi$ exists everywhere and that $\varphi$ solves a path-dependent backward Kolmogorov equation associated to SDE \eqref{eq:2016-02-09:00}.
We argue similarly as
 in \cite[Ch.\ 7]{DaPrato2004}, 
where,
differently than in our case,
 the setting is
 non-path-dependent.
The
two main tools of the
 argument   are
\eqref{eq:2016-02-09:04}
and
formula
\eqref{eq:2015-08-29:04}. 

In order to use 
formula \eqref{eq:2015-08-29:04}, we need to make some assumptions regarding existence and regularity of the spatial derivatives of $\varphi$.
In this section, we make such assumptions without any further investigation under which conditions they can be obtained.
We only guess that, at least in the Markovian case, 
i.e.\  when $b$ and $\Phi$ are not path-dependent, and the only path-dependence is due to $f$, the regularity assumptions on $\varphi(t,\cdot)$ should  come 
from continuity assumption on $ \partial f$ and $ \partial ^2f$ with respect to $\sigma^s$, and
 from regularity assumptions on the coeffiecients $b$ and $\Phi$, thanks to the results
in \cite[Ch.\ 7]{DaPrato2004}.
In the following section,  
we will prove that the regularity assumptions 
on the spatial derivatives of $\varphi$
are satisfied
for a particular class of dynamics $X$.

\bigskip
For a function $v(t,\mathbf{x})$,
defined for 
$(t,\mathbf{x})\in [0,T]\times \Bone{H}$,
the more concise notation 
$ \partial _{\mathbb{B}^1}v$
stands for
$ \partial _{\Bone{H}}v$, and
$ \partial^2 _{\mathbb{B}^1}v$ stands for $ \partial ^2_{\Bone{H}}v$.
For a function $v$ such that, for all $t\in (0,T)$,  $v(t,\cdot)$ 
satisfies Assumption~\ref{2015-08-28:03},
we define
$\mathcal{L}v$ by
\begin{equation*}
  \mathcal{L}v(t,\mathbf{x})\coloneqq 
\overline{ \partial _\mathbb{W} v}(t,\mathbf{x}).
(\mathbf{1}_{[t,T]}b(t,\mathbf{x}))
+
\frac{1}{2}\mathbf{T} \left[ \overline{\partial ^2_\mathbb{W}
          v}(t,\mathbf{x}
),\mathbf{1}_{[t,T]}\Phi(t,\mathbf{x})\right]\qquad \forall (t,\mathbf{x})\in (0,T)\times \mathbb{W}.
\end{equation*}


\medskip
\begin{theorem}\label{2016-02-10:10}
Let $\varphi$ be defined by
\eqref{eq:2016-02-09:01}.
If $\varphi$ satisfies
Assumption~\ref{2017-05-30:11}(\ref{2017-05-31:06}),
then $\varphi$
satisfies also
Assumption~\ref{2017-05-30:11}(\ref{2017-05-31:10}) and
\begin{equation}
  \label{eq:2016-02-09:20}
      \mathcal{D}_t^-\varphi(t,\mathbf{x})+
\mathcal{L}\varphi(t,\mathbf{x})
=0
 \qquad \forall (t,\mathbf{x})\in (0,T)\times \mathbb{W}.
    \end{equation}
\end{theorem}
\begin{proof}
    Let $t',t\in(0,T)$, $t'<t$, $\mathbf{x}\in \mathbb{W}$.
By assumption on the spatial derivatives of $\varphi(t,\cdot)$,
 we can apply Proposition~\ref{2016-01-25:01} to $\varphi(t,X^{t',\mathbf{x}}_{t\wedge \cdot})$, and obtain
\begin{equation}
  \label{2016-02-09:05}
  \begin{multlined}[c][.85\displaywidth]
   \mathbb{E} \left[  \varphi(t,X_{t\wedge \cdot}^{t',\mathbf{x}}) \right] = \varphi(t,\mathbf{x}_{t'\wedge \cdot})
    +\int_{t'}^{t} \mathbb{E} \left[       \overline{ \partial_\mathbb{W} \varphi}(t,X^{t',\mathbf{x}}_{s\wedge \cdot}).(\mathbf{1}_{[s,T]}b(s,X^{t',\mathbf{x}})) \right] 
ds\\
+
\frac{1}{2} \int_{t'}^{t}
\mathbb{E} \left[ 
\mathbf{T} \left[ \overline{\partial ^2_\mathbb{W}
          \varphi}(t,X^{t',\mathbf{x}}_{s\wedge
          \cdot}),\mathbf{1}_{[s,T]}\Phi(s,X^{t',\mathbf{x}})\right] \right]
     ds.
   \end{multlined}
\end{equation}
By non-anticipativity, $\varphi(t,X^{t',\mathbf{x}})=\varphi(t,X^{t',\mathbf{x}}_{t\wedge \cdot})$.
Then, by
 \eqref{eq:2016-02-09:04} and
\eqref{2016-02-09:05}, we have
\begin{multline*}
    \varphi(t,\mathbf{x}_{t'\wedge \cdot})-\varphi(t',\mathbf{x})=
-\int_{t'}^{t} \mathbb{E} \left[       \overline{ \partial_{\mathbb{W}} \varphi}(t,X^{t',\mathbf{x}}_{s\wedge \cdot}).(\mathbf{1}_{[s,T]}b(s,X^{t',\mathbf{x}})) \right] 
ds\\
- \frac{1}{2}\int_{t'}^{t}
\mathbb{E} \left[ 
\mathbf{T} \left[ \overline{\partial ^2_\mathbb{W}
          \varphi}(t,X^{t',\mathbf{x}}_{s\wedge
          \cdot}),\mathbf{1}_{[s,T]}\Phi(s,X^{t',\mathbf{x}})\right] \right]
     ds.
\end{multline*}
By  continuity of \eqref{eq:2016-02-09:02},
\begin{equation}
  \label{eq:2016-02-10:05}
  \lim_{t'\rightarrow t^-}\sup_{s\in[t',t]}
|
X^{t',\mathbf{x}}_{s\wedge \cdot}-
\mathbf{x}_{t\wedge \cdot}|_H=0\ \mbox{on\ }\Omega.
\end{equation}
By non-anticipativity and continuity of $b$ and $\Phi$, we then obtain, on $\Omega$,
  \begin{gather*}
      \lim_{t'\rightarrow t^-}
\sup_{s\in[t',t]}|b(s,X^{t',\mathbf{x}})-b(t,\mathbf{x})|_H
=0\\
\lim_{t'\rightarrow t^-}
\sup_{s\in[t',t]}|\Phi(s,X^{t',\mathbf{x}})-\Phi(t,\mathbf{x})|_{L_2(U,H)}
=0.
\end{gather*}
Then, by Proposition~\ref{2015-08-24:00}\emph{(\ref{2016-01-25:08})},
for any sequence $\{(t'_n,s_n)\}_{n\in \mathbb{N}}$ with $t'_n\leq s_n\leq t$ and $t'_n\rightarrow t$, 
 we have
\begin{equation}
  \label{eq:2016-02-10:03}
  \begin{dcases}
    \lim_{n\rightarrow \infty}\mathbf{1}_{[s_n,T]}b(s_n,X^{t'_n,\mathbf{x}})=\mathbf{1}_{[t,T]}b(t,\mathbf{x})\ \mbox{in}\ \Sones{H}\\
    \lim_{n\rightarrow \infty}\mathbf{1}_{[s_n,T]}\Phi(s_n,X^{t'_n,\mathbf{x}})=\mathbf{1}_{[t,T]}\Phi(t,\mathbf{x})\ \mbox{in}\ \Sones{L_2(U,H)}.
  \end{dcases}
\end{equation}
By assumption,  $\overline{\partial_\mathbb{W}
\varphi}(t,\mathbf{x}).\mathbf{v} $ and $\overline { \partial ^2_\mathbb{W}
\varphi}(t,\mathbf{x}).(\mathbf{v},\mathbf{v})$ are uniformly bounded for $\mathbf{x}\in \mathbb{W}$ and $\mathbf{v}\in \Sone{H}$, $|\mathbf{v}|_\infty\leq 1$, and  sequentially continuous in $(\mathbf{x},\mathbf{v})\in \mathbb{W}\times \Sones{H}$.
Then, by 
\eqref{eq:2016-02-10:05}, 
\eqref{eq:2016-02-10:03}, and  Lebesgue's dominated convergence theorem,
we have
\begin{equation}
  \label{eq:2016-02-10:06}
  \lim_{t'\rightarrow t^-}\sup_{s\in[t',t]}
 \mathbb{E} \left[  \left|     \overline{ \partial_\mathbb{W}
 \varphi}(t,X^{t',\mathbf{x}}_{s\wedge \cdot}).(\mathbf{1}_{[s,T]}b(s,X^{t',\mathbf{x}}))-
\overline{\partial _\mathbb{W} \varphi}(t,\mathbf{x}).(\mathbf{1}_{[t,T]}b(t,\mathbf{x}))\right|
 \right] =0
\end{equation}
and
\begin{equation}
  \label{eq:2016-02-10:07}
 \lim_{t'\rightarrow t^-}\sup_{s\in[t',t]}
\mathbb{E} \left[ 
\left|
\mathbf{T} \left[ \overline{\partial ^2_\mathbb{W}
          \varphi}(t,X^{t',\mathbf{x}}_{s\wedge
          \cdot}),\mathbf{1}_{[s,T]}\Phi(s,X^{t',\mathbf{x}})\right]
-
\mathbf{T} \left[ \overline{\partial ^2_\mathbb{W}
          \varphi}(t,\mathbf{x}),\mathbf{1}_{[t,T]}\Phi(t,\mathbf{x})\right]
\right|
 \right]=0
 \end{equation}
Thanks to \eqref{eq:2016-02-10:06} and \eqref{eq:2016-02-10:07}, we can finally write
  \begin{equation*}
\lim_{t'\rightarrow t^-}
\frac{1}{t-t'}
\int_{t'}^{t} \mathbb{E} \left[       \overline{ \partial_\mathbb{W}
 \varphi}(t,X^{t',\mathbf{x}}_{s\wedge \cdot}).(\mathbf{1}_{[s,T]}b(s,X^{t',\mathbf{x}})) \right] 
ds =
\overline{\partial _\mathbb{W}\varphi}(t,\mathbf{x}).(\mathbf{1}_{[t,T]}b(t,\mathbf{x}))
\end{equation*}
and 
\begin{equation*}
  \lim_{t'\rightarrow t^-}
\frac{1}{t-t'}
\int_{t'}^{t}
\mathbb{E} 
\left[ 
\mathbf{T} \left[ \overline{\partial ^2_\mathbb{W}
          \varphi}(t,X^{t',\mathbf{x}}_{s\wedge
          \cdot}),\mathbf{1}_{[s,T]}\Phi(s,X^{t',\mathbf{x}})\right] \right]
     ds
=
\mathbf{T} 
\left[ \overline{\partial ^2_\mathbb{W}
          \varphi}(t,\mathbf{x}),\mathbf{1}_{[t,T]}\Phi(t,\mathbf{x})\right]
\end{equation*}
This proves that $\mathcal{D}^-_t\varphi(t,\mathbf{x})$ exists
and that \eqref{eq:2016-02-09:20} holds true.

We now show that
$\mathcal{D}^-_t\varphi(t,\mathbf{x})$ is continuous in $\mathbf{x}$ and that
$$
\sup_{\substack{t\in (0,T)\\\mathbf{x}\in K}}|\mathcal{D}^-_t\varphi(t,\mathbf{x})|<\infty,
$$
for all compact sets $K\subset \mathbb{W}$.
By \eqref{eq:2016-02-09:20}, it is sufficient to show that $$
\mathbb{W}\rightarrow \mathbb{R},\ \mathbf{x} \mapsto \mathcal{L}\varphi(t,\mathbf{x})
$$
is continuous, for all $t\in (0,T)$, and that
$$
\sup_{\substack{t\in (0,T)\\\mathbf{x}\in K}}|\mathcal{L}\varphi(t,\mathbf{x})|<\infty.
$$
But this is straightforward from the
sublinear growth and continuity assumptions in $\mathbf{x}$ of $b,\Phi$ and from the boundedness and continuity assumption on
 $ \overline{ \partial _\mathbb{W}\varphi}, \overline{ \partial^2 _\mathbb{W}\varphi}$.
\end{proof}

\begin{corollary}\label{2017-06-08:02}
If $\varphi$ satisfies
Assumption~\ref{2017-05-30:11}(\ref{2017-05-31:06}),(\ref{2017-05-31:07}),
then, for all $t\in [\hat t,T]$, we have the following representation:
\begin{equation}\label{2017-06-02:00}
  \varphi(t,X^\hatc)=
\varphi(\hat t,\hat Y)+
\int_{\hat t}^t
\overline{ \partial _\mathbb{W}\varphi}(s,X^\hatc ).(\mathbf{1}_{[s,T]}\Phi_s)dW_s\qquad \mathbb{P}\mbox{-a.e..}
\end{equation}
\end{corollary}
\begin{proof}
  By Theorem~\ref{2016-02-10:10}, the assumptions of Theorem~\ref{2017-05-30:02} are satisfied for $\varphi$.
By applying formula~\eqref{2017-05-31:09}
to $\varphi(t,X^\hatc)$
and
recalling
\eqref{eq:2016-02-09:20},
we obtain
 \eqref{2017-06-02:00}.
\end{proof}

\section{\!\!\!The case  $b(t,\mathbf{x})\!\!=\!\!b(t,\int_{[0,T]}\mathbf{\tilde x}(t-s)\mu(ds))$ and additive noise}\label{2016-02-12:11}

In this section, in a case of interest,
we show
that
 Theorem~\ref{2016-02-10:10}
and Corollary~\ref{2017-06-08:02}
can be applied.

\medskip
The following assumption will be standing for the remaining  of this section.

\begin{assumption}\label{2016-02-11:02}
${}$
  \begin{enumerate}[(i)]
  \item $\mu\in \rad$; 
  \item 
$b\colon [0,T]\times H\rightarrow H$ 
 is continuous and there exists $N>0$ such that
  \begin{align}
    \label{eq:2016-02-11:01}
    \sup_{t\in[0,T]}|b(t,y)|_H&\leq N(1+|y|_H) && \forall y\in H,\\
    \sup_{t\in[0,T]}|b(t,y)-b(t,y')|_H&\leq N|y-y'|_H
&& \forall y,y'\in H.
  \end{align}
\item\label{2016-02-11:15} for all $t\in[0,T]$, $b(t,\cdot)\in \mathcal{G}^2(H,H)$,
  \begin{align}
&    N_1\coloneqq \sup_{\substack{(t,y)\in [0,T]\times H\\ v\in H,\ |v|_H\leq 1}}| \partial_Hb(t,y).v|_H<\infty,\label{2016-02-11:18}
\\
&N_2\coloneqq \sup_{\substack{(t,y)\in [0,T]\times H\\ v,w \in H,\ |v|_H\vee |w|_H\leq 1}}| \partial^2_Hb(t,y).(v,w)|_H<\infty,\label{2016-02-11:20}
\end{align}
and $ \partial _Hb(t,y).v$, $ \partial ^2_Hb(t,y).(v,w)$ are jointly continuous in $t,y,v,w$.
\end{enumerate}
\end{assumption}

We define
$$
\hat b(t,\mathbf{y})\coloneqq b \left( t,
\int_{[0,T]}
\mathbf{\tilde y}(t-s)
\mu(ds) \right) \qquad \forall (t,\mathbf{y})\in[0,T]\times \Bone{H}.
$$
where
\begin{equation}
  \label{eq:2016-02-11:16}
 \mathbf{\tilde y}(r)\coloneqq 
\mathbf{1}_{[-T,0)}(r)\mathbf{y}(0)+\mathbf{1}_{[0,T]}(r)\mathbf{y}(r) \qquad \forall r\in[-T,T].
\end{equation}

Then $\hat b(t,\mathbf{y})$ is
 a 
function of
$t$ and
 the
convolution between $\mu$ and $\mathbf{y}$
computed taking into account the past history of $\mathbf{y}$ on the time window $[t-T,t]$.

\begin{remark}
\label{2016-02-12:01}
  The fact that $b(t,\cdot)\in \mathcal{G}^2(H,H)$,
with differentials uniformly bounded,
 implies that $b(t,\cdot)\in C^1_b(H,H)$, i.e.\ $b(t,\cdot)$ is Fr\'echet differentiable and the Fr\'echet differential $Db(t,\cdot)$ is continuous and bounded (with bound uniform in $t$, due to our assumptions on $b$).
For the proof, see \cite[Proposition 7.4.1]{DaPrato2004}.
\end{remark}

Let again
 $W$ denote a  $U$-valued cylindrical  Wiener process
and let $B\in L_2(U,H)$. 
Consider
 the following SDE:
\begin{equation}\label{2016-02-11:00}
  \begin{dcases}
    dX_s=
\hat b(s,X)ds
+ BdW_s &\quad \forall s\in [\hat t,T]\\
    X_{ \hat t\wedge \cdot}= \hat Y_{ \hat  t\wedge \cdot},&
  \end{dcases}
\end{equation}
for $\hat Y\in \mathcal{L}^p_{\mathcal{P}_T} (\mathbb{W})$, $p>2$.
Notice that Assumption~\ref{2017-06-09:05}
is verified
 with the present coefficients $b$ and $\Phi\equiv B$.
Our aim is to prove
a certain  regularity 
of the solution $X^\hatc$ of \eqref{2016-02-11:00}
 with respect to the initial datum $\hat Y\in \mathcal{L}^p_{\mathcal{P}_T}(\mathbb{W})$, $p>2$, suitable
to apply
Theorem~\ref{2016-02-10:10}
and
Corollary~\ref{2017-06-08:02}.

\begin{remark}
  The choice $\mu=\delta_0$, Dirac measure in $0$, corresponds to the Markovian case $\hat b(s,\mathbf{y})=b(t,\mathbf{y}(t))$.
By choosing $\mu=\delta_{a}$, Dirac measure centered in $a\in(0,T]$, we obtain a drift $\hat b(t,\mathbf{y})=b(t,\mathbf{y}(t-a))$ with a pointwise delay.
\end{remark}

By Assumption  \ref{2016-02-11:02}, we have that
$\hat b(\cdot,\mathbf{y})$ is continuous for all $\mathbf{y}\in \Cb{H}$.
Moreover,
$$
\left|\hat b(t,\mathbf{y}_1)-\hat b(t,\mathbf{y}_2)
\right|_H
\leq
N\int_{[0,T]}\left|\mathbf{\tilde y}_1(t-r)-\mathbf{\tilde y}_2
(t-r)
\right|_H\mu(dr)\qquad
\forall \mathbf{y}_1,\mathbf{y}_2\in \Bone{H}.
$$
Then, if $\{\mathbf{y}_n\}_{n\in \mathbb{N}}\subset \mathbb{W}$ and $\mathbf{y}_n\rightarrow \mathbf{y}$ in $\Bones{H}$,
we have $\hat b(t,\mathbf{y}_n)\rightarrow \hat b(t,\mathbf{y})$ for all $t\in[0,T]$.
Hence $\hat b(\cdot, \mathbf{y})\in \Bone{H}$,
for all $\mathbf{y}\in \Bone{H}$.
In particular, for all $\mathbf{y}\in \Bone{H}$, the indefinite integral
$$
[0,T]\rightarrow H,\ \xi \mapsto \int_t^{t\vee \xi} \hat b(s,\mathbf{y})ds
$$
is continuous.

These considerations entails the well-posedness,
{\underline{for any fixed $\omega\in \Omega$},\label{2017-06-09:00}
 of the map
\begin{equation}
  \label{eq:2016-02-11:03}
  \psi\colon [0,T]\times \Bone{H}\times \Bone{H}\rightarrow \Bone{H}
\end{equation}
defined by
$$
\psi(t,\mathbf{x},\mathbf{y})\coloneqq \mathbf{x}_{t\wedge \cdot}+\int_t^{t\vee \cdot}\hat b(s,\mathbf{y})ds+(W^B_{t\vee \cdot}(\omega)-W^B_t(\omega))\quad \forall (t,\mathbf{x},\mathbf{y})\in [0,T]\times \Bone{H}\times \Bone{H},
$$
where $W^B$ is a short notation for a fixed representant of $\int_0^\cdot BdW_s$.

In the following propositions, we prove existence and uniqueness of a fixed point for $\psi(t,\mathbf{x},\cdot)$ and  study how the fixed point depends on $t,\mathbf{x}$.
We use standard arguments based on contractions in Banach spaces.
What is important is that the SDE is here considered pathwise, in order to have better insight 
about the regularity of the paths $X^{t,\mathbf{x}}(\omega)$ with respect to $\mathbf{x}$.

\begin{remark}\label{2016-04-22:06}
  In the notation $\psi$,
the  dependence on $\omega$ is not explicit.
Nevertheless, we stress the very important fact that all the 
bounds for
the Lipschitz constants and 
 the differentials, which appear in the following propositions, are independent of $\omega$.
More precisely,
the terms $\lambda,\alpha$  appearing in 
Proposition~\ref{2016-02-11:27}\emph{(\ref{2016-02-11:11})},
the bounds for \eqref{2016-04-22:03} and 
\eqref{2016-04-22:04},
the bounds for $ \partial _{\mathbb{B}^1}\Lambda^{t,\cdot}$ and $ \partial _{\mathbb{B}^1}^2\Lambda^{t,\cdot}$ in 
Proposition~\ref{2016-02-11:37}, 
can be --- and we assume that they are --- 
chosen
independently of $\omega$.
\end{remark}

For $\lambda>0$, we introduce on $\Bone{H}$ the norm
$$
|\mathbf{x}|_\lambda\coloneqq \sup_{t\in [0,T]}e^{-\lambda t}|\mathbf{x}(t)|_H,\qquad \forall \mathbf{x}\in \Bone{H}.
$$
Then $|\cdot|_\lambda$ is equivalent to $|\cdot|_\infty$.

Hereafter,  we denote by $\Boninf{H}$ the Banach space $(\Bone{H},|\cdot|_\infty)$ and by $\Bonl{H}$ the equivalent Banach space $(\Bone{H},|\cdot|_\lambda)$.

\begin{proposition}\label{2016-02-11:27}
${}$
  \begin{enumerate}[(i)]
  \item\label{2016-02-11:11}   There exists $\lambda>0$ and $\alpha\in(0,1)$ such that
  \begin{equation}
    \label{eq:2016-02-11:10}
   \sup_{(t,\mathbf{x})\in[0,T]\times \Bone{H}} |\psi(t,\mathbf{x},\mathbf{y})
-    \psi(t,\mathbf{x},\mathbf{y}')|_\lambda\leq \alpha
|\mathbf{y}-\mathbf{y}'|_\lambda\qquad \forall \mathbf{y},\mathbf{y}\in \Bone{H}.
  \end{equation}
\item\label{2016-02-11:12} The restriction of $\psi$ to $[0,T]\times \mathbb{W}\times \Boninf{H}$ is $\mathbb{W}$-valued and continuous.
\item\label{2016-02-11:13} For all $t\in[0,T]$, the section
$$
\psi(t,\cdot,\cdot)\colon \Boninf{H}\times \Boninf{H}\rightarrow  \Boninf{H},\ (\mathbf{x},\mathbf{y}) \mapsto \psi(t,\mathbf{x},\mathbf{y})
$$ 
is strongly continuously G\^ateaux differentiable up to order $2$, i.e.\
$$
\psi(t,\cdot,\cdot)\in \mathcal{G}^2(\Boninf{H}\times \Boninf{H},\Boninf{H}).
$$
Moreover,  
\begin{equation}\label{2016-04-22:03}
   \sup_{\substack{t\in[0,T],\ \mathbf{x},\mathbf{y}\in \Boninf{H}\\\mathbf{v}\in \Boninf{H},\ |\mathbf{v}|_\infty\leq 1}}| \partial _2\psi(t,\mathbf{x},\mathbf{y}).\mathbf{v}|_\infty+
  \sup_{\substack{t\in[0,T],\ \mathbf{x},\mathbf{y}\in \Boninf{H}\\\mathbf{v}\in \Boninf{H},\ |\mathbf{v}|_\infty\leq 1}}| \partial _3\psi(t,\mathbf{x},\mathbf{y}).\mathbf{v}|_\infty<\infty
\end{equation}
\begin{equation}\label{2016-04-22:04}
\hskip-2cm    \sup_{\substack{t\in[0,T],\ \mathbf{x},\mathbf{y}\in \Boninf{H}\\\mathbf{v},\mathbf{w}\in \Boninf{H},\ |\mathbf{v}|_\infty\vee |\mathbf{w}|_\infty\leq 1}}| \partial^2 _3\psi(t,\mathbf{x},\mathbf{y}).(\mathbf{v},\mathbf{w})|_\infty<\infty,
  \end{equation}
where $ \partial _i\psi$ and $ \partial ^2_i\psi$ denote the
first- and second-order G\^ateaux differential of $\psi$ with respect to the
$i$-th variable.
\item\label{2016-02-11:44}
If $t_n\rightarrow t$ in $[0,T]$,
 $\mathbf{x}_n\rightarrow \mathbf{x}$ in $\Boninf{H}$,
 $\mathbf{y}_n\rightarrow \mathbf{y}$ in $\Bones{H}$, $\mathbf{v}_n\rightarrow \mathbf{v}$ in $\Bones{H}$,
$\mathbf{w}_n\rightarrow \mathbf{w}$ in $\Bones{H}$, then 
\begin{gather}
   \partial _3\psi(t_n,\mathbf{x}_n,\mathbf{y}_n).\mathbf{v}_n
\rightarrow
 \partial _3\psi(t,\mathbf{x},\mathbf{y}).\mathbf{v}\mbox{ in }\Boninf{H}\label{2016-02-12:00}\\[2pt]
   \partial^2 _3\psi(t_n,\mathbf{x}_n,\mathbf{y}_n).(\mathbf{v}_n,\mathbf{w}_n)
\rightarrow
 \partial^2 _3\psi(t,\mathbf{x},\mathbf{y}).(\mathbf{v},\mathbf{w})\mbox{ in }\Boninf{H}.\label{2017-06-09:04}
\end{gather}
\end{enumerate}
\end{proposition}
\begin{proof}
  \emph{(\ref{2016-02-11:11})}
For $t\in [0,T]$ and $\mathbf{x}\in \Bone{H}$, 
by standard computations,
we have
\begin{equation*}
  \begin{split}
    e^{-\lambda s}
|\psi(t,\mathbf{x},\mathbf{y})(s)
-
\psi(t,\mathbf{x},\mathbf{y}')(s)|_H
&\leq
e^{-\lambda s}\int_0^s|\hat b(r,\mathbf{y})-\hat b(r,\mathbf{y}')|_Hdr\\
&\leq 
\int_0^s
e^{-\lambda(s-r)}
e^{-\lambda r}|\hat b(r,\mathbf{y})-\hat b(r,\mathbf{y}')|_Hdr\\
&\leq \frac{1-e^{-\lambda T}}{\lambda}|\hat b(\cdot,\mathbf{y})-\hat b(\cdot,\mathbf{y}')|_\lambda\\
&\leq \frac{1-e^{-\lambda T}}{\lambda}
N|\mu|_1|\mathbf{y}-\mathbf{y}'|_\lambda,
\end{split}
\end{equation*}
for all $\mathbf{y},\mathbf{y}'\in \Bone{H}$ and all $s\in[0,T]$.
Then, for all  $t,\mathbf{x},\mathbf{y},\mathbf{y}$,
\begin{equation}
  \label{eq:2016-02-11:14}
|\psi(t,\mathbf{x},\mathbf{y})
-
\psi(t,\mathbf{x},\mathbf{y}')|_\lambda
\leq \frac{1-e^{-\lambda T}}{\lambda}
N|\mu|_1|\mathbf{y}-\mathbf{y}'|_\lambda.
\end{equation}
By defining $\alpha\coloneqq \frac{1-e^{-\lambda T}}{\lambda}
N|\mu|_1$, for $\lambda$ sufficiently large we obtain
\emph{(\ref{2016-02-11:11})}.

 \emph{(\ref{2016-02-11:12})}
Due to \emph{(\ref{2016-02-11:11})}, it is 
sufficient
 to prove that $\psi(\cdot,\cdot,\mathbf{y})$ is $\mathbb{W}$-valued and continuous on $[0,T]\times \mathbb{W}$, for
all
 $\mathbf{y}\in \Bone{H}$.
But this comes from the continuity of the maps
$$
[0,T]\times \mathbb{W}\rightarrow \mathbb{W},\ 
(t,\mathbf{x}) \mapsto \mathbf{x}_{t\wedge \cdot}
\qquad\qquad
[0,T]\rightarrow H,\ s \mapsto \int_0^s\hat b(r,\mathbf{y})dr +W^B_s(\omega).
$$

  \emph{(\ref{2016-02-11:13})}+\emph{(\ref{2016-02-11:44})}
We begin by showing that, for all $t\in[0,T]$,
\begin{equation}
  \label{eq:2016-02-11:38}
  \Psi_t\colon \Boninf{H}\rightarrow \Boninf{H},\ \mathbf{y} \mapsto \int_t^{t\vee \cdot}\hat b(r,\mathbf{y})dr
\end{equation}
is strongly continuously G\^ateaux differentiable up to order $2$, with bounded differentials (bound uniform in $t$).
By standard computations, due to Assumption 
\ref{2016-02-11:02}\emph{(\ref{2016-02-11:15})}, we have
\begin{equation*}
  \begin{split}
&  h^{-1}(  \Psi_t(\mathbf{y}+h\mathbf{v})-\Psi_t(\mathbf{y}))\\
&=
\int_t^{t\vee \cdot} \left( \int_0^1
\langle \nabla _H
b\left(r,\int_{[0,T]}\mathbf{\tilde y}(r-s)\mu(ds)
+\theta h
\int_{[0,T]}\mathbf{\tilde v}(r-s)\mu(ds)
 \right),
 \int_{[0,T]}\mathbf{\tilde v}(r-s)\mu(ds)  \rangle _H
d\theta
 \right)dr,
  \end{split}
\end{equation*}
where $\nabla_Hb$ represents $ \partial _Hb$ in $H$.
Due to the assumptions on $ \partial_Hb$, we can pass to the limit
$
h^{-1}(  \Psi_t(\mathbf{y}+h\mathbf{v})-\Psi_t(\mathbf{y}))$
 in $\Boninf{H}$ as $h\rightarrow 0$ and obtain
\begin{equation}
  \label{eq:2016-02-11:17}
   \partial \Psi_t(\mathbf{y}).\mathbf{v}=\int_t^{t\vee \cdot}
\langle  \nabla_Hb \left( r,
\int_{[0,T]}\mathbf{\tilde y}(r-s)
\mu(ds)
 \right),
\int_{[0,T]}\mathbf{\tilde v}(r-s)
\mu(ds) 
\rangle_H
 dr.
\end{equation}
Notice that, if $\mathbf{y}_n\rightarrow \mathbf{y}$ and
$\mathbf{v}_n\rightarrow \mathbf{v}$ in $\Bones{H}$, then
, by 
Proposition~\ref{2015-08-24:00}\emph{(\ref{2016-01-25:08})},
$$
\int_{[0,T]} \mathbf{\tilde y}_n(r-s)\mu(ds)
\rightarrow \int_{[0,T]} \mathbf{\tilde y}(r-u)\mu(du) \qquad \forall r\in[0,T]
$$
and the family
$$
 \left\{ \int_{[0,T]} \mathbf{\tilde y}_n(r-u)\mu(du) \right\} _{\substack{r\in [0,T]\\ n\in \mathbb{N}}}
$$
is bounded in $H$.
The same holds with respect to $\mathbf{\tilde v}_n$ and $\mathbf{\tilde v}$.
If $t_n\rightarrow t$ in $[0,T]$, by strong continuity of $ \partial _Hb$
and using
 \eqref{eq:2016-02-11:17},
 we conclude that
\begin{equation}
  \label{eq:2016-04-22:05}
  | \partial \Psi_t(\mathbf{y}).\mathbf{v}-
 \partial \Psi_{t_n}(\mathbf{y}_n).\mathbf{v}_n|_\infty\rightarrow 0.
\end{equation}
This proves $\eqref{2016-02-12:00}$, because $ \partial _3
\psi(t,\mathbf{x},\mathbf{y})= \partial \Psi_t(\mathbf{y})$ for all $(t,\mathbf{x},\mathbf{y})\in[0,T]\times \Bone{H}\times \Bone{H}$.
In particular, the limit  \eqref{eq:2016-04-22:05} holds when $t_n=t$, for all $n\in \mathbb{N}$, and the convergences
$\mathbf{y}_n\rightarrow \mathbf{y}$ and $\mathbf{v}_n\rightarrow \mathbf{v}$ take place in $\Boninf{H}$.
This shows that $\Psi_t\in \mathcal{G}^1( \Boninf{H},\Boninf{H})$, and, by
\eqref{2016-02-11:18}
and
 \eqref{eq:2016-02-11:17},
that the first order differentials are bounded, with bound
uniform in $t$.
By observing that $ \partial _2\psi(t,\mathbf{x},\mathbf{y}).\mathbf{v}=\mathbf{v}_{t\wedge \cdot}$ for all $t\in[0,T]$, $\mathbf{x},\mathbf{y},\mathbf{v}\in \Bone{H}$,
we have then proved that $\psi (t,\cdot,\cdot)\in \Gatot{\Boninf{H}\times \Boninf{H}}{\Boninf{H}}{1}$ and that \eqref{2016-04-22:03} holds true.

Regarding the second order derivative, 
by using similar
 arguments as above,
we obtain
\begin{equation}
  \label{eq:2016-02-11:19}
  \begin{split}
\hskip-5pt       \partial ^2&\Psi_t(\mathbf{x}).(\mathbf{v},\mathbf{w})
=\\
&=\int_t^{t\vee \cdot}
  \partial^2_Hb \left( r,
\int_{[0,T]}\mathbf{\tilde y}(r-s)
\mu(ds)
 \right).
 \left(
\int_{[0,T]}\mathbf{\tilde v}(r-s)\mu(ds)
,
\int_{[0,T]}\mathbf{\tilde w}(r-s)\mu(ds)
  \right) dr
\end{split}
\end{equation}
and the continuity of
$$
[0,T]\times \Bones{H}\times \Bones{H}\times \Bones{H}\rightarrow \Boninf{H},\ (t,\mathbf{y},\mathbf{v},\mathbf{w}) \mapsto  \partial ^2\Psi_t(\mathbf{y}).(\mathbf{v},\mathbf{w}).
$$
Then,
since $ \partial ^2_2\psi(t,\mathbf{x},\mathbf{y})=0$ and $ \partial ^2_3\psi(t,\mathbf{x},\mathbf{y})= \partial ^2\Psi_t(\mathbf{y})$,
 $ \psi(t,\cdot,\cdot)\in \mathcal{G}^2(\Boninf{H}\times \Boninf{H},\Boninf{H})$.
By \eqref{2016-02-11:20}, also the second order differentials
 $ \partial ^2_2\psi, \partial ^2_3\psi$
 are bounded, with bound uniform in $t$.
\end{proof}

In the following proposition we see how the
regularity properties of $\psi$ are inherited by the associated fixed-point map.

\begin{proposition}\label{2016-02-11:37}
${}$
  \begin{enumerate}[(i)]
  \item \label{2016-02-11:22}
  For all $(t,\mathbf{x})\in[0,T]\times \Bone{H}$, there exists a unique $\Lambda^{t,\mathbf{x}}\in \Bone{H}$ such that
$$
\Lambda^{t,\mathbf{x}}=\psi(t,\mathbf{x},\Lambda^{t,\mathbf{x}}).
$$
\item\label{2016-02-11:23} The map
$$
\Lambda\colon [0,T]\times \Boninf{H}\rightarrow \Boninf{H},\ (t,\mathbf{x}) \mapsto \Lambda^{t,\mathbf{x}}
$$
is Lipschitz in $\mathbf{x}$, with 
a bound for the
Lipschitz constant 
independent of $t$.
\item\label{2016-02-11:24} The restriction of $\Lambda$ to $[0,T]\times \mathbb{W}$ is continuous and $\mathbb{W}$-valued.
\item\label{2016-02-11:25} For all $t\in[0,T]$, $\Lambda^{t,\cdot}\in \mathcal{G}^2(\Boninf{H},\Boninf{H})$ and $ \partial_{\mathbb{B}^1}\Lambda^{t,\cdot}$, $ \partial ^2_{\mathbb{B}^1}\Lambda^{t,\cdot}$ are uniformly bounded, uniformly in $t$.
\item\label{2016-02-11:32} For all $t\in[0,T]$ and $\mathbf{x}\in \Bone{H}$,
 $I- \partial _3\psi(t,\mathbf{x},\Lambda^{t,\mathbf{x}}) \in L(\Boninf{H})$ is invertible and
 \begin{equation}
   \label{eq:2016-02-11:34}
   \Boninf{H}\rightarrow L(\Boninf{H}),\ \mathbf{x} \mapsto (I- \partial _3\psi(t,\mathbf{x},\Lambda^{t,\mathbf{x}}))^{-1}
 \end{equation}
 is strongly continuous.
\item\label{2016-02-11:33} 
For all $t\in[0,T]$, $\mathbf{x},\mathbf{v},\mathbf{w}\in \Bone{H}$,
we have
  \begin{gather*}
   \partial _{\mathbb{B}^1}\Lambda^{t,\mathbf{x}}.\mathbf{v}
=
\left(I- \partial _3\psi(t,\mathbf{x},\Lambda^{t,\mathbf{x}}) \right) ^{-1}  \left( \partial _2\psi(t,\mathbf{x},\Lambda^{t,\mathbf{x}}).\mathbf{v} \right)  \\[4pt]
   \partial ^2_{\mathbb{B}^1}
\Lambda^{t,\mathbf{x}}.(\mathbf{v},\mathbf{w})=
\left(I- \partial _3\psi(t,\mathbf{x},\Lambda^{t,\mathbf{x}}) \right) ^{-1} 
 \left( 
 \partial ^2_3
\psi(t,\mathbf{x},\Lambda^{t,\mathbf{x}}).
 \left( 
 (\partial_{\mathbb{B}^1}\Lambda^{t,\mathbf{x}}.\mathbf{v}),
(\partial_{\mathbb{B}^1}\Lambda^{t,\mathbf{x}}.\mathbf{w})
 \right) 
 \right)   
\end{gather*}
\end{enumerate}
\end{proposition}
\begin{proof}
By Proposition \ref{2016-02-11:27}\emph{(\ref{2016-02-11:11})}, we can choose $\lambda>0$ such that $\psi(t,\mathbf{x},\cdot)$ is an $\alpha$-contraction on $\Bonl{H}$, with $\alpha\in(0,1)$, uniformly in $(t,\mathbf{x})\in [0,T]\times \Bone{H}$.

 \emph{(\ref{2016-02-11:22})}
Apply Banach's contraction principle to $\psi(t,\mathbf{x},\cdot)$ on $\Bonl{H}$.

\emph{(\ref{2016-02-11:23})}
For every $t\in[0,T]$, we have
$$
|\psi(t,\mathbf{x},\mathbf{y})
-
\psi(t,\mathbf{x}',\mathbf{y})|_\lambda\leq 
|\mathbf{x}-\mathbf{x}'|_\lambda\qquad \forall \mathbf{x},\mathbf{x}'\in \Bone{H}.
$$
The conclusion 
follows by 
\cite[p.\ 13, inequality ($***$)]{Granas2003}.

\emph{(\ref{2016-02-11:24})}
Since $\psi$ maps $[0,T]\times \mathbb{W}\times \mathbb{W}$ into
 $\mathbb{W}$
by
Proposition~\ref{2016-02-11:27}\emph{(\ref{2016-02-11:12})},
we also have that $\Lambda$ maps $[0,T]\times \mathbb{W}$ into $\mathbb{W}$.
Let us denote by $\Lambda_\mathbb{W}$ the map
$$
\Lambda_\mathbb{W}\colon [0,T]\times \mathbb{W}\rightarrow \mathbb{W},\ (t,\mathbf{x}) \mapsto \Lambda^{t,\mathbf{x}}.
$$
By
\emph{(\ref{2016-02-11:23})}, to prove the continuity 
of $\Lambda_\mathbb{W}$, it is sufficient to show the continuity of $\Lambda^{\cdot,\mathbf{x}}$, for fixed $\mathbf{x}\in \mathbb{W}$.
Let $t_n\rightarrow t$ in $[0,T]$.
We have
$$
\psi(t_n,\mathbf{x},\mathbf{y})\rightarrow \psi(t,\mathbf{x},\mathbf{y})\mbox{ in }\mathbb{W},\ \forall \mathbf{x}, \mathbf{y}\in \mathbb{W}.
$$
Then the conclusion follows by 
\cite[Theorem 7.1.5]{DaPrato2004}.

\emph{(\ref{2016-02-11:25})}+\emph{(\ref{2016-02-11:32})}+\emph{(\ref{2016-02-11:33})}
Thanks to Proposition \ref{2016-02-11:27}\emph{(\ref{2016-02-11:13})},
 we can apply
\cite[Theorems 7.1.2 and 7.1.3]{DaPrato2004}
 to all maps $\psi(t,\cdot,\cdot)$, for all $t\in[0,T]$.
This shows
\emph{(\ref{2016-02-11:25})}
and \emph{(\ref{2016-02-11:33})}.

It remains only to comment the strong continuity of
\eqref{eq:2016-02-11:34} (which is indeed contained in the proof of \cite[Theorems 7.1.2 and 7.1.3]{DaPrato2004}).
This comes from the fact that, for all $t\in[0,T]$,
by
\emph{(\ref{2016-02-11:23})} and 
 Proposition \ref{2016-02-11:27}\emph{(\ref{2016-02-11:13})}, the map
$$
\Bonl{H}
\rightarrow
L(\Bonl{H}),\ \mathbf{x} \mapsto  \partial _3\psi(t,\mathbf{x},\Lambda^{t,\mathbf{x}})
$$
is strongly continuous and
 $| \partial _3\psi(t,\mathbf{x},\Lambda^{t,\mathbf{x}})|_{L(\Bonl{H})}\leq \alpha$ for all $\mathbf{x}\in \Bone{H}$.
By writing
\begin{equation}
  \label{eq:2017-06-04:01}
  (I-\partial _3\psi(t,\mathbf{x},\Lambda^{t,\mathbf{x}}))^{-1}\mathbf{v}=\sum_{n\in \mathbb{N}}
(\partial _3\psi(t,\mathbf{x},\Lambda^{t,\mathbf{x}}))^n\mathbf{v}\qquad \forall \mathbf{v}\in \Bone{H}
\end{equation}
and
by  Lebesgue's dominated convergence theorem (for sums), we see
the strong continuity of \\
$$
\displaystyle{\Bonl{H}
\rightarrow
L(\Bonl{H}),\ \mathbf{x} \mapsto (I- \partial _3\psi(t,\mathbf{x},\Lambda^{t,\mathbf{x}}))^{-1}}.
\eqno\qed
$$
\let\qed\relax
\end{proof}

\medskip
The following proposition provides the good continuity of the differentials of $\Lambda$ with respect to $\mathbf{x}$, that we will later need in order to apply Theorem
\ref{2016-02-10:10}
and Corollary~\ref{2017-06-08:02}
 when the process $X$ has the dynamics 
\eqref{2016-02-11:00}.

\begin{proposition}\label{2016-02-12:02}
Let  $t\in [0,T]$.
\begin{enumerate}[(i)]
\item\label{2017-06-03:08} If $\mathbf{x}_n\rightarrow \mathbf{x}$ in $\Boninf{H}$, $\mathbf{v}_n\rightarrow \mathbf{v}$ in $\Bones{H}$, and
  $\mathbf{w}_n\rightarrow \mathbf{w}$ in $\Bones{H}$, then
\begin{align}
     \partial _{\mathbb{B}^1}\Lambda^{t,\mathbf{x}_{n}}.\mathbf{v}_{n}
&\rightarrow
 \partial _{\mathbb{B}^1}\Lambda^{t,\mathbf{x}}.\mathbf{v}\mbox{ in }\Bones{H}\label{2016-02-11:36}\\[2pt]
   \partial^2 _{\mathbb{B}^1}\Lambda^{t,\mathbf{x}_{n}}.(\mathbf{v}_{n},\mathbf{w}_{n})
&\rightarrow
 \partial^2 _{\mathbb{B}^1}\Lambda^{t,\mathbf{x}}.(\mathbf{v},\mathbf{w})\mbox{ in }\Boninf{H}.
\label{2016-02-11:35} 
  \end{align}
\item\label{2017-06-03:07}
 If $t_n\rightarrow t^+$ in $[0,T]$, $\mathbf{x}\in \mathbb{W}$, $\mathbf{v},\mathbf{w}\in \Bone{H}$, then
  \begin{align}
   \partial _{\mathbb{B}^1}  \Lambda^{t_n,\mathbf{x}}.\mathbf{v}&\rightarrow
 \partial _{\mathbb{B}^1}  \Lambda^{t,\mathbf{x}}.\mathbf{v}\ \mbox{in\ }\Bones{H}\label{2017-06-09:02}\\[2pt]
   \partial^2 _{\mathbb{B}^1}  \Lambda^{t_n,\mathbf{x}}.(\mathbf{v},\mathbf{w})
&\rightarrow
 \partial^2 _{\mathbb{B}^1}  \Lambda^{t,\mathbf{x}}.
(\mathbf{v},\mathbf{w})
\ \mbox{in\ }\mathbb{W}.\label{2017-06-09:03}
  \end{align}
\end{enumerate}
\end{proposition}
\begin{proof}
\emph{(\ref{2017-06-03:08})} Let $t\in[0,T]$, $\mathbf{x}_n\rightarrow \mathbf{x}$ in $\Boninf{H}$,
$\mathbf{v}_n\rightarrow \mathbf{v}$ in $\Bones{H}$.
By
Proposition~\ref{2015-08-24:00}\emph{(\ref{2016-01-25:08})}, $\{\mathbf{v}_n\}_{n\in \mathbb{N}}$ is bounded in $\Boninf{H}$.
By Proposition~\ref{2016-02-11:37}\emph{(\ref{2016-02-11:25})}, 
$
\left\{  \partial _{\mathbb{B}^1}\Lambda^{t,\mathbf{x}_n}.\mathbf{v}_n   \right\} _{n\in \mathbb{N},t\in[0,T]}
$
is bounded in $\Boninf{H}$.
In particular, $\left\{  \partial _{\mathbb{B}^1}
\Lambda^{t,\mathbf{x}_n}.\mathbf{v}_n   \right\} _{n\in \mathbb{N}}$ is bounded in the Hilbert space $L^2([0,T],H)$, which is separable. 
Then we can find subsequences $\{\mathbf{x}_{n_k}\}_{k\in \mathbb{N}}$ and 
$\{\mathbf{v}_{n_k}\}_{k\in \mathbb{N}}$ 
such that
$$
\lim_{k\rightarrow \infty} \partial _{\mathbb{B}^1}
\Lambda^{t,\mathbf{x}_{n_k}}.\mathbf{v}_{n_k}   
= Z 
\mbox{ weakly in $L^2([0,T],H)$,}
$$
for some $Z\in L^2([0,T],H)$.
We recall that $ \partial _3\psi(t,\mathbf{ x'},\mathbf{y'})= \partial \Psi_t(\mathbf{y'})$
for all $\mathbf{x'}, \mathbf{y'}\in \Bone{H}$,
 where $\Psi_t$ 
was defined in the proof of 
Proposition \ref{2016-02-11:27}
by \eqref{eq:2016-02-11:38}.
By \eqref{eq:2016-02-11:17}, for $\mathbf{ y'},\mathbf{ v'}\in \Bone{H}$, we have
\begin{equation*}
  \begin{split}
       (\partial \Psi_t(\mathbf{y'}).\mathbf{ v'})(\xi)&=
\int_t^{t\vee \xi}
\big\langle \nabla_yb \left( r,
\int_{[0,T]}\mathbf{\tilde y'}(r-u)
\mu(du)
 \right),
\int_{[0,T]}\mathbf{\tilde v'}(r-u)
\mu(du) \big\rangle_H  dr\\
&=\int_t^{t\vee \xi}
 \left( \int_{[0,T]}
\big\langle \nabla_yb \left( r,
\int_{[0,T]}\mathbf{\tilde y'}(r-u)
\mu(du)
 \right),
\mathbf{\tilde v'}(r-s)
 \big\rangle_H \mu(ds)\right) dr\\
&=
 \int_{[0,T]}
 \left(\int_t^{t\vee \xi}
\big\langle \nabla_yb \left( r,
\int_{[0,T]}\mathbf{\tilde y'}(r-u)
\mu(du)
 \right),
\mathbf{\tilde v'}(r-s)
 \big\rangle_H dr\right)
\mu(ds).
\end{split}
\end{equation*}
By replacing
$\mathbf{x'}$ by $\mathbf{x}_{n_k}$,
 $\mathbf{y'}$ by $\Lambda^{t,\mathbf{x}_{n_k}}$, and $\mathbf{v'}$ by $ \partial _{\mathbb{B}^1}\Lambda^{t,\mathbf{x}_{n_k}}.\mathbf{v}_{n_k}$,
 we obtain (\footnote{If the argument $\mathbf{y}$ of the notation $\mathbf{\tilde y}$ is long, we write $(\mathbf{y})^\sim$.})
  \begin{align}
    &
\left(   \partial _3\psi(t,\mathbf{x}_{n_k},\Lambda^{t,\mathbf{x}_{n_k}}).
 \left(  \partial _{\mathbb{B}^1}
\Lambda^{t,\mathbf{x}_{n_k}}.\mathbf{v}_{n_k} \right) \right)(\xi)\label{2016-02-11:42}  \\
&\quad\quad=
\int_{[0,T]}
 \left( 
\int_t^{t\vee \xi}
\big\langle\nabla_y b
 \left( r,
\int_{[0,T]} (\Lambda^{t,\mathbf{x}_{n_k}})^\sim (r-u)\mu(du)
 \right),
(\partial _{\mathbb{B}^1}\Lambda^{t,\mathbf{x}_{n_k}}.\mathbf{v}_{n_k})^\sim(r-s)
\big\rangle_H dr 
 \right) \mu(ds).\notag
\end{align}
Due to the fact that 
$\{\partial _{\mathbb{B}^1}\Lambda^{t,\mathbf{x}_{n_k}}.\mathbf{v}_{n_k}\}_{k\in \mathbb{N}}$ is uniformly bounded in $\Boninf{H}$, passing to another subsequence if necessary, we can assume that
$(\partial _{\mathbb{B}^1}\Lambda^{t,\mathbf{x}_{n_k}}.\mathbf{v}_{n_k})(0)$ is weakly convergent in $H$ to  some $z_0\in H$.
Then
\begin{equation}
  \label{eq:2016-02-11:39}
\lim_{k\rightarrow \infty}(\partial _{\mathbb{B}^1}
\Lambda^{t,\mathbf{x}_{n_k}}.\mathbf{v}_{n_k})^\sim=
\mathbf{1}_{[-T,0)}(\cdot)z_0+\mathbf{1}_{[0,T]}Z\ \mbox{weakly  in }\ L^2([-T,T],H).
\end{equation}
By Proposition \ref{2016-02-11:37}\emph{(\ref{2016-02-11:23})}, we  have
$$
\Lambda^{t,\mathbf{x}_{n_k}}\rightarrow \Lambda^{t,\mathbf{x}}\mbox{ in }\Boninf{H},\qquad
$$
then, 
since $b(r,\cdot)\in C^1_b(H,H)$
(see Remark~\ref{2016-02-12:01}), 
\begin{equation*}
\lim_{k\rightarrow \infty}  \left|\nabla_y b
 \left( r,
\int_{[0,T]} (\Lambda^{t,\mathbf{x}_{n_k}})^\sim (r-u)\mu(du)
 \right)
-
\nabla_y b
 \left( r,
\int_{[0,T]} (\Lambda^{t,\mathbf{x}})^\sim (r-u)\mu(du)
 \right)\right|_H=0,
\end{equation*}
for all $r\in[0,T]$.
In particular, by Lebesgue's dominated convergence theorem,
\begin{equation}
  \label{eq:2016-02-11:40}
\nabla_y b
 \left( \cdot,
\int_{[0,T]} (\Lambda^{t,\mathbf{x}_{n_k}})^\sim (\cdot-u)\mu(du)
 \right)
\rightarrow
\nabla_y b
 \left( \cdot,
\int_{[0,T]} (\Lambda^{t,\mathbf{x}})^\sim (\cdot-u)\mu(du)
 \right)
\end{equation}
strongly in $L^2([0,T],H)$.
By
\eqref{eq:2016-02-11:39} and
\eqref{eq:2016-02-11:40}, we have, for all $s,\xi\in[0,T]$,
\begin{equation*}
    \begin{split}
      \lim_{k\rightarrow \infty}&
\int_t^{t\vee \xi}
\big\langle\nabla_y b
 \left( r,
\int_{[0,T]} (\Lambda^{t,\mathbf{x}_{n_k}})^\sim (r-u)\mu(du)
 \right),
(\partial _{\mathbb{B}^1}\Lambda(t,\mathbf{x}_{n_k}).\mathbf{v}_{n_k})^\sim(r-s)
\big\rangle_H dr =\\
&=
\int_t^{t\vee \xi}
\big\langle\nabla_y b
 \left( r,
\int_{[0,T]} (\Lambda^{t,\mathbf{x}})^\sim (r-u)\mu(du)
 \right),
\mathbf{1}_{[-T,0)}(r-s)z_0+\mathbf{1}_{[0,T]}(r-s)Z\big\rangle_H dr.
\end{split}
\end{equation*}
Since the latter limit holds for all $\xi$ and $s$,
by \eqref{2016-02-11:42} we have that the limit
$$
\lim_{k\rightarrow \infty}\left(   \partial _3\psi(t,\mathbf{x}_{n_k},\Lambda^{t,\mathbf{x}_{n_k}}).
 \left(  \partial _{\mathbb{B}^1}\Lambda^{t,\mathbf{x}_{n_k}}.\mathbf{v}_{n_k} \right) \right)(\xi)
$$
exists for all $\xi\in[0,T]$.
By Proposition~\ref{2015-08-24:00}\emph{(\ref{2016-01-25:08})},
since the sequence is uniformly bounded,
we can finally conclude that
$$
 \left\{    \partial _3\psi(t,\mathbf{x}_{n_k},\Lambda^{t,\mathbf{x}_{n_k}}).
 \left(  \partial _{\mathbb{B}^1}\Lambda^{t,\mathbf{x}_{n_k}}.\mathbf{v}_{n_k} \right)  \right\} _{k\in \mathbb{N}}
$$
converges in $\Bones{H}$.
Now we are almost done.
By
Proposition \ref{2016-02-11:37}\emph{(\ref{2016-02-11:33})},
we have
\begin{equation}\label{2016-02-11:43}
  \begin{split}
       \partial _{\mathbb{B}^1}\Lambda^{t,\mathbf{x}_{n_k}}.\mathbf{v}_{n_k}&=
 \partial _3\psi(t,\mathbf{x}_{n_k},\Lambda^{t,\mathbf{x}_{n_k}}).
 \left(  \partial _{\mathbb{B}^1}\Lambda^{t,\mathbf{x}_{n_k}}.\mathbf{v}_{n_k} \right) +
 \partial _2\psi(t,\mathbf{x}_{n_k},\Lambda^{t,\mathbf{x}_{n_k}}).\mathbf{v}_{n_k}\\
&=\partial _3\psi(t,\mathbf{x}_{n_k},\Lambda^{t,\mathbf{x}_{n_k}}).
 \left(  \partial _{\mathbb{B}^1}\Lambda^{t,\mathbf{x}_{n_k}}.\mathbf{v}_{n_k} \right) +(\mathbf{v}_{n_k})_{t\wedge \cdot}.
\end{split}
\end{equation}
By considering what proved above and the assumptions on $\{\mathbf{v}_n\}_{n\in \mathbb{N}}$, 
there exists $\gamma\in \Bone{H}$ such that
\begin{equation}\label{2017-06-10:03}
  \partial _3\psi(t,\mathbf{x}_{n_k},\Lambda^{t,\mathbf{x}_{n_k}}).
 \left(  \partial _{\mathbb{B}^1}\Lambda^{t,\mathbf{x}_{n_k}}.\mathbf{v}_{n_k} \right) +(\mathbf{v}_{n_k})_{t\wedge \cdot}\rightarrow \gamma\ \mbox{in}\ \Bones{H}.
\end{equation}
Then,
\eqref{2016-02-11:43},
\eqref{2017-06-10:03},
 and Proposition \ref{2016-02-11:27}\emph{(\ref{2016-02-11:44})}, we have
$$
\partial _3\psi(t,\mathbf{x}_{n_k},\Lambda^{t,\mathbf{x}_{n_k}}).
 \left(  \partial _{\mathbb{B}^1}\Lambda^{t,\mathbf{x}_{n_k}}.\mathbf{v}_{n_k} \right) \rightarrow 
\partial _3\psi(t,\mathbf{x},\Lambda^{t,\mathbf{x}}).
\gamma\mbox{ in }\Boninf{H},\mbox{ hence in }\Bones{H}.
$$
By taking the
limit in $\Bones{H}$ in \eqref{2016-02-11:43}, we have
$$
\gamma=\partial _3\psi(t,\mathbf{x},\Lambda^{t,\mathbf{x}}).
\gamma+
\mathbf{v}_{t\wedge \cdot},
$$
which entails $\gamma= \partial _{\mathbb{B}^1}\Lambda^{t,\mathbf{x}}.\mathbf{v}$, 
by Proposition \ref{2016-02-11:37}\emph{(\ref{2016-02-11:33})}.
This shows that
$$
\partial _{\mathbb{B}^1}\Lambda^{t,\mathbf{x}_{n_k}}.\mathbf{v}_{n_k}\rightarrow 
\partial _{\mathbb{B}^1}\Lambda^{t,\mathbf{x}}.\mathbf{v}
\ \mbox{ in }\ \Bones{H}.
$$
Since the original sequences $\{\mathbf{x}\}_{n\in \mathbb{N}},\{\mathbf{v}_n\}_{n\in \mathbb{N}}$ were arbitrary,
 \eqref{2016-02-11:36} is proved.

To prove
\eqref{2016-02-11:35},
we use
Proposition \ref{2016-02-11:37}\emph{(\ref{2016-02-11:33})}.
But now most of the work is done.
By
\eqref{2016-02-11:36}, 
we have
\begin{align*}
    \partial _{\mathbb{B}^1}\Lambda^{t,\mathbf{x}_{n}}.\mathbf{v}_{n}&\rightarrow 
\partial _{\mathbb{B}^1}\Lambda^{t,\mathbf{x}}.\mathbf{v}\mbox{ in }\Bones{H}\\
    \partial _{\mathbb{B}^1}\Lambda^{t,\mathbf{x}_{n}}.\mathbf{w}_{n}&\rightarrow 
\partial _{\mathbb{B}^1}\Lambda^{t,\mathbf{x}}.\mathbf{w}\mbox{ in }\Bones{H}.
  \end{align*}
By Proposition \ref{2016-02-11:27}\emph{(\ref{2016-02-11:44})}, we have
  \begin{multline*}
    \lim_{k\rightarrow \infty} \partial ^2_3
\psi(t,\mathbf{x}_{n},\Lambda^{t,\mathbf{x}_{n}}).
 \left( 
 (\partial _{\mathbb{B}^1}\Lambda^{t,\mathbf{x}_{n}}.\mathbf{v}_{n}),
 (\partial _{\mathbb{B}^1}\Lambda^{t,\mathbf{x}_{n}}.\mathbf{w}_{n})
 \right) =\\
= \partial ^2_3\psi(t,\mathbf{x},\Lambda^{t,\mathbf{x}}).
 \left( 
 (\partial _{\mathbb{B}^1}\Lambda^{t,\mathbf{x}}.\mathbf{v}),
 (\partial _{\mathbb{B}^1}\Lambda^{t,\mathbf{x}}.\mathbf{w})
 \right)
\end{multline*}
where the limit is taken in $\Boninf{H}$.
We can now conclude by using the strong continuity claimed in
Proposition \ref{2016-02-11:37}\emph{(\ref{2016-02-11:32})}   and
the formula for the second order derivative
provided by
Proposition \ref{2016-02-11:37}\emph{(\ref{2016-02-11:33})}.

\emph{(\ref{2017-06-03:07})}
Let $t_n\rightarrow t^+$ in $[0,T]$,
$\mathbf{x}\in \mathbb{W}$, $\mathbf{v}\in \Bone{H}$.
By Proposition~\ref{2016-02-11:37}\emph{(\ref{2016-02-11:33})}
and by taking into account  formula \eqref{eq:2017-06-04:01},
we can write
\begin{equation*}
  \begin{split}
     \partial _{\mathbb{B}^1}\Lambda^{t_n,\mathbf{x}}.\mathbf{v}&=
  (I-\partial _3\psi(t,\mathbf{x},\Lambda^{t,\mathbf{x}}))^{-1}
 \left( 
 \partial _2\psi(t_n,\mathbf{x},\Lambda^{t_n,\mathbf{x}}).\mathbf{v}
 \right) \\
&=\sum_{k\in \mathbb{N}}
(\partial _3\psi(t_n,\mathbf{x},\Lambda^{t_n,\mathbf{x}}))^k\mathbf{v}_{t_n\wedge \cdot}
=
\mathbf{v}_{t_n\wedge \cdot}+
\sum_{k\geq 1}
(\partial _3\psi(t_n,\mathbf{x},\Lambda^{t_n,\mathbf{x}}))^k\mathbf{v}_{t_n\wedge \cdot}.
  \end{split}
\end{equation*}
The fact that $t_n\rightarrow t$ from the right assures that $\mathbf{v}_{t_n\wedge\cdot}\rightarrow \mathbf{v}_{t\wedge \cdot}$ in $\Bones{H}$.
Moreover,
by
Proposition~\ref{2016-02-11:37}\emph{(\ref{2016-02-11:24})},
$\Lambda^{t_n,\mathbf{x}}\rightarrow \Lambda^{t,\mathbf{x}}$ in $\mathbb{W}$.
Then, by
Proposition~\ref{2016-02-11:27}\emph{(\ref{2016-02-11:13})},\emph{(\ref{2016-02-11:44})}, and Lebesgue's dominated convergence theorem for sums, we have
$$
\sum_{k\geq 1}
(\partial _3\psi(t_n,\mathbf{x},\Lambda^{t_n,\mathbf{x}}))^k\mathbf{v}_{t_n\wedge \cdot}\rightarrow 
\sum_{k\geq 1}
(\partial _3\psi(t,\mathbf{x},\Lambda^{t,\mathbf{x}}))^k\mathbf{v}_{t\wedge \cdot} \mbox{ in }\Boninf{H}.
$$
Then  
$     \partial _{\mathbb{B}^1}\Lambda^{t_n,\mathbf{x}}.\mathbf{v}\rightarrow 
     \partial _{\mathbb{B}^1}\Lambda^{t,\mathbf{x}}.\mathbf{v}$ in $\Bones{H}$ and \eqref{2017-06-09:02} is proved.

Regarding \eqref{2017-06-09:03}, the argument is similar, by using the expression for
$ \partial ^2_{\mathbb{B}^1}\Lambda^{t_n,\mathbf{x}}.(\mathbf{v},\mathbf{w})$
provided by
 Proposition~\ref{2016-02-11:37}\emph{(\ref{2016-02-11:33})},
the convergence
\eqref{2017-06-09:02} just proved,
and
\eqref{2017-06-09:04}
in Proposition~\ref{2016-02-11:27}\emph{(\ref{2016-02-11:44})}
\end{proof}


We defined $\psi$ for a given, fixed, $\omega\in\Omega$ (p.\ \pageref{2017-06-09:00}).
For every such $\psi$,  Propositions~\ref{2016-02-11:27},
\ref{2016-02-11:37},
\ref{2016-02-12:02} apply.
We can then define the map
\begin{equation}
  \label{eq:2016-02-12:03}
  \Omega\times [0,T]\times \Bone{H}\rightarrow \Bone{H},\ (\omega,t,\mathbf{x}) \mapsto X^{t,\mathbf{x}}(\omega)
\end{equation}
where $X^{t,\mathbf{x}}(\omega)$ is
 the function $\Lambda^{t,\mathbf{x}}$ provided by 
Proposition \ref{2016-02-11:37}, when $\psi$ is associated to $\omega$.
It should be clear that $X^{t,\mathbf{x}}$ is the unique
strong solution to SDE~\eqref{2016-02-11:00} in $\mathcal{L}^0_{\mathcal{P}_T}(\mathbb{W})$.

\medskip

Let $f\colon \Bone{H}\rightarrow \mathbb{R}$ be a function.
Hereafter, we assume that $f$ satisfies the following assumption.

\begin{assumption}\label{2016-02-12:06}
${}$
  \begin{enumerate}[(i)]
  \item  $f\in\mathcal{G}^2(\Boninf{H},\mathbb{R})$;
  \item  the differentials $ \partial f$ and $ \partial ^2f$ are bounded;
  \item\label{2016-02-12:04} $  \Boninf{H} \times \Bones{H}\rightarrow \mathbb{R},\ (\mathbf{x},\mathbf{v}) \mapsto  \partial f(\mathbf{x}).\mathbf{v}$ is sequentially continuous;
  \item\label{2016-02-12:05} $  \Boninf{H} \times \Bones{H}\times \Bones{H}\rightarrow \mathbb{R},\ (\mathbf{x},\mathbf{v},\mathbf{w}) \mapsto  \partial^2 f(\mathbf{x}).(\mathbf{v},\mathbf{w})$ is sequentially continuous.
  \end{enumerate}
\end{assumption}

The following theorem shows that the main results of 
Section~\ref{2017-05-11:07}
 can be applied in the present framework.

\begin{theorem}\label{2016-02-13:02}
  Let $X$ be the unique strong solution of  \eqref{2016-02-11:00}
and let
$$
\varphi\colon [0,T]\times \mathbb{W}\rightarrow \mathbb{R},\ (t,\mathbf{x}) \mapsto \mathbb{E} \left[ f(X^{t,\mathbf{x}}) \right].
$$
Then $\varphi$ verifies 
Assumption~\ref{2017-05-30:11}.
Moreover,
for all $t\in(0,T)$ and
all $\mathbf{x}\in \mathbb{W}$,
\begin{equation}
  \mathcal{D}_t^-\varphi(t,\mathbf{x}_{t\wedge \cdot})+
  \overline{ \partial _\mathbb{W}\varphi}(t,\mathbf{x}).
  (\mathbf{1}_{[t,T]}b(t,\mathbf{x}))
  +
  \frac{1}{2}\mathbf{T} \left[ \overline{\partial ^2_\mathbb{W}
      \varphi}(t,\mathbf{x}),B
  \right]=0
\end{equation}
and for all $t\in[0,T]$, $t'\in[t,T]$, $Y\in \mathcal{L}^p_{\mathcal{P}_T}(\mathbb{W})$, $p>2$,
\begin{equation}
  \label{eq:2017-06-09:01}
  \varphi(t',X^{t,Y})
=\varphi(t, Y)+\int_t^{t'}
\overline{ \partial _\mathbb{W}\varphi}(s,X^{t,Y}).(\mathbf{1}_{[s,T]}b(s,X^{t,Y}))
dW_s\qquad \mathbb{P}\mbox{-a.e..}
\end{equation}
\end{theorem}
\begin{proof}
It is sufficient to show that
 $\varphi$ verifies
Assumption~\ref{2017-05-30:11}\emph{(\ref{2017-05-31:06}),(\ref{2017-05-31:07})},
since the remaining part of the theorem comes from
Theorem~\ref{2016-02-10:10}
and 
Corollary~\ref{2017-06-08:02}.

We begin by verifying 
Assumption~\ref{2017-05-30:11}\emph{(\ref{2017-05-31:06})}.
  By 
Propositin \ref{2016-02-11:37}\emph{(\ref{2016-02-11:25})},
for all $(\omega,t)\in\Omega \times [0,T]$, the map
$\mathbf{x} \mapsto X^{t,\mathbf{x}}(\omega)$ belongs to $\mathcal{G}^2(\Boninf{H},\Boninf{H})$ and has differentials $ \partial _{\mathbb{B}^1}X^{t,\cdot}(\omega)$
and
$ \partial^2 _{\mathbb{B}^1}
X^{t,\cdot}(\omega)$
bounded, with bound uniform in $\omega,t$
(recall Remark~\ref{2016-04-22:06}).
Then, since $f\in \mathcal{G}^2(\Boninf{H},\mathbb{R})$ and
$ \partial f$ and $ \partial ^2f$ are uniformly bounded, 
 the composition $\mathbf{x} \mapsto f(X^{t,\mathbf{x}}(\omega))$ belongs to $\mathcal{G}^2(\Boninf{H},\mathbb{R})$ and has differentials $ \partial _{\mathbb{B}^1}f(X^{t,\cdot}(\omega))$
and
$ \partial^2 _{\mathbb{B}^1}f(X^{t,\cdot}(\omega))$
bounded, with bound uniform in $\omega,t$.
We have
\begin{equation}\label{2016-02-12:07}
    \partial _{\mathbb{B}^1}f(X^{t,\mathbf{x}}(\omega)).\mathbf{v}= \partial f(X^{t,\mathbf{x}}(\omega)).( \partial _
{\mathbb{B}^1}X^{t,\mathbf{x}}(\omega).\mathbf{v})
  \end{equation}
  for all $t\in[0,T]$, $\omega\in\Omega$, $ \mathbf{x}, \mathbf{v}\in \Bone{H}$, and
  \begin{align}
     \partial^2 _{\mathbb{B}^1
}&f(X^{t,\mathbf{x}}(\omega))
.(\mathbf{v},\mathbf{w})
=\label{2016-02-12:08}\\
&=\partial^2 f(X^{t,\mathbf{x}}(\omega)).( 
(\partial _{\mathbb{B}^1}X^{t,\mathbf{x}}(\omega).\mathbf{v}).
(\partial _{\mathbb{B}^1}X^{t,\mathbf{x}}(\omega).\mathbf{w})
)
+
\partial f(X^{t,\mathbf{x}}(\omega)).( \partial^2 _{\mathbb{B}^1}X^{t,\mathbf{x}}(\omega).(\mathbf{v},\mathbf{w}))\notag
\end{align}
  for all $t\in[0,T]$, $\omega\in\Omega$, $\mathbf{x},\mathbf{v},\mathbf{w}\in \Bone{H}$.
Since 
$ \partial _{\mathbb{B}^1}f(X^{t,\mathbf{x}}(\omega))$
and
$ \partial^2 _{\mathbb{B}^1}f(X^{t,\mathbf{x}}(\omega))$
are bounded, with bound uniform in $\omega,t,\mathbf{x}$,
we can easily see that
\begin{align}
 \partial _{\mathbb{B}^1}\varphi(t,\mathbf{x}).\mathbf{v} &=
\mathbb{E} \left[  \partial _{\mathbb{B}^1}
f(X^{t,\mathbf{x}}) .\mathbf{v}\right]\label{2016-02-12:09} \\[2pt]
 \partial ^2_{\mathbb{B}^1}\varphi(t,\mathbf{x}).(\mathbf{v},\mathbf{w})&=
\mathbb{E} \left[  \partial^2 _{\mathbb{B}^1}f(X^{t,\mathbf{x}}) .(\mathbf{v},\mathbf{w})\right],\label{2016-02-12:10}
\end{align}
for all $t\in[0,T]$, $\mathbf{x},\mathbf{v},\mathbf{w}\in \Bone{H}$.
Finally, by
\eqref{2016-02-12:09},
\eqref{2016-02-12:10}, boundedness of $ \partial _{\mathbb{B}^1}f(X^{t,\cdot})(\omega)$
and
$ \partial ^2_{\mathbb{B}^1}f(X^{t,\cdot}(\omega))$,
strong continuity of 
$ \partial _{\mathbb{B}^1}f(X^{t,\cdot})(\omega)$
and $ \partial ^2_{\mathbb{B}^1}f(X^{t,\cdot}(\omega))$,
we  obtain that 
 $\varphi(t,\cdot)$
belongs to $ \Gatot{\Boninf{H}}{\mathbb{R}}{2}$ and has bounded first and second order differentials.
To conclude the verification of
Assumption~\ref{2017-05-30:11}\emph{(\ref{2017-05-31:06})}, it is sufficient to show that, for all $t\in[0,T]$,
the maps
\begin{gather*}
  \mathbb{W}\times \Bones{H}\rightarrow \mathbb{R},\ 
(\mathbf{x},\mathbf{v}) \mapsto 
 \partial _{\mathbb{B}^1}\varphi(t,\mathbf{x}).\mathbf{v}\\[2pt]
\mathbb{W}
\times \Bones{H}\times \Bones{H}\rightarrow \mathbb{R},\ 
(\mathbf{x},\mathbf{v},\mathbf{w}) \mapsto 
 \partial^2 _{\mathbb{B}^1}\varphi(t,\mathbf{x}).(\mathbf{v},\mathbf{w})
\end{gather*}
are sequentially continuous.
This comes immediately by combining
\eqref{2016-02-12:07},
\eqref{2016-02-12:08},
\eqref{2016-02-12:09},
\eqref{2016-02-12:10},
Proposition \ref{2016-02-11:37}\emph{(\ref{2016-02-11:23})},
Proposition \ref{2016-02-12:02}\emph{(\ref{2017-06-03:08})},
Assumption \ref{2016-02-12:06}\emph{(\ref{2016-02-12:04})},\emph{(\ref{2016-02-12:05})},
 the uniform boundedness of the differentials involved
(we recall again Remark~\ref{2016-04-22:06})
 and of the convergent sequences in $\Bones{H}$.

Similarly, we can see that
Assumption~\ref{2017-05-30:11}\emph{(\ref{2017-05-31:07})} is verified
by taking into account
\eqref{2016-02-12:07},
\eqref{2016-02-12:08},
\eqref{2016-02-12:09},
\eqref{2016-02-12:10},
Proposition \ref{2016-02-11:37}\emph{(\ref{2016-02-11:24})},
Proposition \ref{2016-02-12:02}\emph{(\ref{2017-06-03:07})},
Assumption \ref{2016-02-12:06}\emph{(\ref{2016-02-12:04})},\emph{(\ref{2016-02-12:05})},
 the uniform boundedness of the differentials involved
 and of the convergent sequences in $\Bones{H}$.
\end{proof}

\begin{remark}
  Let $g\colon [0,T]\times H\rightarrow H$ be a continuous function, with $g(t,\cdot)\in C^2_b(H,H)$ and with differentials $D_Hg, D^2_Hg$ uniformly continuous.
Let us introduce the function $\hat b^g$ defined by
$$
\hat b^g(t,\mathbf{y})\coloneqq b
 \left( t,\int_{[0,T]} \tilde g(t-s,\mathbf{\tilde y}(t-s))\mu(ds)
 \right) \qquad \forall (t,\mathbf{y})\in [0,T]\times \Bone{H},
$$
where $\tilde g(r,x)\coloneqq g(0,x)$ if $r<0$.
Consider the function
$$
G\colon \Bone{H}\rightarrow \Bone{H},\ \mathbf{y} \mapsto 
\{g(t,\mathbf{y}(t))\}_{t\in[0,T]}.
$$
Then $G$ is well-defined, $G$ belongs to $C^2_b(\Boninf{H},\Boninf{H})$,
and $\hat b^g(t,\mathbf{y})=\hat b(t,G(\mathbf{y}))$.
By using these observations and
 the explicit expressions of $DG,D^2G$ in terms of $D_Hg$, $D^2_Hg$,
it is not difficult to 
 show that the results
 proved in this section can be extended to the case in which the drift $\hat b$ 
in 
SDE~\eqref{2016-02-11:00}
is replaced by the more general drift $\hat b^g$.
\end{remark}


\addcontentsline{toc}{chapter}{References}
\bibliographystyle{plain}
\bibliography{rosestolato_2018-06-20_funct-ito.bbl}

\end{document}